\documentclass{amsart}
\usepackage{amsmath,amscd,amsfonts,amssymb,amsthm, multirow, newlfont,exscale,color,mathrsfs, multicol}

\title{Strength conditions, small subalgebras, and Stillman bounds in degree $\leq 4$}

\author{Tigran Ananyan and Melvin Hochster$^1$}
\date{\today}
\usepackage{enumerate}
\usepackage{bm}
\theoremstyle{plain}
\newtheorem{theorem}{Theorem}[section]
\newtheorem{corollary}[theorem]{Corollary}
\newtheorem{proposition}[theorem]{Proposition}
\newtheorem{lemma}[theorem]{Lemma}
\newtheorem{conjecture}[theorem]{Conjecture}
\theoremstyle{remark}
\newtheorem{remark}[theorem]{Remark}

\newtheorem{discussion}[theorem]{Discussion}
\theoremstyle{definition}
\newtheorem{definition}[theorem]{Definition}
\newtheorem{example}[theorem]{Example}
\newtheorem{examples}[theorem]{Examples}
\newtheorem{question}[theorem]{Question}

\newcommand{\A}{\mathbb{A}}

\newcommand{\al}{\alpha}

\newcommand{\by}{\times}
\newcommand{\dbimp}{\Longleftrightarrow}
\newcommand{\Der}{\hbox{\rm Der}}
\newcommand{\di}{\hbox{dim}}
\newcommand{\inc}{\subseteq}
\newcommand{\ch}{\mathrm{char}}

\newcommand{\cB}{\mathcal{B}}
\newcommand{\cC}{\mathcal{C}}
\newcommand{\cD}{\mathcal{D}}
\newcommand{\cE}{\mathcal{E}}
\newcommand{\cF}{\mathcal{F}}
\newcommand{\cL}{\mathcal{L}}

\newcommand{\cS}{\mathcal{S}}
\newcommand{\cT}{\mathcal{T}}
\newcommand{\cU}{\mathcal{U}}
\newcommand{\cV}{\mathcal{V}}
\newcommand{\cW}{\mathcal{W}}
\newcommand\cs[1]{$#1$-safe}

\newcommand{\inj}{\hookrightarrow}

\newcommand{\cK}{\mathcal{K}}
\newcommand{\cM}{\mathcal{M}}
\newcommand{\cN}{\mathcal{N}}
\newcommand{\cZ}{\mathcal{Z}}
\newcommand{\etA}{{}^\eta\!A}
\newcommand{\etB}{{}^\eta\!B}

\newcommand{\etuA}{{}^\eta\!{\underline{A}}} 
\newcommand{\fA}{\mathfrak{A}}

\newcommand{\fC}{\mathfrak{C}}
\newcommand{\fd}{\mathfrak{d}}
\newcommand{\fD}{\mathfrak{D}}

\newcommand{\fH}{\mathfrak{H}}

\newcommand{\fK}{\mathfrak{K}}
\newcommand{\fJ}{\mathfrak{J}}
\newcommand{\fm}{\mathfrak{m}}

\newcommand{\fP}{\mathfrak{P}}
\newcommand{\fQ}{\mathfrak{Q}}

\newcommand{\imp}{\Rightarrow}

\newcommand{\LCK}{LC$_K$}
\newcommand{\Ijac}{I_2\bigl(J(F,G)\bigr)}

\newcommand{\N}{\mathbb{N}}

\newcommand{\pd}{\textrm{pd}}

\newcommand{\PP}{\mathbb{P}}
\newcommand{\ov}{\overline}

\newcommand{\Rad}{\textrm{Rad}}
\newcommand{\rkj}{\textrm{J-rank}}
\newcommand\rom[1]{\text{\textnormal{#1}}}  

\newcommand{\surj}{\twoheadrightarrow}
\newcommand{\tD}{\widetilde{\mathcal{D}}}
\newcommand{\tr}{^{\mathrm{tr}}}   

\newcommand{\uA}{\underline{A}}
\newcommand{\ux}{\underline{x}}

\newcommand{\Z}{\mathbb{Z}}

\newcommand{\GL}{\hbox{\rm GL}}

\newcommand\ud{\underline{d}}
\newcommand\vect[2]{#1_1,\,\ldots,\, #1_{#2}}

\newcommand\wt[1]{\widetilde{#1}}
\newcommand{\gu}{strong}

\def\vect#1#2{{#1}_1, \, \ldots, \, {#1}_{#2}}
\newcommand{\lc}{\langle}
\newcommand{\rc}{\rangle}
\newcommand\sq[1]{\lc #1 \rc^2}
\newcommand{\mx}{\begin{pmatrix}}
\newcommand{\emx}{\end{pmatrix}}
\newcommand{\und}{\underline}
\def\todo#1
{\par\vskip 7 pt \noindent {\textcolor{yellow}\footnotesize \color{yellow} \fbox{\parbox{4.8in}{\color{black}
\textbf{To do: } {\color{blue}#1}}}} \vskip 3 pt \par}
\def\forth#1
{\par\vskip 7 pt \noindent {\textcolor{yellow}\footnotesize \color{yellow} \fbox{\parbox{4.8in}{\color{black}
\textbf{For thought: } {\color{blue}#1}}}} \vskip 3 pt \par}

\begin{document}

\begin{abstract} In \cite{AH2}, the authors prove Stillman's conjecture in all characteristics and all degrees
by showing that, independent of the algebraically closed field $K$ or the number of variables, $n$ forms of
degree at most $d$ in a polynomial ring $R$ over $K$ are contained in a polynomial subalgebra of $R$ generated 
by a regular sequence consisting of at most $\etB(n,d)$ forms of  degree at most $d$: we refer to these informally
as ``small" subalgebras.  Moreover,  these forms can be 
chosen so that the ideal generated by any subset defines a ring satisfying the Serre condition R$_\eta$. A critical 
element in the proof is to show that there are functions $\etA(n,d)$ with the following property:  in a graded
$n$-dimensional $K$-vector subspace $V$ of  $R$ spanned by forms of degree at most $d$, if no nonzero form 
in $V$ is in an ideal generated
by $\etA(n,d)$ forms of strictly lower degree (we call this a {\it strength} condition), then any homogeneous
basis for $V$ is an R$_\eta$ sequence.   The methods of \cite{AH2} are not constructive.  In this paper, we
use related but different ideas that emphasize the notion of a {\it key function} to obtain the functions $\etA(n,d)$ in degrees 2, 3, and 4 (in degree 4 we must restrict to characteristic not 2, 3).  We give bounds in closed form for the 
key functions and the $\etuA$ functions, and  explicit recursions that determine the functions $\etB$ from the $\etuA$ functions.  In degree 2, we obtain an explicit value for $\etB(n,2)$
that gives the best known bound in Stillman's conjecture for quadrics when there is no restriction on $n$. 
In particular, for an ideal  $I$ generated by $n$ quadrics, the projective dimension $R/I$ is at most $2^{n+1}(n - 2) + 4$. 
\end{abstract}

\subjclass[2000]{Primary 13D05, 13F20}

\keywords{polynomial ring, ideal, quartic form, cubic form,  quadratic form, projective dimension, regular sequence}

\thanks{$^1$ The second author was partially supported by grants from the National Science Foundation (DMS--0901145 and
DMS--1401384).}

\thanks{This is a corrected version of this paper:  we thank Arthur Bik and Andrew Snowden for pointing out that Proposition (3.3)
in an earlier version had a conclusion that was too strong.}

\maketitle

\pagestyle{myheadings}
\markboth{TIGRAN ANANYAN AND MELVIN HOCHSTER}{STILLMAN BOUNDS IN DEGREE $\leq 4$}

\section{Introduction}\label{intro}   

Throughout this paper, let $R$ denote a polynomial ring over an arbitrary field $K$:  say
$R = K[x_1, \ldots , x_N]$. We will denote the projective dimension of the
$R$-module  $M$ over $R$ by $\textrm{pd}_R(M)$. The following theorem was conjectured
by M.~Stillman in \cite{PS} and proved, in a strengthened form (for submodules with a specified
number of generators of free modules with a specified number of generators, and without the assumption of 
homogeneity) in Theorem C of \cite{AH2}.

\begin{theorem}\label{mainconj}
There is an upper bound, independent of $N$, on $\pd_R(R/I)$, where $I$ is any ideal of $R$
generated by $n$ homogeneous polynomials $\vect F n$ of given degrees $d_1,\,\ldots,\, d_n$.
\end{theorem}

We refer to such bounds, which are now known to exist, as {\it Stillman bounds}. 

Descriptions of earlier work related to this problem \cite{BeS, Br, Bu, CK, E1, E2, E3, Ko, Mc} are given 
in the introductions of \cite{AH1, AH2}  and in \cite{McS}. \cite{Dra1, Dra2, ESS1, ESS2, ESS3, ESS4} contain
recent work, some of which utilizes the results
of \cite{AH2}, on Stillman's conjecture and related questions, including Noetherianity problems.

In a recent paper, Mantero and McCullough \cite{MaMc} have proved that in the case of three cubics, 5 is a sharp
bound for the projective dimension. 

For $n$ quadratic forms generating an ideal $I$ of height $h$, there are examples where
$\pd_R(R/I)$ is $h(n-h+1)$: see \cite{Mc}. 
  The projective dimension cannot be larger when $h=2$  (see \cite{HMMS1}) 
 but when $h \geq 3$ and $n \geq 5$ it is an open question whether $h(n-h+1)$ is a bound.  When $n=4$ and
$h=3$ it has recently been shown that the largest possible projective dimension is 6: see \cite{HMMS2}.

The proof of the existence of Stillman bounds in \cite{AH2} depends  critically on proving auxiliary bounds
$\etuA$, $\etB$.  Our focus here is on giving explicit bounds for $\etA$ in degrees 2, 3, 4, and on $\etB$ when 
the degree is 2.  Our methods in degree 4 are vastly different from those in \cite{AH2}.  

Because Stillman's conjecture is unaffected by a base change on the field, we shall assume from now
on that the base field $K$ is algebraically closed, which is needed for many of our theorems about
the functions $\etuA$, $\etB$. Moreover, at many points there is a restriction that the characteristic of $K$ is either 
0 or larger  than some given integer.  However, all assumptions of this type will be made 
specific.  

Our main results are stated below, but we must first recall some definitions from \cite{AH2}.  
We write $\N$ for the nonnegative integers and $\N_+$ for the positive integers. 
A {\it graded} $K$-algebra $R$ will always be a finitely generated $K$-algebra
graded by $\N$ such that $R_0 = K$. We write $R_d$ for the graded component of degree $d$,
typically thought of as a finite dimensional vector space over $K$.  Unless otherwise specified, polynomial rings always 
have the standard grading in which all variables have degree one. 

Given a finite-dimensional $\N_+$-graded $K$-vector space $V$ over a field $K$ such that $\di_KV_t = n_t$
for $t \geq 1$,  we refer to $\delta = \delta(V) = n_1,\, n_2,\, n_3, \, \dots, \, n_t, \, \ldots$ as the 
{\it dimension sequence} of $V$.  It should cause no confusion when we also write
$\delta(V) = (\vect n d)$ to mean that $\vect n d$  constitute the first $d$ terms of the dimension
sequence of $V$,  and that the other terms in the sequence are  0.  

For any finite-dimensional vector space $V$ over an algebraically closed field $K$ we denote by $\PP(V)$
the projective space whose points correspond to the lines through the origin in $V$.  If $\di(V) = d$,
$\PP(V) \cong \PP^{d-1}_K$.  

\begin{discussion}\label{strength}
 Given an $\N$-graded $K$-algebra and $k \in \N$ we shall say that a form $F$ has a $k$-{\it collapse}
if it is a graded linear combination of $k$  or fewer forms of strictly  smaller positive degree.  Note that only the 0 element
has a $0$-collapse.  Nonzero scalars and nonzero linear forms cannot have a $k$-collapse for any $k$.  With this
terminology, a form is $k$-{\it strong}  
if and only if it has no $k$-collapse. 
Given a graded $K$-algebra $S$,  a finite-dimensional $\N_+$-graded $K$-vector subspace $V \inc S$ with dimension sequence $(\vect nd)$, and a $d$-tuple of non-negative integers $\kappa = (\vect kd)$,
we call $V$ {\it $\kappa$-\gu}  if there is no nonzero form in $V_i$, the graded component of $V$ in degree $i$, 
that has a $k_i$-collapse.   We shall say that a sequence
of forms of positive degree is $\kappa$-\gu\ if the forms are linearly independent over $K$ and the graded $K$-vector space
 they span is $\kappa$-\gu.  This means that if $F_{j_1}, \, \ldots, F_{j_s}$ are elements of the sequence of the same 
 degree $i$ with mutually
 distinct indices and $\vect c s \in K$ are not all zero, then  $\sum_{t=1}^s c_t F_{j_t}$ is nonzero and has no
 $k_i$-collapse.   Hence, a graded $K$-vector space $V$ is $\kappa$-\gu\ if and only if every  sequence of independent
 forms in $V$ is $\kappa$-\gu.     If all entries of $\kappa$ are equal to the same integer $k$, we may use the term $k$-\gu\ instead of $\kappa$-\gu.
 
 We call a prime ideal of a polynomial ring  $k$-{\it linear} if it is generated by at most $k$ linear forms. 
 In a polynomial ring over a field,
a form $F$ of degree two or three has a $k$-collapse if and only if it is contained in a $k$-linear prime ideal.
This is not true for forms of degree 4 or higher:  for example, a degree 4 form may have a collapse
in which some of the summands are the product of two quadrics. 
 
 We define the {\it strength} of a nonzero form $F$ of positive degree as the largest integer $k$ such that $F$ is 
 $k$-strong.  If $F$ is a nonzero 1-form, we make the convention that its strength is $+\infty$.  We define
 the strength of a nonzero vector space $V$ consisting of forms of the same degree as the smallest strength of a nonzero
 element of $V$.  For nonzero vector spaces $V$ of linear forms, the strength of $V$ is $+\infty$. Thus $F$ (respectively, $V$)  is $k$-strong if and only if its strength is at least $k$.

  We note that notions closely related to strength (q-rank, slice rank) have been considered independently in \cite{DES}
  and in \cite{BCC}, which was inspired by \cite{Tao}. E.g., the notion of q-rank for cubic forms utilized
  in \cite{DES} is the same as the strength of the cubic form minus one.
\end{discussion}
 
 \begin{discussion}\label{RC}
Recall that a Noetherian ring $R$ satisfies the Serre condition R$_i$ if $R_P$ is regular for every prime
$P$ of height $\leq i$.  If the singular (i.e., non-regular) locus in $R$ is closed with defining ideal
$J$, this means that $J$ has height at least $i+1$.  (If $R$ is regular, $J = R$ and has height $+\infty$.) 
In the sequel, the rings that we are studying are
standard graded algebras over a field $K$:  they are finitely generated $\N$-graded rings $R$ over
$K$ such that $R_0 = K$ and $R_1$ generates $R$ as a $K$-algebra.  If such a ring is normal, it
is a domain, since it is contained in its localization at the homogeneous maximal ideal.  We will know 
inductively that the rings we are studying are complete intersections.  As just noted,  R$_1$ implies 
normal domain.  \end{discussion}

By a theorem of A.~Grothendieck, if a homogeneous complete intersection in a polynomial ring is R$_3$ then it is a UFD.  
This is Corollary~\ref{groth} below.

\begin{theorem}\label{samconj}  Let $R$ be a local ring that is a  complete intersection whose localization at any prime
ideal of height 3 or less is a UFD.  Then $R$ is a UFD.  Hence, if $R$ is a local ring that is a complete intersection
in which the singular locus has codimension 4 or more, then $R$ is a UFD. \end{theorem}

We need this in a graded version (cf. \cite{Sam}, Proposition 7.4):

\begin{corollary}\label{groth} Let $S$ be an $\N$-graded algebra finitely generated over a field $K$ that is a complete
intersection (the quotient of a polynomial ring over $K$ by the ideal generated by a regular sequence of
forms of positive degree).  Suppose that the singular locus in $S$ has codimension at least 4.  Then
$S$ is a UFD. \end{corollary}    

 Theorem~\ref{samconj} was conjectured by P.~Samuel in the local case and was proved by A.~Grothendieck. 
An algebraic proof is given in \cite{Ca}. \\
 
The following result is proved in \cite{AH2}.  Our main goal here is to give specific bounds for the functions
$A$ and $\etA$ in degrees 2, 3, and 4.  

\begin{theorem}\label{etA} There is a function $\uA$ (respectively, for every $\eta \geq 1$ there is a function $\etuA$) 
from dimension sequences  $\delta = (\vect n d)$ to $\N^d$ with the following property. 
If $V$ is a finite-dimensional $\N_+$-graded subspace of a polynomial ring $R$ over an algebraically
closed field $K$  
with dimension sequence $\delta$  that is $\uA(\delta)$-\gu (respectively, $\etuA(\delta)$-\gu)
then if $\vect F h$ are $K$-linearly independent forms in $V$, they form a regular sequence (respectively,
a regular sequence such that $R/(\vect F h)$ satisfies the Serre condition \rom{R}$_\eta$).  If $\eta = 1$,  
then $R/(\vect F h)$ is a normal domain and if $\eta=3$ then $R/(\vect F h)$ is a UFD. 
  \end{theorem}
  
\begin{remark}\label{Anot} We write $A_i(\delta)$ (respectively, $\etA_i(\delta)\,$) for the $i\,$th entry
of $\uA(\delta)$ (respectively, $\etuA(\delta)$).   That is, when the underscored notation is used, the
value of the function is a vector of integers.  Note also that the functions $B$, $\etB$ introduced below always
take values consisting of a single integer, and so are never underscored.  

In case the vector space $V$ has only one graded component, whose degree is $d$,  so that
$\delta = (0, \, \ldots, \, 0, \, n)$,  we may write $A_d(n)$ (respectively, $\etA_d(n)\,$) instead of
$A_d(\delta)$ (respectively, $\etA_d(\delta)\,$).  \end{remark}

\begin{remark}\label{stregh}  We may take $\uA$ to be any $\etuA$ for $\eta \geq 1$.  However, it is frequently
the case that a smaller choice will yield regular sequences.  For example, consider an $n$-dimensional
vector space of forms of the same degree $d$. If one has a level of strength $s \geq 1$ that guarantees that any $n-2$
linearly independent elements form a regular sequence such that the quotient by the ideal they generate
is a UFD, then strength at least $s$ guarantees that any $n$ independent forms in the vector space give a regular 
sequence for the following reason.   The first $n-2$ define a UFD by assumption, the next therefore gives a domain by 
Proposition~\ref{obv} (g), and, consequently, the last is a nonzerodivisor. The level of strength ${}^3\!A_d(n-2)$ will always 
suffice for this argument, by Corollary~\ref{groth}.  
 \end{remark}

Theorem~\ref{etA} above implies the existence of small subalgebras (the precise statement is given below):
this is Theorem B in \cite{AH2}.  We give a brief discussion of the
argument in \S\ref{BfromA}.  It should be noted that the proof of this result gives an explicit recursive formula for
obtaining the functions $\etB$ from the functions $\etuA$.  

We call a function of several integer variables {\it ascending} if it is nondecreasing as a function of each variable
when the other variables are held fixed.

We say that a sequence of elements $\vect G s$ in a Noetherian ring $R$ is an $\rom{R}_\eta$-sequence
if it is a regular sequence and for $0 \leq i \leq s$,  $R/(\vect G i)$ satisfies the Serre condition $\rom{R}_\eta$.

\begin{theorem}\label{small}  
There is an  ascending function $B$ from dimension sequences $\delta = (\vect n d)$ to $\N_+$
with the following property. 
If $K$ is an algebraically closed field and
$V$ is a finite-dimensional $\N_+$-graded $K$-vector subspace of a polynomial ring $R$ over $K$  with  
dimension sequence $\delta$,  then $K[V]$ is contained in
a $K$-subalgebra of $R$ generated by a regular sequence $\vect G s$ of forms of degree at most $d$, where
$s \leq B(\delta)$. 

Moreover, for every $\eta \geq 1$ there is a function $\etB$ with the same property as above such that, in addition, 
every sequence consisting of linearly independent homogeneous linear combinations of the elements in  
$\vect G s$ is an $\rom{R}_\eta$-sequence.
\end{theorem}

In \S\ref{BfromA}, we explain how Theorem~\ref{small} follows from Theorem~\ref{etA}, and also how Theorem~\ref{small}
yields bounds on projective dimension that prove Stillman's conjecture.

Our main results here are as follows.
\newpage
\begin{theorem} 
Let $V$ be a vector space of quadratic forms in $R$ of dimension $n$ over $K$.  
If every element of $V-\{0\}$ is $(n-1)$-strong, every sequence of linearly independent elements of 
$V$ is a regular sequence.  
If $\eta \geq 1$ and every element of  $V-\{0\}$ is $(n-1 +\lceil \frac{\eta}{2}\rceil)$-strong,  
then the quotient by the ideal generated 
by any subset of $V$ satisfies the Serre condition \rom{R}$_\eta$.  Hence, every sequence of linearly independent elements of $V$ is an $\rom{R}_\eta$-sequence. \smallskip
 \end{theorem}  
 
 Note that this is equivalent saying that one may take  $A_2(n) = n-1$ and $\etA_2(n) = (n-1 +\lceil \frac{\eta}{2}\rceil)$.
 See Theorem~\ref{nforms}, Theorem~\ref{ngu}, and Corollary~\ref{A2}.
 
\begin{theorem} A vector space of quadrics in the polynomials ring $R$ that has dimension $n$ is contained in a polynomial subring
generated by a regular sequence consisting of at most $2^{n+1}(n-2)+4$ linear and quadratic forms.  Hence,
the projective dimension of $R/I$,  where $I$ is the ideal generated by these forms, is at most
$2^{n+1}(n-2)+4$.  \end{theorem}

See Theorem~\ref{B2}.
\begin{theorem}\label{A3} 
Let $b = 2(n_2+n_3) + \eta +1$ if $n_2 \not=0$,  and $2(n_2+n_3) + \eta$ if $n_2 = 0$. 
Let $\fJ(b) = (2b+1)(b-1)$ if the characteristic is not 2 or 3,  $\fJ(b) = 2(2b+1)(b-1)$ if
$\ch(K)$ is 2,  and $\fJ(b) = 2b^2-b$ if $\ch(K)$ is 3.   Then we may take
 $$\etuA(n_1,n_2,n_3) = \Bigl(0, \lceil {b\over 2}\rceil + n_1, \fJ(b) +n_1\Bigr).$$  
\end{theorem}

See Theorem~\ref{A3b}. 

Our construction of the functions $\uA, \etuA$ in degrees 3 and 4 will depend on proving
the existence of certain key functions $\fK_i(k)$.  

\begin{definition}\label{key}   If $F$ is a form in a polynomial ring $K[\vect X N]$ over
$K$,  we write $\cD F$ for the $K$-vector space spanned by the partial derivatives of $F$. 
One may also think of $\cD F$ as the image of $F$ under all $K$-derivations
from  $R$ to $R$ that are homogeneous of degree $-1$,  and so it is unaffected by $K$-linear changes 
of coordinates in $R$.

 We say $\fK_i$ is a {\it key} function for degree $i \geq 3$ and for a specified set $\cS$ of characteristics 
if, independent of the algebraically closed field with a characteristic in $\cS$ and the number of variables $N$, for any form $F$ of degree $i$, if $F$ is $\fK_i(k)$-strong,  then the elements in a Zariski dense open subset of $\cD F$ are 
$k$-strong (equivalently, one element in $\cD F$ is $k$-strong).   \end{definition}

Our primary goal in the sequel is to construct key functions for degree 3 in all characteristics and for degree 4 in
characteristic not 2 or 3. 

We shall not need the notion of key function for degree 2 or smaller in our proofs of the existence of $\etuA$
or $\etB$.  One can make the same definition if $i =1$ or $i=2$. The restriction one needs on the one-form when
$i=1$ is that it not be 0:  this implies that some partial has infinite strength. Thus, we may take $\fK_1(k) = 0$ for all $k$. 
If $i=2$, the condition needed on the quadratic form if the characteristic is not 2 is that it not be zero, for then some partial derivative does not vanish 
and it has infinite strength.  If the characteristic is 2,  we need that the form not be a square, which follows if it has 
strength 1.  Thus, we may take $\fK_2(k) = 0$ for all $k$ if the characteristic is not 2,  and $\fK_2(k) = 1$ for
all $k$ if the characteristic is 2. 

The following result plays an important role in our treatment of $\fK_3$ in characteristic not 2 or 3.

\begin{theorem}\label{2trI} Let $K$ be an algebraically closed field of characteristic $\not=2$.  Let $V$ be a vector
space of quadratic forms that is not contained in an ideal generated by $2k$ or fewer linear forms. Then there 
exists a Zariski dense open subset of $V$
consisting of quadratic forms with no $k$-collapse. \end{theorem}

See Theorem~\ref{2tr}. 

\begin{theorem}\label{K2I} For any algebraically closed field $K$ of characteristic $\not=2,3$,
we may take $\fK_3(k) = 2k$. \end{theorem}

Theorem~\ref{K2I} is immediate from Theorem~\ref{K2} below in the case $b=1$.  

We also consider functions $\fJ_d$, which we call {\it J-rank functions}, such that if a form $F$ of degree  
$d$ has strength at least $\fJ_d(k)$,  then  the height of the ideal
generated by $F$ and $\cD F$ is at least $k$.  This is the same as requiring that $V(F)$ be R$_{k-2}$ for
$k \geq 2$.  The functions $\fJ_d$ coincide with the function $\etA$ for
a vector space spanned by one form of degree $d$, with a shift in the index.  See Discussion~\ref{fRet} and 
Definition~\ref{rK} below.  J-rank functions can be used to construct the functions $\etA$ in general:  see 
Theorem~\ref{SJrank}.  When the characteristic is greater than $d$, one can use key functions to construct
J-rank functions, and this is the basis for our theorems on the existence of $\etA$ in degrees 3 and 4
when the characteristic is not $2$ or $3$.  See Theorem~\ref{GtoS}  and Corollary~\ref{KtoR}.

For degree 3, if the algebraically closed field has characteristic 2 or 3, we instead use a direct construction to
obtain our main result on the existence of $\fJ_3(k)$:

\begin{theorem}\label{R323I} Let $R$ be a polynomial ring over an algebraically closed field $K$.
If $K$ has characteristic 2, then we may take $\fJ_3(k) = 2(k-1)(2k+1)$.  
If $K$ has characteristic 3, we may take $\fJ_3(k) =  2k^2-k$.
\end{theorem}

See Theorem~\ref{R323}.  Note that these expressions for $\fJ_3$ are not much worse than
the result one can obtain for characteristic not $2,\,3$ using Theorem~\ref{K2I}, namely
$\fJ_3(k) = (2k+1)(k-1)$. See also Theorem~\ref{G3}.

The following is one of our main results here:  it shows that one can construct the functions
$\etA$ explicitly (and, hence, the functions $\etB$) up through degree $d$ whenever one can construct the functions 
$\fK_i$ for $i \leq d$.

\begin{theorem}\label{SJ} Let $K$ be algebraically closed field of characteristic 0 or $>d$, where $d \geq 3$, 
and let $R= K[\vect xN]$. Suppose that for $2 \leq i \leq d-1$, we have a function
 $A_i:\N_+ \to \N_+$ such that if a vector space of  forms of degree $i$ of dimension $n$ is $A_i(n)$-strong,
 then every sequence of linearly independent forms of $V$ is a regular sequence.  Suppose also 
that for every $i$, $2 \leq i \leq d$,  we have a  key function $\fK_i$ that is also nondecreasing.
Let $\delta = (\vect n d)$ be a dimension sequence.  Let $h$ be the number of nonzero elements 
among $n_2, \, \ldots, \, n_d$, and let $n' = n_2 + \cdots + n_d$. Let $b := h-1 + 2n' + \eta$.
 Then we may define $\etuA(\delta)= \bigl(\etA_1(\delta),\, \ldots, \etA_d(\delta)\bigr)$ as follows.
\begin{enumerate}
\item $\etA_1(\delta) := 0$. \smallskip

\item $\etA_2(\delta) := \lceil \frac{b}{2}  \rceil + n_1$.  \smallskip

\item For $3 \leq i \leq d$,  $\etA_i(\delta) := \fK_i\bigl(bA_{i-1}(b)\bigr) + b-1 + n_1$. 
 \end{enumerate}
Then whenever a vector space $V$ of forms in $R$ with dimension sequence $\delta$ is $\etuA(\delta)$-strong,
every sequence of linearly independent forms in $V$ is an \rom{R}$_\eta$-sequence.
\end{theorem}

Theorem~\ref{SJ} is proved in \S\ref{keyf}, following the proof of Corollary~\ref{KtoR}, 
after a substantial number of preliminary results are established.

\begin{corollary}\label{explA} With hypothesis as in Theorem~\ref{SJ}, if $\delta = (0,\, \ldots, 0,\, n)$,
corresponding to an $n$-dimensional vector space whose nonzero forms all have degree $d \geq 3$, then 
we may take $\etA_d(\delta) = \fK_d\bigl((2n+\eta)A_{d-1}(2n+\eta)\bigr) + 2n+\eta-1$. 
\end{corollary}
\begin{proof}  This is immediate from Theorem~\ref{SJ}, since in this case $h=1$, $n' =n$, and $b = 2n+\eta$. \end{proof}

Our main result in degree 4 is to give an explicit formula for $\fK_4$:  coupled with Theorem~\ref{SJ}, one obtains
the functions $A,\,\etA$ for degree up to 4.  However, for this result we must exclude characteristics 2 and 3.
See Theorem~\ref{main4G} and Corollary~\ref{explA4}.

\begin{theorem}\label{main4GI}  Let $K$ be an algebraically closed field of characteristic not $2$ or $3$. Then 
$\fK_4(k) = 6k(k+1)4^{k(k+1)} + (k+1)^2$ 
\end{theorem}

\section{Deriving Stillman's Conjecture and the functions $\etB$ from the functions $\etA$}\label{BfromA}  

We describe first the algorithm for obtaining $\etB$ for $\etA$ (and $B$ from $A$, which is the same argument)
from \cite{AH2}.

One linearly orders 
the dimension sequences $\delta = (\vect \delta d)$ so that  $\delta < \delta'$ precisely if $\delta_i < \delta_i'$ for
the largest value of $i$ for which the two are different.  This is a well-ordering.  Assume that $\etB$ is known for
all values all predecessors of $\delta$.  If the vector space is $\etA(\delta)$-strong,  it satisfies $\rom{R}_\eta$
and we are done.  If not, for some $i$ an element of $V_i$ has an $\etA_i(\delta)$-collapse, and we can
express the element in terms of at most $2\cdot\etA_i(\delta)$ forms of lower degree. This enables us
to form a new vector space in which $\delta_j$ remains the same for $j > i$, $\delta_i$ decreases
by 1,  and the $\delta_h$ for $h <i$ increase by a total of $2\cdot\etA_i(\delta)$.  If we let $\delta'$ run through
all dimension sequences  with the properties just described (these automatically precede $\delta$ in the 
well-ordering), we may take $\etB(\delta) = \max_{\delta'}\{\etB(\delta')\}$. 

Once we have the functions $\etB$, it is easy to use them to give bounds on projective dimension independent of
the number of variables.  Suppose one wants to give a bound $C(r,s,d)$ on the projective dimension of the cokernel of an
$r \times s$ matrix over a polynomial ring $R$ whose entries (which need not be homogeneous) have 
degree at most $d$.  Each entry
is the sum of at most $d$ forms of positive degree and, possibly, a scalar.  Hence, all of the entries of the matrix
are in a ring generated by $rsd$ forms of positive degree at most $d$.  The function $\etB$ places all of these
form in a subalgebra of $R$  generated by, say, $B$ forms of degree at most $d$ that form
a regular sequence, where $B$ depends on $r$, $s$, and $d$ but not in the number of variables. The regular
sequence generates a polynomial ring $A$ such that $R$ is flat  (even free) over $A$. (We may extend the
regular sequence to a homogeneous system of parameters, and $R$ is free over the ring the parameters
generate).   The minimal free resolution over
$A$ has length at most $B$, and this is preserved by the flat base change to $R$. 

\section{Basic results on strength and safety}\label{safety}  

Then notions of $k$-collapse and strength were discussed briefly earlier:  see Discussion~\ref{strength}.
In this section we introduce the additional notions of {\it collective collapse} and {\it safety}, and 
systematically study all of these properties.

\begin{definition}\label{defsafety}{\bf Collective collapse and safety.}   
 A set $\cF$ of forms of a graded $K$-algebra $R$  is  {\it  \cs k} if it is not contained   
in an ideal generated by $k$ forms of positive degree.  The statement for the set of forms $\cF$ is evidently equivalent
to the same statement for the ideal the forms in $\cF$ generate.  {\it Note that in this definition there is no restriction on the
degrees of the $k$ forms}, which is an important difference between safety and strength conditions.
The degrees of the generators of the ideal can always be taken to be at most the highest degree of a form
in the set $\cF$. If the set of forms $\cF$ is not 
$k$-safe, we shall say that it has a {\it collective $k$-collapse}.
In the discussion below, we refine this terminology to keep track of how many of the $k$ forms that give
the collapse have the highest degree of an element of the set $\cF$.

  If the highest degree of a form in the set $\cF$ is $d$, we shall say that the set of forms $\cF$ has
a {\it collective $(k-h,h)$-collapse}, where $0 \leq h \leq k$,  if it is contained in an ideal generated by 
at most $k-h$ forms of degree at most $d-1$,
which we call the {\it auxiliary ideal} for the collective $(k-h,h)$-collapse,
and at most $h$ forms of degree $d$. We refer to the $K$-vector space generated by these $h$ $d$-forms
as the {\it auxiliary} vector space for $(k-h,h)$-collapse.

 We shall refer to a collective $(k,0)$-collapse as a {\it strict collective} $k$-collapse.  In the case of
 a single form $f$, a collective $(k,0)$-collapse is the same as a $k$-collapse, defined earlier. However, we 
 may also use the term {\it strict $k$-collapse}, to emphasize that the degrees of the generators are strictly
 smaller than $\deg(f)$. 
 
If a set of forms does not have a collective $(k-h,h)$-collapse we say that it is {\it $(k-h,h)$-safe}.
Thus, a set of forms $\cF$ is  $k$-safe if and only if there do not exist integers $k_0,h \in \N$  such
that $k_0+h \leq k$ and $\cF$ has a $(k_0,h)$-collapse. \end{definition}

\begin{proposition}\label{obv}  Let $K$ be a field and $R$ an  $\N$-graded $K$-algebra. Let $V$ be a finite-dimensional
 graded $K$-vector subspace of $R$ with dimension sequence $(\vect n d)$.  Let $\kappa = (\vect k d)$ and
 $\lambda = (\vect hd)$ be elements of $\N^d$.
Let $k$ and $h$ be positive integers.  We write $\kappa +h$ for $(k_1 +h, \, \ldots,\, k_d +h)$.
\begin{enumerate}[(a)]

\item If $V$ is $\kappa$-\gu\, and $h_i \leq k_i$ for $1 \leq i \leq  d$, then every graded $K$-vector subspace 
of $V$ is $\lambda$-\gu. Hence, if $V$ is $k$-\gu\ and $h \leq k$, then every graded $K$-vector subspace of $V$ is $h$-strong.  

\item If the sequence $\vect F m$ is $\kappa$-strong, the sequence  $\vect G n$ is $\kappa'$-\gu,  
and $\deg(F_i)\not= \deg(G_j)$ for 
all $i,j$,  then $\vect F m,\,\vect G n$ is $\kappa''$-strong, where for each degree $s$ that occurs,
$\kappa_s''$ is $\kappa_s$ (respectively, $\kappa_s'$) if $s$ 
is the degree of some $F_i$ (respectively, some $G_j$). 

\item If $\vect  G h, \, \vect F n$ is $(\kappa + h)$-\gu\ and $\deg(G_i) \leq \deg(F_j)$ for all $i$ and $j$
(which is automatic if the $G_i$ are linear forms) then the image of the sequence $\vect F n$ is 
$\kappa$-\gu\ in $R/(\vect G h)R$. 
If $V$ is $(\kappa+h)$-\gu\ and contains no linear forms, and $W$ is a vector space generated by linear forms  
with dim$_KW \leq h$, then the image of $V$ is $\kappa$-\gu\ in $R/(W)$.

\item If a family of forms of positive degree is $k$-safe, then so is every larger family.   

\item If a set of forms of positive degree is \cs {(k+h)}, then it is still \cs k if one omits $h$ elements.   

\item Suppose that $V$ is $\kappa$-strong, that $F \in V$ has degree $e$, and that $F$ is part 
of a basis $\cB$ for $V$ consisting of forms.  Let $n'$ be the number of forms in $\cB$ of degree 
$< e$, i.e. $n' = \sum_{j=1}^{e-1}n_j$, and let $\vect G m$ be any forms in $R$ having positive 
degrees $< e$.  Suppose that $n'+m \leq k_e$ (respectively, $< k_e$).  Let $I$ be an ideal of 
$R$ generated by the $\vect G m$, a subset of the elements of $\cB$ different from $F$, and a set of forms of
$R$ of degree $> e$.   Then $F$ is nonzero (respectively, irreducible) modulo $I$.  

\item If $V$ is $k$-strong for $k \geq 1$, the elements of $V-\{0\}$ all have the same degree,  and $\vect F h$ is any 
sequence of $K$-linearly independent elements of $V$, then each $F_i$ is irreducible modulo the ideal generated
by the $F_j$ for $j \not= i$.  

\end{enumerate} \end{proposition}

\begin{proof} (a) and (d) are immediate from the definitions. Part (b) follows from the fact that in checking
the condition (including whether one has linear independence), one need only consider linear combinations of 
forms of the same degree.   

For part (c), if one has a $k_i$-collapse of a homogeneous element  $F$ of degree $i$ 
mod $(\vect Gh)R$  or if $F$ becomes 0,   then there is a homogenous lifting, which means
that  $F = \sum_{s=1}^{k_i} P_sQ_s + \sum_{t=1}^h H_t G_t$  where $\deg(P_s) + \deg(Q_s) = \deg(F) =i$
for every $s$, and $\deg(H_t) + \deg(G_t) =i$ unless $H_t = 0$.  The $G_t$ of degree $i$ have
scalar coefficients:  let  $G$ be the sum of all such terms on the right.  The other $G_t$ with nonzero
coefficients have degree strictly smaller than $i$. Then $F-G$ is a nonzero
form of degree $i$ (since the original elements are linearly independent over $K$) that has a
$(k_i+h)$-collapse in $R$,  since the other nonzero terms have a factor of degree smaller than $i$. 
The second statement follows from the first, because $W+V = W \oplus V$ is also $(\kappa +h)$-\gu\ 
(a vector space of linear forms is $t$-strong for all positive integers $t$). 

For part (e), if the forms are contained in
a $k$-generated homogeneous ideal $I$ after forms  $\vect G h$ are omitted, then the original forms
are contained in the ideal  $I + (\vect G h)R$.

In part (f), the parenthetical statement about reducibility follows from the first statement:  if $F$ reduces,
we may include a representative of one of the factors among the $G_j$, which increases $m$ by 1.
If the conclusion of the main statement fails, then for a $K$-linear combination $F^*$ of forms of degree $e$ in 
$\cB$ other than $F$, $F - F^* \in V$ would be 0  modulo
the ideal $J$ generated by $\vect G m$ and the elements of $\cB$ of degree $< e$ 
(by a degree argument, we do not need to consider the generators of degree $>e$).  
Since $J$ has at most $m+n'$ generators, this expresses $F-F^*$ as a sum of multiples of 
at most $m+n'$ forms of strictly lower degree. 
Since $F-F^*$   is a nonzero form of $V$,  this contradicts the fact that $k_e\geq m+n'$.    Part (g) is a very 
special case of (f) in which $n' = m = 0$.
\end{proof}

\begin{proposition}\label{colcl}  Let $K$ be an algebraically closed field, and let $R$ be 
a polynomial ring in finitely many variables over  $K$. Let $d \geq 1$ be an integer and let 
$R_d$ denote the $K$-vector space of $d$-forms in  $R$.  Let $k$ be a positive integer, and let
$\vect d k$ be positive integers that are $\leq d$.
\begin{enumerate}[(a)]

\item The set of elements $f$ of $R_d$ that are contained in an ideal $I$ with $k$ generators
$\vect f k$ such that $f_i \in  R_{d_i}$  is constructible in $R_d$.  The set of points of $\PP(R_d)$ represented
by such a nonzero $f$ is constructible in $\PP(R_d)$.  

\item Let $h \in \N$. The set of elements $f$ of $R_d$ that have a $(k,h)$-collapse is constructible in $R_d$.
The set of points of  $\PP(R_d)$ represented by such a nonzero $f$ is constructible in $\PP(R_d)$.  

\item The set of elements $f$ of $R_d$  having a strict $k$-collapse is constructible in $R_d$.  
The set of  points of $\PP(R_d)$ represented by such a nonzero $f$ is constructible in $\PP(R_d)$. 

\end{enumerate}
\end{proposition}
\begin{proof} (a) The specified set of elements is the image of the map of affine algebraic sets  $\prod_{i=1}^k (R_{d_i} \times R_{d-d_i})
\to R_d$ sending the element whose $i\,$th entry is $(f_i,g_i)$, $1 \leq i \leq k$, to $\sum_{i=1}^k f_ig_i$, and the image of a map
of affine algebraic sets is constructible.  The corresponding statement for projective case follows.

(b) and (c).  The set in (b) is the finite union of the sets corresponding to all choices of $\vect d {k+h}$ 
such that for precisely $h$ values of $i$,  $d_i   = d$, while for all other values of $i$,  $d_i < d$.  
Each of these sets is constructible by part (a), and, hence, so is the union.   Part (c) is the special case of (b) where $h = 0$. 
\end{proof}

\begin{remark} See also Corollary~\ref{closed} for a strengthening of part (c) 
in the case of quadratic forms. \end{remark}

\begin{proposition}\label{basech} Let $K \inc L$ be algebraically closed fields.  Let $R = K[\vect x N]$
be a polynomial ring, and $S = L\otimes_K R \cong L[\vect xN]$.  Let $a \leq k, h \in \N$. 
\begin{enumerate}[(a)]
\item Let $F$ be a form of positive degree $d$ in $R$, and let  $\vect d k$ and $\vect e k$ be two sequences of
positive integers such that $d_i+e_i = d$ for $1 \leq i \leq k$. 
Suppose that $F$ can be written in the form  $\sum_{i=1}^k G_iH_i$,
 where every $G_i \in S$ has degree $d_i$ or is 0 and
every $H_i \in S$ has degree $e_i$ or is 0. Also suppose that $\vect G a \in R$. 
Then $F$ can be written
in the same form over $R$ without changing  $\vect G a$.   
\item Let $F$ be a form in $R$.  $F$ has a $k$-collapse in $R$ if and only if it has a $k$-collapse in $S$. 
\item A sequence of forms in $R$ is $k$-\gu\ if and only if it is $k$-\gu\ in $S$.   
\item Let $\kappa = (\vect k d) \in \N^d$. A graded vector space $V \inc \oplus_{i=1}^d R_i$ is 
$\kappa$-\gu\ if and only if $L \otimes_KV$ is $\kappa$-\gu\ in $S$.  
\item A finite set of forms in $R$ (or a finite-dimensional graded $K$-vector subspace of $R$, or a 
homogeneous ideal of $R$) is \cs k  (respectively, \cs {(k,h)}) if and only if it is \cs k  (respectively, \cs {(k,h)}) in $S$.
\end{enumerate}
\end{proposition}
\begin{proof}(a) Let $\wt{G}_i = G_i$ for $i \leq a$, let $\wt{G}_{a+1}, \, \ldots \, \wt{G}_k \in R$ have degrees 
$d_{a+1}, \, \ldots, \, d_k$,
respectively, in $\vect x N$ with unknown coefficients, and let $\vect {\wt{H}} k$ be polynomials in $\vect x N$
of degrees $\vect e k$, respectively, with unknown coefficients.  Equating coefficients of distinct monomials in the
variables $\vect x N$ in $F - \sum_{i=1}^k \wt{G}_i\wt{H}_i$  to 0 yields a finite system of polynomial equations
over $K$ in the finitely many unknown coefficients of $\wt{G}_{a+1}, \, \ldots, \,  \wt{G}_k$ and $\vect {\wt{H}} k$.  
Since there is a solution in $L$,  there is a solution in $K$ by Hilbert's Nullstellensatz.

Note that the ``only if"  parts of (b) and the ``if" parts of (c), (d), and (e) are obvious.  We need to prove the other parts.
(b) is immediate from (a), and (c) follows from (d).  

To prove (d), suppose that a form in $(L \otimes_K V)_e \cong L \otimes_K V_e$ 
has a $k_e$-collapse.  Choose a basis $G_0, \, \ldots, G_a$ for
$V_e$.    Then some nonzero $L$-linear combination of $G_0,\, \ldots, \, G_a$ has a $k_e$-collapse, 
and one of the $G_i$ has a nonzero coefficient.  We may assume without loss of generality that this
coefficient is 1,  and, by renumbering, that $i = 0$.  Let $F = G_0$.  Then have
an  expression  $F = \sum_{i=1}^a F_i G_i + \sum_{j=1}^{k_e} H_jP_j$, where the $F_i,\, H_j,\,$ and $P_j$ are
in $S$ and satisfy certain degree
constraints (in particular, $F_i$ has degree 0 for $1 \leq i \leq a$, while the degrees of the polynomials
$H_j$, $P_j$ are strictly smaller than $\deg(F)$), and we may apply part (a).  

The proof of (e) is similar to the proof of (a). The problem for a finite set of forms $\vect F s$ is the same as for the
$K$-vector space or the ideal they generate, and we work with the first case. 
If the $F_i$ are all in an ideal generated by $k$ forms $\vect G k$  of 
positive degree over $L$,   each $F_i$ will be a linear
combination of these over $S$:  replace the coefficients of the $G_j$ and of their multipliers in the expression
for each $F_i$ by 
unknowns. One is led to a system of equations over  $K$ which has a solution in $L$.  Therefore it has a solution 
in $K$. In the case of  $(k,h)$-safety there are additional constraints on the degrees, but the equational nature of
the problem is unchanged. \end{proof}

\section{Quadratic forms}\label{quad}
  
We briefly discuss quadratic forms in the polynomial ring $R = K[\vect x N]$ in $N$ variables over an
algebraically closed field.  There are differences between the case when the characteristic is not 2
and when it is 2.  This will not have a large effect on our study of quadratic forms, but the problem 
turns out to be much greater when studying cubic forms.

\subsection*{Background and basic results}\label{subseqQ1}
\begin{discussion}\label{qumat} If the characteristic is not 2,  
a quadratic form  $F$ is determined by a symmetric matrix $M$ of scalars:
in fact,  $M = {1 \over 2} \bigl(\partial^2F/\partial x_i \partial x_j\bigr)$.  We refer to the
rank of $M$ as the {\it rank} of the quadratic form $F$, and note that the rows of $2M$ correspond
to the partial derivatives $\partial F/ \partial X_i$:  the entries of the $i\,$th row give the coefficients
of the linear form $\partial F/\partial X_i$.  Hence, the rank of $F$ is the same as the dimension
of the vector space $\cD F = \{DF: D \in \Der_K(R,R)\}$.   If  $X = \bigl(x_1 \,\, \ldots \,\,x_N\bigr)\tr$
then $\bigl(F\bigr) = X\tr MX$.  If $X = AY$,  where $Y$ is a (possibly  different) basis for the linear forms of 
$R$, then  $X\tr M X = Y\tr (A\tr M A) Y$.   By a change of basis, the matrix $M$ can be brought
to a very special form:  the direct sum of an $r \times r$ identity matrix $I_r$ and an $(N-r) \times (N-r)$
zero matrix.  Thus, for a suitable choice of $A$,  $F$ is represented as the sum of $r$ squares of
mutually distinct variables.  Since $x_1^2 + x_2^2$ can be written as $y_1y_2$ after a change of basis,
a quadratic form of rank $r$ can be written as $$x_1x_2 + \cdots x_{2h-1}x_{2h}$$ if  $r = 2h$ is even
and as $$x_1x_2 + \cdots x_{2h-1}x_{2h} + x_{2h+1}^2$$ if $r = 2h+1$ is odd.   
\end{discussion} 

If the characteristic of $K$ is 2, quadratic forms are classified in \cite{Arf}.  There is a primarily
expository version in \cite{LR}.  In this case, we define the rank of the quadratic form $F$ to
be the least number of variables occurring in $F$ after a linear change of coordinates.  This can
be taken as the definition of rank in all characteristics.  

By the classification of quadratic forms over an algebraically closed field one has:

\begin{proposition}\label{quadcl} If $K$ is an algebraically closed field of arbitrary characteristic and 
$F$ is a quadratic form of rank $r$
then, after a linear change of variables, $F$ can be written either as $$x_1x_2 + \cdots + x_{2h-1}x_{2h}\quad(r = 2h)$$
or as $$x_1x_2 + \cdots + x_{2h-1}x_{2h} + x_{2h+1}^2\quad (r = 2h+1).$$  
In both cases, the dimension of the $K$-vector space $\cD F$ is $r$ except in characteristic 2
when $r$ is odd, in which case the dimension of $\cD F$ is $r-1$. 

$F$ is in the ideal generated by $\cD F$ except when the characteristic of $K$ is 2 and $r$ is odd,
in which case it is in the ideal generated by $\cD F$ and one additional linear form. \qed \end{proposition}

From this we have at once:
\begin{proposition}\label{qucol} If $K$ is any algebraically closed field and $F$ is a quadratic form of rank $r$, then $F$ 
has a $k$-collapse if and only if $k \geq \lceil {r \over 2}\rceil$.  Equivalently, $F$ has a $k$-collapse if
and only if the rank of $F$ is at most $2k$. 

Also, if $\cD F$ has dimension at most $t$,  then $F$ is not $\lceil \frac{t}{2} \rceil$-strong if the
characteristic of $K$ is not  2,  and is not $\lceil \frac{t+1}{2}\rceil$-strong if the characteristic of $K$
is 2.   \qed \end{proposition}

We can improve on the constructibilty result of Proposition~\ref{colcl}(c) in the case of quadrics.
\begin{corollary}\label{closed} Let $K$ be any algebraically closed field of characteristic $\not=2$ and 
let $R = K[\vect xN]$ be a polynomial ring. Let $k \in \N_+$. Then the set of quadratic forms that 
have a strict $k$-collapse is closed in the vector space $R_2$ of $2$-forms.  \end{corollary}

\begin{proof}  This set is defined by the ideal of $2k+1$ size minors of the symmetric matrix $M$ of 
Discussion~\ref{qumat}.\end{proof}

\begin{discussion} \label{maxrk} Let $K$ be an algebraically closed field, and $R$ the polynomial ring $K[\vect xN]$.
It is well known that when the characteristic of $K$ is not $2$,  the determinant of $(\partial^2F/\partial x_i\partial x_j)$
vanishes if and only if the form has rank less than $N$, and this is also true in characteristic 2 when $N$ is even. 
A similar criterion exists in characteristic 2 for the case where $N$ is odd.   If one
lets the coefficients of the form be indeterminates over $\Z$, one can compute the determinant as a polynomial
in these indeterminates.  The coefficients turn out to be even integers, and so one can consider the polynomial
over $\Z$ in the coefficients obtained by dividing by 2, sometimes called the {\it reduced discriminant} or the
{\it half-discriminant}.   It is then
correct that the form has rank less than $N$ over any field $K$, regardless of characteristic, if and only
if the half-discriminant vanishes.  Moreover, in the case where $N$ is odd in characteristic 2, the half-discriminant
agrees with the result of substituting for every $x_i$ in $F$ the Pfaffian of the matrix $(\partial^2F/\partial x_i\partial x_j)$ 
corresponding to deletion of the $i\,$th row and column. See \cite{DoDu}, \cite{Gr}, and \cite{Kne}.
Hence, the forms of rank less than $N$ form a
subvariety of codimension 1 in the vector space of all quadratic forms in all characteristics. Also, it follows
that if $F$ and $G$ are any two linearly independent quadratic forms in $R$, then for some choice of $a, b \in K$, not both 0, the form $aF+bG$ has rank less than $N$:  the vector space $KF+KG$ must have nontrivial
intersection with the codimension one variety of quadratic forms of rank at most $N-1$.  
\end{discussion}

\begin{proposition}\label{rk} Let $K$ be an algebraically closed field,
and let $F,\,G$ be quadratic forms in a polynomial ring $R = K[\vect x N]$ over $K$ such that $F$ has rank $r$.   
\begin{enumerate}[(a)]
\item For any $G$ and in all characteristics,  for all but at most $r$ choices of $c \in K$, 
the rank of $cF+G$ is at least $r$.  
\item Suppose that
the characteristic of $K$ is $\not=2$ or that $r$ is even.  If $G$ is not in the ideal
$(\cD F)R$,  then for all but at most $r$ choices of $c \in K$, the rank of $cF+G$ is at least $r+1$.  
\item If $K$ has characteristic 2 and the rank of $F$ is odd,
then either $G$ is the sum of a quadratic form in $(\cD F)R$ and the square of a linear form,  or
for all but at most $r-1$ choices of $c \in K$,  the rank of $cF + G$ is at least $r+1$.  
\end{enumerate}
 \end{proposition} 
\begin{proof}   For any quadratic form $H$, we use $M_H$ for the Hessian matrix of $H$. 

We first prove (a) when $K$ has characteristic not 2 or characteristic 2 and the rank $r$ is even.
After a change of basis, we may assume that the matrix $M_F$ of $F$ is the direct sum of an $r \times r$ matrix $M_0$ and a zero matrix, where $M_0$ is either an $r \times r$ identity matrix or the direct sum  $A$ of $r/2$ copies of  $\Lambda = \mx 0 & 1\\ 1 & 0 \emx$, depending on whether the characteristic is not 2 or 2. 
The determinant of the size $r$ square submatrix in the upper
left corner of $cM_F+M_G$ is a monic polynomial in $c$ of degree $r$.  This shows that for all but at most $r$
choices of $c$,  the rank of $cF+G$ is at least $r$.  

If $G$ is not in $(\cD F)R$ first suppose that $M_G$ has a nonzero entry 
$a_{ij}$ for $i, j > r$, which is always true if the characteristic is not 2. 
Consider the size $r+1$ square submatrix of $cM_F + M_G$ that contains the $r \times r$ submatrix in the upper left corner
and the $i,j$ entry.  The determinant of this submatrix is a polynomial in $c$ whose highest degree term
is  $a_{ij}c^r$, and so this polynomial is nonzero except for at most $r$ values of $c$, and the rank of
$cF + G$ is at least $r+1$ except for these values.   

In the remaining cases for (a), (b), and  (c) we may now assume that the field has characteristic 2 and that 
with   $F_ 0 =x_1x_2 + \cdots + x_{2h-1}x_{2h}$,
we have that $r = 2h$ and $F = F_0$ or that $r = 2h+1$ and $F = F_1 = F_0 + x_{2h+1}^2$.  We may write
$G = P+Q$ where $P$ is in the ideal generated by $\vect x {2h} = (\cD F)$ and $Q \in K[x_{2h+1}, \ldots, \, x_N]$.
It remains to prove part (a) when the rank of $F$ is odd,  part (b) when the rank of $F$ is even and $Q$ is
a nonzero square (otherwise some $a_{ij}$ for $i,j >r$ is nonzero, a case which has already been covered), and part 
(c). In considering part (c),  we may assume that $Q$ is not a square.  

Part (a) for forms $F$ of odd rank in characteristic 2 may be deduced as follows.
We have assumed that $F$ can be expressed in terms of $\vect x r$.  It suffices to prove the
statement after specializing the other variables to 0.  But then $cF + c'G$ has maximum rank
if and only if its half-discriminant is not 0, and this is a homogeneous polynomial of degree $r$ in $c$ and $c'$ 
which is nonzero when $c = 1$, $c' = 0$.   Hence,  $c^r$ has nonzero coefficient in $K$.  It follows that
when we substitute $c' = 1$,  there are at $r$ most values of $c$ for which the half-discriminant vanishes.
We have now proved part (a) in all cases.


We next prove the remaining case of part (b).  We may
assume that $r = 2h$ is even and that $Q$ is a square. We may make a change of variables, and then we 
might as well assume that
$Q = x^2$, where $x = x_{2h+1}$.  Since the rank can only drop when we kill $x_j$ for $j > 2h+1$,  we 
may assume that $\vect x {2h+1}$ are all the variables.  Then we may 
assume $F = F_0$ and $G =  P + x^2$ where $P\in (\vect x {2h})R$.  Note that $M_F$ is the direct
sum of $h$ copies of $\Lambda = \mx 0 & 1\\ 1 & 0 \emx$ and a $1 \times 1$ zero matrix. 
We want to prove that except possibly for $2h$ values of $c$,  the rank of $cF+G$ is $2h+1$. 

By Discussion~\ref{maxrk} a quadratic form $H$ in $2h+1$ variables has rank $2h+1$ if and only if $H$ is nonzero at
$(\vect p {2h+1})$,  where $p_i$ is the Pfaffian of the matrix obtained by omitting the $i\,$th row and
column of the Hessian $M_H$.   Consider the Pfaffians of $cM_F + M_G$.
Then $p_{2h+1}$ is monic in $c$ of degree $h$ (its square is monic of degree $2h$),  
while all of the other Pfaffians $p_i$, $1 \leq i \leq 2h$ are polynomials in $c$ of
degree at most $h-1$ (their squares are polynomials of degree at most $2h-2$) .  
When we substitute the $p_i$ into $cF + G$,  the largest power of $c$
that can occur in $cF$ is $c(c^{h-1})^2 = c^{2h-1}$.   Likewise, the terms of $G$ other than $x^2$ yield
polynomials in $c$ of degree at most $2h-1$, because $G$ is in the ideal generated by $\vect x {2h}$.   
Thus, the $x^2$ term contributes a unique term
of degree $c^{2h}$ with coefficient 1, and so the value of $cF+G$ at $(\vect p {2h+1})$ is a monic
polynomial in $c$ of degree $2h$.   Hence, there are at most $r = 2h$ values of $c$ such that
$cF + G$ has rank strictly less than $2h+1 = r+1$. 

Finally, to prove (c) we may assume that $F = F_1$ has rank $2h+1$, and that $G = P + Q$, where
$P \in (\vect x {2h})R$ and $Q \in K[x_{2h+1}, \, \ldots, \, x_N] -\{0\}$ is {\it not} a square.  Hence, 
$Q$ has a term that is  $ax_{2h+1} x_j$ for  $j > 2h+1$ or, if not, a term that is  $ax_jx_k$ 
where $k > j > 2h+1$, where in both cases, $a \in K - \{0\}$. In the latter case we may kill 
$x_k - x_{2h+1}$, and so reduce to the former case. 
 Consider the size $2h+2$ square submatrix $D$
 of the matrix of $cF_1 + G$ corresponding to the rows and columns numbered $1, 2, \, \ldots, \, 2h+1, j$.
 $D$ has the block form $\mx cA+B & C \\ C\tr & a\Lambda \emx$,  and where $A$ is the direct sum of
 $h$ copies of $\Lambda = \mx 0 &1\\ 1 & 0 \emx$.  The determinant of $D$ is a polynomial in  $c$ in which the highest degree term
 is $a^2c^{2h}$.  Hence, the rank of $cF+G$ is at least $2h+2 = r+1$ for all but at most $r-1$ values of $c$.
\end{proof}

\begin{corollary}\label{cDsq} Let $K$ be an algebraically closed field and let $V$ be a vector space of quadratic 
forms in a polynomial ring $R$ over $K$.  Let $F \in V$ have maximum rank among forms in $V$.  
If the characteristic is not 2, then
then $V$ is contained in the ideal $(\cD F)$.  If the characteristic of $K$ is 2,  then the image of $V$ modulo
the ideal $(\cD F)$ is a vector space whose nonzero elements are squares of linear forms. \end{corollary} 
\begin{proof} If $F$ has maximum rank and $G \in V$ is not in $(\cD F)$, then by Proposition~\ref{rk}, when the
characteristic is not 2
we can construct a form $cF+G$ of larger rank, while if the characteristic is $2$  we can do that unless the
image of $G$ modulo $\cD F$ is a square. \end{proof}

The following result, Theorem~\ref{2tr}, plays an important role in the analysis of the case of cubics if the
characteristic is not 2.

\begin{theorem}\label{2tr} Let $K$ be an algebraically closed field of characteristic $\not=2$.  Let $V$ be a vector
space of quadratic forms that is $(2k,0)$-safe, i.e., not contained in an ideal generated by $2k$ or fewer
linear forms.   Then there exists a Zariski dense open subset of $V$
consisting of quadratic forms with no $k$-collapse. \end{theorem}

\begin{proof} By Corollary~\ref{closed} the forms with a $k$-collapse are closed, and so it suffices to show that there 
is at least one  element of $V$ with no $k$-collapse.  Assume
that every element of $V$ has a $k$-collapse.  Let $r$ denote the rank of a  form $F \in V$, with $r$
as large as possible.   Then
$k \geq \lceil {r \over  2} \rceil$ and $r \leq 2k$.  By Corollary~\ref{cDsq}, $V$ must be contained in $(\cD F)R$.  
But then $V$ is not  $(r,0)$-safe, and so is
not $(2k,0)$-safe, a contradiction.  \end{proof}

 \emph{Note that Theorem~\ref{2tr}
is false in characteristic 2.}  In fact, more generally, 
in any positive characteristic $p$, the forms $x_1^p, \, \ldots, x_N^p$ are 
$(N-1)$-safe,  but any linear combination of them has a $1$-collapse if $K$ is perfect. 
However, one can still handle the case of cubics in characteristic 2 with a careful
application of Corollary~\ref{cDsq}.

\subsection*{Existence of $\etA$ in degree 2}\label{subsecQ2}

In this section we show the existence of the functions $\etA(n_1,n_2)$.  We begin with an analysis 
of when a vector space of quadratic forms consists entirely of reducible elements, and the strength 
conditions needed to guarantee that the quotient of a polynomial ring by an ideal generated by one or 
two quadratic forms is reduced, or a domain, or a normal domain, or a UFD, and so forth.  When the
base field $K$ is understood and  $S$ is
a set of linear forms, we denote by $\sq S$ the degree 2 component of $K[S]$, i.e., the set of quadratic
forms expressible as $K$-linear combinations of products of two elements of $S$.  E.g., $\sq{x,y} = 
Kx^2 + Kxy + Ky^2$.

\begin{proposition}\label{allred} Let $V$ be a $K$-vector subspace of $R_2$.  If every element of $V$ is reducible, 
then either:
\begin{enumerate}[(1)]
\item $V$ is contained in a 1-linear prime of $R$,  or  
\item $V$ is contained in $\sq{u,v}$ for variables $u, v \in R_1$, or
\item $K$ has characteristic 2 and every element of $V$ is the square of a linear form. 
\end{enumerate}
\end{proposition}
\begin{proof} Let $W$ be the span of all the squares in $V$. If $W$ is all of $V$ then we are done if 
the characteristic is 2, since (3) holds. 
If the characteristic is not 2 and there are three or more squares of independent linear forms we have a contradiction 
because $x^2 + y^2 + z^2$ is irreducible when $x,y,z$ are variables.  
Thus, there are at most two, and (2) holds.  Henceforth we assume that $W$ is a proper subspace of $V$. 

Choose an element, which must factor  $xy$,  of $V -W$.  Let $T  = \sq{x,y}$.
We may assume that $V$ is not contained in $T$, or (2) holds.

Now consider any other element of $V-(W \cup T)$  ($V$ is spanned by these).  Call this element $uv$.   
Then $u, v, x, y$ cannot be
independent, since $uv + xy$ would then be irreducible. Moreover, $u,v$ are not both in $Kx+Ky$,  or
else $uv \in T$.    Consider a $K$-linear relation on $u, v, x, y$.   This relation cannot involve
both $u$ and $v$ with nonzero coefficient.  If it does, then we can solve for one of them:  say  $v = u + ax + by$ (the coefficient
of $u$ may be absorbed into $u$).  Here, $x,y,u$ must be variables.  Then $u(u+ax+by) + cxy = u^2 +(ax+by)u + cxy$
must be reducible for all choices of $c$.  This contradicts the following fact:\\

\noindent $(\dagger)$  If $u, x, y$ are variables and $a,b,c \in K$, then for fixed $a,b$, $u^2+ (ax+by)u + cxy$ reduces only when $c = 0$ or
when $c = ab$.  In particular, there is a value of $c$ for which it is irreducible.\\

(To see this, note that since $K[x,y]$ is normal, if there is a factorization it must be as $(u + L_1)(u+ L_2)$ where $L_1, L_2 \in K[x,y]$.
Then $L_1 L_2 = cxy$ with $c \not=0$  implies that $L_1 = a'x$ and $L_2 = b'y$, say, where  $a'b' = c$.  Since $a'x + b'y$ must be $ax + by$,
we must have $a' = a$ and $b' = b$  as well. This establishes the conclusion of $(\dagger)$.)\\

Hence, one of $u,v$, say $u$, depends linearly on $x,y$,  while $x,y,v$ are variables.  
Thus, for every element $uv$ of  $V - (W \cup T)$,  we have that
one factor is in $Kx + Ky$ while the other is $K$-linearly independent of $x,y$.  Suppose $u = ax + by$ where 
$a, b \not = 0$.  Then $(ax+by)v + xy$ is irreducible, since it is linear in $v$ and the coefficients of $v$ are relatively
prime.  Thus, $a$ or $b$ must be in 0, and so every element of $V - (T \cup W)$ is in $xR_1$ or in $yR_1$.  
This yields that
$V \inc T \cup W \cup xR_1 \cup yR_1$, and so $V$ is contained in $xR_1$ or $yR_1$, and (1) holds. \end{proof}

We next consider the case of a single quadratic form, where the behavior is well known.

\begin{proposition}\label{onequad} Let $R$ be a polynomial ring in finitely many variables over an 
algebraically closed field $K$
and let $F$ be a nonzero quadratic form of $R$.  If the rank of $F$ is  $r \geq 1$, then the ideal generated
by $F$ and its partial derivatives in $R$ has height $r$.  Hence, if $r \geq 2$,
the codimension of the singular locus in $R/FR$ is  $r -1$, so that $R/FR$ satisfies the Serre
condition \rom{R}$_{r-2}$.  If $F$ is $k$-strong for $k \geq 1$,  then rank of $F$ is $\geq 2k+1$, 
and so $R/FR$ satisfies \rom{R}$_{2k-1}$.  Hence, $R/FR$ satisfies \rom{R}$_\eta$ if 
$k \geq \lceil (\eta + 1)/2\rceil$. 

\end{proposition}
\begin{proof} If the characteristic of $K$ is not 2, or if the characteristic is 2 and $F$ has even rank,
the partial derivatives of $F$ span a vector space of dimension $r$, and $F$ is in the ideal they generate.
Hence, the ideal generated by $F$ and its partial derivatives has height $r$, and the height of the
defining ideal of the singular locus in $R/FR$ is $r-1$.  If the characteristic is 2 and $F$ has odd rank,
we may assume that $F$ has the form $x_1x_{h+1} + \cdots + x_hx_{2h} + x_{2h+1}^2$.  The ideal
generated by $F$ and its partial derivatives in $R$ is $(\vect x {2h}) + (x_{2h+1}^2)$, and the height
is still  $r = 2h+1$.   
 \end{proof}

\begin{remark}\label{rkUFD} 
Note that by Corollary \ref{groth}, if $F$ has rank 5 or more (equivalently, if $F$ is 2-strong), then
$R/FR$ is a UFD.  This is a much more elementary result which follows from the classification
of quadratic forms and Lemma~\ref{UFD} below. \end{remark}

 We first note the following result of Nagata:  see \cite{Sam}, Theorem 6.3 and its Corollary.
 
 \begin{lemma}\label{Nag} Let $f$ be a nonzero prime element in a domain $R$.  Then $R$ is a UFD
 if and only if $R_f$ is a UFD. \qed \end{lemma}

\begin{lemma}\label{UFD} Let $A$ be a UFD and $R = A[x_1, x_2]$ be a polynomial ring over $A$. 
Let $F = x_1x_2 + G$  where  $G \in A[x_1]$ and $G(0) = a \in A-\{0\}$ is irreducible.  Then $S = R/FR$ 
is a UFD.  \end{lemma}
\begin{proof}  First note that $F$ is irreducible in the UFD $R$,  since it is linear in $x_2$ and the coefficients
have greatest common divisor 1. Hence, $S$ is a domain. 
The hypothesis also yields that $x_1$ is prime in $S$ (the quotient is $(A/aA)[x_2]$, a domain).
Hence, by Lemma~\ref{Nag},  $S$ is a UFD if 
and only if  $S_{x_1}$ is a UFD.  But once we localize at the
element $x_1$,  the equation $F$ is, up to a unit,  $x_2 + x_1^{-1}G$, and killing $FR$ yields
$A[x_1]_{x_1}$, a UFD. \end{proof}

The following result describes behavior for a vector space of quadratic forms of dimension $n$.  We state the general
result, then prove the case where $n=2$, which plays a special role in the proof of the general result.

\begin{theorem}\label{nforms}  Let $K$ be an algebraically closed field and let $R$ be a polynomial ring over $K$. Let 
$\vect F n$ be linearly independent quadratic forms in $R$, where $n \geq 2$, and let $V$ be the $K$-vector space they span.   
Let $I = (\vect F n)R$.  If $V$ is $k$-strong for $k \geq n-1$,  then $\vect F n$ is a regular sequence, and $R/I$ satisfies the Serre condition  \rom{R}$_{2(k+1-n)}$.  In particular, if $k \geq n-1$,  then $R/I$ is reduced, if $k \geq n$ then $R$ is normal, and if $k \geq n+1$ then $R/I$ is a UFD.  \end{theorem}

The authors would not be surprised if the result above is in the literature, but do not know a reference for it.  It is easy
to obtain slightly weaker results.  The authors would like to thank Igor Dolgachev for suggesting the use of Steinerians
in the proof.  The arguments we give use modifications of this idea.

Before giving the proof, we note that this result is the best possible, as shown by the following example.

\begin{example}
Assume that $k\geq n-1$.  
Let $Y  = (y_{ij})$  be an $n \times (k+1)$ matrix of indeterminates and let $\vect x {k+1}$ be $k+1$ additional indeterminates.
For $1 \leq i \leq n$, let $F_i$ denote  $\sum_{j=1}^{k+1} x_jy_{ij}$.  Any nonzero $K$-linear combination 
of the $F_i$ also can be written
as $\sum_{j=1}^{k+1} x_jy_j'$ where  $\vect x {k+1},  y_1', \, \ldots \, y_{k+1}'$ are indeterminates, 
and so has rank $2k+2$.
Thus,  the vector space $V$ spanned by the $F_i$ is $k$-strong. These are $n$ general linear combinations
of the generators of an ideal of depth $k+1$, and therefore form a regular sequence:  see, for example, the
first paragraph of  \cite{Ho}.  

The Jacobian matrix of $\vect F n$ has $Y$ as
a submatrix, while the rest of the matrix consists of $n$ rows each of which contains  $\vect x {k+1}$ and entries that
are 0:  moreover, the occurrences of the $x_i$ in a given row are in columns that are different from where they
occur in any other row.  Thus, the columns that contain $x_i$ give a permutation of the columns of the matrix 
$x_i \hbox{\bf 1}_n$.  For a suitable
numbering of the variables, the Jacobian matrix has the form
$$\mx 
y_{1,1}&\ldots &y_{1,k+1}&x_1\,\, \ldots \,\, x_{k+1}&\!\!\!0\ \, \ldots \ \ \ 0 & \ldots & \!\!\!0\ \, \ldots \ \ \ 0\\
y_{2,1} &\ldots &y_{2,k+1} & \!\!\!0\ \, \ldots \ \ \ 0 & x_1\,\, \ldots \,\, x_{k+1} & \ldots&\!\!\!0\ \, \ldots \ \ \ 0\\
           \vdots                       & \vdots &    \vdots   &        \!\!\!\!   \vdots   &   \!\!\!\! \vdots  &  \vdots  & \!\!\!\!\vdots\\
y _{n,1}&\ldots & y_{n,k+1} &\!\!\!0\ \, \ldots \ \ \ 0 &\!\!\!0\ \, \ldots \ \ \ 0 &\ldots&x_1 \,\, \ldots \,\, x_{k+1}
\emx$$
It follows that the radical of the ideal of minors of the Jacobian contains the ideal $(\vect  x {k+1})$, and, hence, all of the 
$F_i$.  Thus, the radical of the ideal of minors of the Jacobian matrix is $I_n(Y) + (\vect x {k+1})$.  
Under the assumption that $k\geq n-1$,  i.e., that 
$k+1 \geq n$, we have from the main result of \cite{EN} or \cite{HE}  that this ideal has height  
$(k+1)-n +1 + (k+1) = 2(k+1) - n + 1$ in the polynomial ring.  Thus, the height
of its image in $R/(\vect F n)$ is $2(k+1-n)+1$.  This means that the quotient satisfies R$_{2(k+1-n)}$, but 
not R$_{2(k+1-n)+1}$.  \end{example}

We first prove Theorem~\ref{nforms} when $n = 2$,  and then give the proof in the general case. 
In the case where $n=2$ we change notation slightly to avoid subscripts and write $F$, $G$ for
$F_1$, $F_2$.   Before proving
the result when $n=2$ we note: 

\begin{lemma}\label{iso} Let $R = K[\vect x N]$ where $K$ is an algebraically closed field. 
Assume that $F$ and $G$ do not lie in a subring of $R$ generated $N-1$ or
fewer linear forms, i.e., roughly speaking, that we cannot decrease $N$.  Assume also that
$V(F,G)$ has a singularity other than the origin.  Then we can make a linear change of
variables and choose new generators for $V = KF+KG$ one of which has the form $x_1x_2 + Q(x_3, \, \ldots, \, x_N)$,
and the other of which only involves $x_2, \ldots, x_N$. \end{lemma}
\begin{proof}  Since the singular locus is not just the origin, after
a linear change of coordinates we may assume that $p_0 = (1,0,\, \ldots,\,0)$ is a singular point of
$V(F,G)$.  
Then $x_1^2$ does not occur in either
$F$ or $G$ (since they vanish at $p_0$), and we may write them as $F = Ax_1+ Q$, $G = A'x_1+Q'$,
where $A, A', Q, Q' \in K[x_2, \, \ldots, \, x_N]$.  Then the Jacobian matrix at $p_0$ has the form
$$\mx   0 & c_2 & \ldots c_N\\ 
              0 & c_2' & \ldots c_N'\emx$$
where $A =\sum_{j=2}^N c_j x_j$ and $A' =\sum_{j=2}^N c_j' x_j$  with $c_j, c'_j \in K$.  
Since $p_0$ is a singular point, the Jacobian matrix has rank at most 1 at $p_0$, so that $A, A'$ span a $K$-vector
space of dimension at most one.  Therefore, we may replace $F, G$ by independent $K$-linear combinations such
that $A'$ is 0, and so we may assume that $F = x_1A + Q$ and $G = Q'$, where $A$, $Q$, and $Q'$ 
involve only $x_2, \, \ldots, \, x_N$.  If $A = 0$ we can reduce the number of variables $N$, and so
we may assume that $A \not=0$.  After a change of variables in $Kx_2 + \cdots + Kx_N$,  we may
assume that  $A = x_2$.  We may write   $Q = x_2B + Q''$  where $Q''$ involves only $x_3, \, \ldots, \, x_N$.
Thus,  $F = (x_1 +B)x_2 + Q''$  where $Q''$ involves only $x_3, \ldots, x_N$,  and $G$ involves only
$x_2, \, \ldots, x_N$.  Finally, we may replace $x_1$ by $x_1 + B$.  \end{proof}

\begin{discussion}  Let $F$ be a quadratic form in $K[\vect  x N]$,  where $K$ is algebraically closed.
If the characteristic of $K$ is different from 2 or if $K$ has characteristic 2 and $F$ has even rank,
then $\cD F$ is the smallest space of degree one forms such that $F$ is in the polynomial ring
they generate.  If the characteristic of $K$ is 2 and $F$ has odd rank $2h+1$, after a change of variables
$F$ has the form $x_1x_2 + \cdots x_{2h-1}x_{2h} + x_{2h+1}^2$.   In this case, there is still
a smallest space $W$ of forms of degree 1, the $K$-span of $\vect x {2h+1}$,  such that $F$ is in the polynomial
ring generated by these linear forms:  it also may be described as $\cD F + Kx_{2h+1}$.  Here,  while
$x_{2h+1}$ is not uniquely determined, the vector space $W$ is, and may be described as the intersection
of $\Rad(\cD F, F)$ with the forms of degree 1 in the polynomial ring $R$.  We shall use the notation
$\tD F$ to denote this space.  It is the same as $\cD F$ whenever the characteristic is not 2 or $F$ has even rank.  
The image of the ideal generated by $\tD F$ in $R/F$ is always the defining ideal of the singular locus of $F$.  \end{discussion}

We now give the proof of Theorem~\ref{nforms}. 
\begin{proof} We first give the proof when $n=2$, and write $F := F_1$ and $G := F_2$, as mentioned above.
Since $F$ is 1-strong, it is irreducible, so that $FR$ is prime. We have that $G$ is irreducible and, in particular,
nonzero  mod 
$F$ (although not necessarily prime) by part (g) of Proposition~\ref{obv}.  Thus, it is obvious that
$F, \, G$ is a regular sequence.  Note that once we know this,  reduced is equivalent to \rom{R}$_0$,
and normal is equivalent to \rom{R}$_1$.  In the normal case, the domain condition is automatic: see
Discussion~\ref{RC}.

We know that $k \geq 1$, and we want to show that $R$ is reduced. 
If any form in $V$ has rank 5 or more,  say  $F$, the quotient by it is a UFD (see Remark~\ref{rkUFD}).
Then, since  $G$ is irreducible in $R/FR$,   $R/(F,G)R$ is a domain.  

Therefore, we may assume that every nonzero element of $V$ has rank 3 or rank 4. 
We may assume without loss of generality that
$F = xy - z^2$ or  $xy-wz$. Let the remaining variables be $\ux$. Then $K[x,y,z, \ux]/(F)$ 
(respectively, $K[w,x,y,z, \ux]/(F)\,$) is isomorphic with  $C[\ux]$ where $$C =
K[u^2, v^2, uv] \inc K[u,v] = D$$ or  
$$C =   K[su, sv, tu, tv] \inc  K[s,t,u,v] =  D,$$ where $s,t,u,v$
are new indeterminates.  Note 
$ C\to D$ is split over $C$, and this remains true when we adjoin $\ux$.  Thus, every ideal
of $C[\ux]$ is contracted from $D[\ux]$.  

Let $g = G(u^2, v^2, uv, \ux)$  (respectively,  $G(su, sv, tu, tv, \ux)$).
Then  $R/(F,G) \cong
C[\ux]/(g) \inj D[\ux]/(g)$.  It will suffice to show that $g$ is square-free
in the polynomial ring $D[\ux]$, which  we grade so that $u,v$ (and $s, t$) have degree one and the elements of
$\ux$ have degree 2.  If $g$ is not square-free then either (1) $g = h^2$, where $h$ has degree 2,
or (2)  $g = h^2r$ where $h$ has degree 1 and $r$ has degree 2.  We obtain a contradiction by showing
that $g$ factors in $C[\ux]$,  i.e., that $G$ factors in $R/FR$,  which contradicts Proposition~\ref{obv}~(g).

We first study cases (1) and (2) when $C = K[u^2, v^2, uv]$.  
In case (1),   $h = Q + M$  where $Q$ is quadratic
in $u,v$ and $M$ is linear in $\ux$.   Then $Q \in C$, obviously, and we have the required
factorization.  

In case (2) for $C = K[u^2, v^2, uv]$,   $h$ is a linear form in $u,v$  and
$r$ must be the sum of a quadratic form $Q$ in $u$, $v$ and a linear form $M$ in the
$K$-span of $\vect xn$.  Thus,  $g = h^2(Q  +  M)$.   Then $h^2 \in C$,  and
because of the splitting $C[\ux] \to D[\ux]$, $g$ is a multiple of $h^2$ in $C[\ux]$.  Since the multiplier
needed is uniquely determined, 
$Q+M$ must be in $C[\ux]$, and we have a factorization of $g$ in $C[\ux]$.  

Now assume $C = K[su, sv, tu, tv]$.  Let the multiplicative group $G = \GL(1,K)$ of $K$ act on the 
polynomial ring $D[\ux] = K[s,t,u,v,\ux]$
so that $a \in G$ sends  $s \mapsto as$, $t \mapsto at$,  $u \mapsto a^{-1}u$, $v \mapsto a^{-1}v$ and 
fixes all of the variables $\ux$.  The fixed ring of this action is $C[\ux]$.  

In case (1),   $h = Q + M$  where $Q$ is quadratic
in $u,v, s, t$ and $M$ is linear in $\ux$.   Working mod $(\ux)$, we see that $Q^2  \in C$.  
Since $Q^2$ is invariant under the action of $G$,  so is $Q$ (the group $G$ is connected),
i.e., $Q \in C$.  But then  $h \in C[\ux]$, which yields a factorization of $G$ mod $F$.

In case (2), since $g = h^2r$ is fixed by the action of $G$,  the ideals generated by
$h$ and $r$ in $D[\ux]$ 
must be fixed by this action (they cannot be permuted, since $G$ is connected). This implies that the action of
$G$ maps $h$ to multiples of itself by elements of $K$, i.e., $h$ is semi-invariant for the action, and the
same holds for $r$.  Hence, $h$
must be $\alpha s + \beta t$  or   $\alpha u + \beta v$,  where $\alpha, \, \beta \in K$ are not both 0.   
Then  $h^2$ is a semi-invariant that is multiplied by $a^2$ or $a^{-2}$ as $a \in G$ acts.  Now $r$ must have 
the form $Q + M$, where $Q$ is quadratic in $u,v,s,t$ and $M$ is linear in the variables in $\ux$, and since
$r = Q+ M$ is a semi-invariant that is multiplied by $a^{-2}$ or $a^2$ as $a \in G$ acts, we must have that $M =0$,
and that if $h$ involves $u,v$ (respectively, $s,t$), then $Q$ involves $s,t$ (respectively, $u,v$).  
Then  $Q$ factors $LL'$ where  $L, L'$ are linear forms in the appropriate pair
of variables,  and we have  $g = (hL)(hL')$ where the factors are in $C \inc C[\ux] \cong R/(F)$.  As before, the fact that $G$ factors mod $(F)$ contradicts Proposition~\ref{obv}~(g).
This completes the proof of the \rom{R}$_0$ condition for $k \geq 1$.

In the remainder of the proof 
we may  assume that $F$ and $G$ are not contained in a polynomial $K$-subalgebra of $R$ generated
by fewer than $N$ linear forms:  if they are, we carry out the proof in this subring.  The codimension of the singular
locus and the properties we are considering do not change when one omits the superfluous variables.
Assume that $k \geq 2$ and that we know the result for smaller $k$.  
We may assume that $R$ is generated by $\tD F$ and $\tD G$:  again, omitting unneeded variables does
not change the codimension of the singular locus.  Moreover, we may assume that
$N \geq 2k+2$:  otherwise, by Discussion~\ref{maxrk}, some  non-trivial linear combination
of $F$ and $G$ has rank at most $2k$ and so is not $k$-strong. We write $J(F,G)$ for the Jacobian matrix of $F$, $G$. 
To prove that the singular locus in $V(F,G)$ has codimension at least $2k-1$, we want to prove that the height of the defining ideal in $R/(F,G)$ is at least $2k-1$.   This is equivalent to showing that the height
of the ideal $I = (F,G) + \Ijac$ is at least $2k+1$ in $R$.  We assume, to the contrary, that
the height of some minimal prime $P$ of $I$ is at most $2k$.   

If the singular locus consists only of the origin, then it has codimension $N-2 \geq (2k+2)-2 = 2k$,
and we are done.  Hence, we may apply Lemma~\ref{iso} and assume  $F = x_1x_2 + Q(x_3,\, \ldots, \, x_N)$
while $G$ involves only $x_2, \, \ldots, \, x_N$.  We may write $G = x_2L + M$  where
$L = \sum_{j=2}^N c_jx_j$ with the $c_j \in K$ and $M \in K[x_3, \, \ldots, x_N]$.  
Then the Jacobian matrix $J(F,G)$ is 
$$ \mx  x_2 & x_1 & Q_3 & \ldots & Q_N \\
               0& L +c_2x_2 & M_3+ c_3x_2 & \ldots & M_N + c_Nx_2\emx ,$$
where $Q_j, M_j$ denote partial derivatives with respect to $x_j$.  It follows that $x_2 \in P$ or 
that all the partial derivatives of $G$ are in $P$.  Since $G \in P$ as well, we have that
$\tD G \in P$.  Since $G$ has no $k$-collapse,  its rank is at least $2k+1$, and this gives a 
contradiction.  Thus, $x_2 \in P$.

We next observe that  $KQ + KM$ is $(k-1)$-strong in $K[x_3, \, \ldots, x_N]$:  if  $a,b$ are not both
0 and $aQ +bM$ has a $(k-1)$-collapse, then
$$aF + bG = x_2(ax_1 +bL) + (aQ+bM)$$ has a $k$-collapse, a contradiction.  This implies that
$(Q,M) + I_2\bigl(J(Q,M)\bigr)$ has height at least $2k-1$ in $K[x_3, \, \ldots, \, x_N]$.   Since
$x_2, F, G \in P$  we also have that $Q,\, M \in P$.  It follows that the contraction $P_1$ of $P$ to
$K[x_3, \, \ldots, x_N]$ has height at least $2k-1$.  Since $P$ contains $P_1$ and $x_2$,  it
follows that the contraction $P_0$ of $P$ to $K[x_2, \, \ldots, \, x_N]$ has height at least $2k$.
Since we are assuming that $P$ has height at most $2k$, we must have that $P = P_0R$ has
height exactly $2k$.  Thus,  $x_1 \notin P$.  Working mod $P$ and omitting columns that are 0
we have, since $x_2 \in P$  that $J(F,G)$ has the form
$$\mx x_1 & Q_3 & \ldots & Q_N\\
        L    & M_3& \ldots & M_N\emx. $$
Since  $x_1M_j - LQ_j \in P_0R$ for $j \geq 3$, and $M_j, L, Q_j \in k[x_2,\,\ldots, \, x_N]$,
and $P = P_0R$ is the expansion of an ideal from $K[x_2, \, \ldots, \, x_N]$,  we must have that
all of the $M_j \in P_0 \inc P$, and that all of the $LQ_j \in P$.   Since $M \in P$ as well, we have that $\tD M \inc P$.    
We must have that $L \notin P$,  since it follows otherwise that the second
row of the Jacobian matrix $J(F,G)$ is in $P$, and then we would have $\tD G \inc P$, which we have already
noted gives a contradiction ($G$ has rank at least $2k+1$.)
Since all the $LQ_j$ are in $P$ for $j \geq 3$,  we also have that all the $Q_j \in P$ for
$j \geq 3$.  Since $Q$ is in $P$, we have that $\tD Q \inc P$.  Hence, the linear
space $Kx_2 + \tD M + \tD Q \inc P$, and since  $P$ has height $2k$ and
$\tD M + \tD Q \inc Kx_3 + \ldots + Kx_N$ does not meet $Kx_2$,  we see that
$\tD M + \tD Q$ has $K$-vector space dimension at most $2k-1$.  It follows from Discussion~\ref{maxrk} 
that there are scalars $a,b \in K$, not both 0, such that $aM + bQ$ has rank at most
$2k-2$.  But then $aF + bG = (ax_1 + bL)x_2 + (aM + bQ)$ has rank at most
$(2k-2) +2 = 2k$,  a contradiction. 

In the remainder of this argument we assume that $n >2$. We use induction on $n$. 
We know inductively that $R/(\vect F {n-1})$ is a domain and so that $\vect F n$ is regular sequence.
Let $X = V(\vect F n)$ and 
let  $Z$ denote the set of points of $X$ where the rank of the Jacobian matrix is precisely $n-1$.  The singular
locus is the union of $Z$ and the closed set $Y$ where the rank of the Jacobian matrix is at most $n-2$.    
We must show that the dimension of $Z \cup Y$ is at most $N - (2k+3-n)$.  (We must show that
the codimension of the singular locus in $X$ is at least  $2(k+1-n) + 1$, i.e., that its dimension
is at most $N-n -\bigl(2(k+1-n)+1\bigr)$, since $X$ has dimension $N-n$.) To see this for $Y$,
note that $\vect F {n-1}$ vanish on $Y$, and the rank of the Jacobian matrix of $\vect F {n-1}$ on $Y$
is at most $n-2$.  Thus, $Y$ is contained in the singular locus of  $X_0 = V(\vect F {n-1})$, and $\vect F {n-1}$
is $k$-strong.  By the induction hypothesis, the dimension of the singular locus of $X_0$ is at
most $N-(2k+3-n+1)$, which is one less than we require.  

Hence, to complete the argument it will suffice to show that the dimension of each irreducible component $C$ 
of the locally closed set $Z$ is at most $N-(2k+3-n)$.  Given such a component $C$,  we have a map
$\theta: C \to \PP^{n-1}_K$ that sends each point  $z \in C$ to the (unique up to multiplication by a nonzero scalar)
relation on the rows of the Jacobian matrix evaluated at $c$:  the $n \times N$ matrix has rank exactly $n-1$.
Either the image of $\theta$ has dimension $\leq n-2$,  or the image contains a nonempty Zariski
open set $U$ in $\PP^{n-1}_K$.  Consider a point $v = [a_1: \cdots : a_n]$ in the image  of $C$,  and
let $F_v$ denote $a_1F_1 + \cdots + a_nF_n$, which is well-defined up to multiplication by a nonzero
scalar in $K$.  The set of points in $C$ in the fiber is the same as the set of points of $C$ in
$W_v  = V(\vect F n, J_v)$,  where $J_v$ is generated by the partial derivatives of $F_v$.  Hence,
the defining ideal of $W_v$ contains $\tD F_v$.  
From the $k$-strength hypothesis, the dimension of $W_v$ is at most $N-(2k+1)$. Hence, if the image
of $\theta$ has dimension $\leq n-2$,  the dimension
of $C$ is at most $N - (2k+1) + (n-2) = N - (2k +3 -n)$, as required.

Therefore we may assume instead that the image of $C$ under $\theta$ contains a dense open subset 
$U$ of $\PP^{n-1}_K$.
After decreasing $U$, if necessary, we may assume that the fibers all have the same dimension, and that
the dimension of $C$ is the dimension of a fiber plus  $n-1$.  Again by decreasing $U$ if necessary, we may
assume that for all points $v \in U$, $F_v$ has the same rank. If this rank is $2k+2$ or more, then since the
fiber over $v$ is contained in $V(\tD F_v)$,  each fiber has dimension at most $N-(2k+2)$,  and the dimension
of $C$ is at most $N-(2k+2) +n-1 =N - (2k+3-n)$, as required.  Therefore we may also assume that the
rank of every $F_v$ is $2k+1$ for $v \in U$.  

  Let $F, \, G$ be $K$-linearly independent elements of 
$KF_1 + \cdots + KF_n$  that correspond to distinct points $v_0, v_1 \in U$.  Then for all but finitely many
choices of $[a:b]$,  $aF +bG$ has the form  $F_v$  with  $v \in U$.  We claim that for $[a:b]$ in general position,   
$F$ is not in $(\tD F_v)$.  To see this, suppose otherwise.   Then the  point of $X_1 = V(\tD F_v)$ is a 
a singular point of  $V(F, F_v)$.   Since this is true for all but finitely many points of the line $L$ through
$v_0$ and $v_1$ in $\PP^{n-1}_K$,  the singular locus of $X_1$ has dimension at least
$N - (2k+1) +1$.  This contradicts the result for $n=2$, which asserts the dimension is at most
$N-(2k+1)$.  

Thus, for almost all points $v \in L$,  the fiber over $v$ in $C$ has dimension at most $N-(2k+2)$,  since
the height of $(F, \tD F_v)$ is at least one more than the height of $(\tD F_v)$.  This implies that
the dimension of $C$ is at most $N-(2k+2) +(n-1)$,  as required. \end{proof}   

\begin{examples} Let $K$ be any algebraically closed field. 

\noindent (a) Let $X = \mx u_1 & u_2 & u_3\\ v_1 & v_2 & v_3 \emx$ be a $2 \times 3$ matrix of
indeterminates over $K$,  let $\Delta_i$ denote the $2 \times 2$ minor obtained by omitting the
$i\,$th column, and let $C_i$ denote the $i\,$th column, $1 \leq i \leq 3$.  Let $V = K\Delta_2 + K\Delta_3$.  The minors 
$\Delta_2$ and $\Delta_3$ overlap in $C_1$, and if $a,b \in K$, not both 0,   $a\Delta_2 + b\Delta_3$ is the determinant
of a $2 \by 2$ matrix whose columns are   $C_1$ and $aC_2 + bC_3$:  the four entries of these two columns
are algebraically independent.  Hence, every nonzero element of $V$ has rank $4$, and $V$ is 1-strong.
The quotient $K[\und{u}, \und{v}]/(V)$ is reduced but is not a domain:  $(V)$ has two minimal primes of height 2,
one generated by all three of the size 2 minors of $X$, and the other by the two variables that are the entries
of the first column.  This shows that 1-strong does not imply that the quotient by $(V)$ is R$_1$ nor even a
domain.

\noindent (b) Now adjoin indeterminates $x,y$ (respectively $x$, $y$, and $z$) to the ring in the example above and let  
$F = xy + \Delta_2,   G = y^2 + \Delta_3$ (respectively, $G = yz + \Delta_3$).  
If $a, b \in K$ then $aF+bG =
y(ax+by) + a\Delta_2 + b\Delta_3$ (respectively, $x(ay+bz) + a\Delta_2 + b\Delta_3$) which has rank 5 or 6 (respectively,
rank  6) when $a,b$ are not both 0:  this follows because $a\Delta_2 + b\Delta_3$ has rank 4  and $y(ax+by)$ (respectively, 
$x(ay+bz)$), which involves disjoint variables from $a\Delta_2 + b\Delta_3$,  has rank 1 or 2 (respectively, rank 2). 
Hence, $V$ is 2-\gu\ but the
quotient is not a UFD:  the image of $y$ is irreducible but not prime.   This also shows that 2-strong does not imply
R$_3$. 
\end{examples}

 \begin{definition}\label{alf} For every integer $\eta \geq  0$ we define a function $\al_\eta$ from $\N_+$ to $\N_+$ 
 as follows. First,  $\al_\eta(1) := \lceil(\eta + 1)/2\rceil$.  If $n \geq 2$,  $\al_\eta(n) =n -1 + \lceil \eta/2 \rceil$.
 We also define $\alpha(n):= n-1$ for $n \geq 1$.
 \end{definition}
 
 We note that $\alpha$ (respectively, $\alpha_\eta$) is the same as $A_2$ (respectively, $\etA_2$) in the notation
 of Remark~\ref{Anot}.  It will be convenient to use the shorter notation here.
 
 \begin{theorem}\label{ngu} Let $K$ be an algebraically closed field and let $R = K[\vect x N]$ be a polynomial
 ring over $K$.  Let $\eta,\,n\geq 1$ be integers, and let $V$ be a vector space of quadratic forms
 of dimension $n$ over $K$.  If $V$ is $\alpha_\eta(n)$-strong, the quotient by the ideal generated by any
 subspace of $V$ is a complete intersection with singular locus of codimension at least $\eta+1$, i.e., satisfies
 the Serre condition \rom{R}$_\eta$,  and so is
 a normal domain.  These quotients are unique factorization domains for $n \geq 1$  if the strength of $V$ 
is at least $n+1$.   
  
 If  $V$ is $\alpha_1(n)$-strong, the quotient of $R$ by the ideal generated by any subspace
 of $V$ is a normal domain.  If $n \geq 1$ 
 and $V$ is $\alpha(n)$-\gu, then every subset of $V$ consisting
 of  elements linearly independent over $K$ is a regular sequence.    \end{theorem}  
  
 \begin{proof}  The cases where $n=1$ and $n \geq 2$ follow from Proposition~\ref{onequad} and Theorem~\ref{nforms},
 respectively. Note that if $n = 1$ and $\eta = 2$, we have that $\alpha_2(1) = 2$, and a $2$-strong form
 has quadratic rank at least 5, and so the quotient is a UFD (see Remark~\ref{rkUFD}). If $n \geq 2$, note that the level of strength, namely, $n-1$,  needed to insure that $n-2$ of the forms constitute a UFD sequence guarantees also that
 the $n$ forms are a regular sequence by Remark~\ref{stregh}.  
  \end{proof}

We can now make $\etA(n_1,\,n_2)$ explicit:
   
 \begin{corollary}\label{A2} Let $A(n_1,n_2):= \alpha(n_2) + n_1$, where $\alpha(n_2) = n_2-1$.  
 If $\eta \geq 1$ is an integer, define  $\etA(n_1,n_2) := \alpha_\eta(n_2)+n_1$.
 Let $K$ be an algebraically closed field of arbitrary characteristic, and let $R = K[\vect xN]$.
 Let $V$ be a graded $K$-vector subspace of $R$ with dimension sequence $(n_1,n_2)$, and let $I$ be 
 the ideal generated by $V$. If $V$ is $A(n_1,n_2)$-strong, then a homogeneous basis for $V$ is a regular
 sequence, so that $R/I$ is a complete intersection.
 
Moreover, if $V$ is $\etA(n_1,n_2)$-strong then not only does the previous conclusion hold, but also
the singular locus of $R/I$ is of codimension at least $\eta+1$ in $R/I$, 
i.e., $R/I$ satisfies the Serre condition \rom{R}$_\eta$,  and so is a normal domain. The ring $R/I$ is a unique 
factorization domain if the strength of $V$ is at least $n_1+n_2+1$.  \end{corollary}

 \begin{proof} By Proposition~\ref{obv}(c), modulo the linear forms in $V$, the images of the quadratic forms in $V$ 
 form a vector space of dimension at most $n_2$ in a polynomial ring and the level of strength has decreased
 by at most $n_1$. The result is 
 now immediate from Theorem~\ref{ngu}, which is the case $n_1 =0$.
 \end{proof}

\subsection*{Existence of $\etB$ in degree 2}\label{subsecQ3} 
 
  In this section we treat the situation where $\delta = (n_1,\, n_2)$. As usual, we fix an algebraically closed field 
  $K$ and consider $R = K[\vect xN]$ over $K$. 
  Let $V$ have dimension sequence $\delta$. We study $\etB(n_1,n_2)$ and also the smallest choice of $B  = B(n_1,n_2)$
 that enables us to find a regular sequence $\vect G B$ of length $B$ such that $V$ is contained in 
 $K[\vect G B]$.   
 
 In the sequel we allow $\etB$ to represent $B$ as well, i.e., we allow $\eta$  to be empty, so that we
 can treat both cases simultaneously. We make the same convention with $\alpha$ and $\alpha_\eta$,
 which, by Corollary~\ref{A2}, are the same as $A(0,n_2)$ and $\etA(0,n_2)$, respectively. 
  
 If $n_2 = 0$, we evidently
 have that $\etB(n_1) = n_1$ for all $\eta$:  the quotients are all polynomial rings. 
  Take $\alpha(n) := n-1$, which
 will give the smallest value of $B$ obtainable by our methods.  For $n \geq 1$,  if a $K$-vector space $V$ of 
 quadratic forms has dimension $n$ and is $\alpha(n)$-strong, then every set of $K$-linearly independent forms 
 in $V$ is a regular sequence by Corollary~\ref{A2}. The function $\alpha_\eta$ is given in Definition~\ref{alf}.   Note that for all $n$,
 $\alpha(n) \leq \alpha_1(n) \leq  \cdots \leq \alpha_\eta(n) \leq \cdots$.  
 
 Let $\vect x {n_1}$ be a $K$-basis for 
 the degree 1 component of $V$ and let $\vect F {n_2}$ be a $K$-basis for the degree 2 component.  We now
 make explicit the idea of  the first two paragraphs of \S\ref{BfromA}.  
 We shall show that the function defined recursively by the formulas $\etB(n_1, 0) = n_1$, $\etB(n_1, n_2) =
 \etB(2n_1 + 2\alpha_\eta(n_2),\, n_2 - 1)$ if $n_2 \geq 1$ has the property we want. Note that $B(n_1,\, n_2) \leq \etB(n_1,\,n_2)$ for
 all $\eta \geq 1$ because $\alpha_\eta(n) \geq \alpha(n)$ for all $\eta\geq 1$ and $n \geq 1$.

 It is straightforward to see 
 that if we apply the recursion $h$ times, where $1\leq h \leq n_2$, we get that 
 $$\etB(n_1,n_2) = \etB\bigl(2^h n_1 + \sum_{t=0}^{h-1} 2^{h-t}\al_\eta(n_2-t), n_2-h\bigr).$$
 Hence,  setting $h = n_2$, we have that $n_2 - h = 0$, and the value of the right term is simply
 the value of the first input. Hence, by letting $s = n_2-t$ in the summation in the second term below, we have
 $$ (*) \quad \etB(n_1, \, n_2) = 2^{n_2}n_1 + \sum_{t=0}^{n_2-1}2^{n_2-t}\al_\eta(n_2-t) = $$ 
 $$2^{n_2}n_1 + \sum_{s=1}^{n_2}2^s(s-1+\lceil \eta/2 \rceil) + 2( \lceil (\eta+1)/2\rceil -  \lceil \eta/2 \rceil)$$
  where the final term displayed, whose value is 2 or 0, depending on whether $\eta$ is even or odd, 
  is included to account for the difference in the formula for $\alpha_\eta(s)$
  when $s = 1$.  That term agrees with $1+(-1)^\eta$.  The formula for $B(n_1,n_2)$ is obtained by
 omitting all terms involving $\eta$ in the display above.   
 
  A routine calculation yields that  $B(n_1, n_2) = 2^{n_2}(n_1 + 2n_2-4) + 4$ and that
  $$\etB(n_1, n_2) = 2^{n_2}(n_1 + 2n_2-4) +4  + \lceil \eta/2 \rceil (2^{n_2+1}-1)  + 1 + (-1)^\eta =$$
 $$ 2^{n_2}(n_1+2n_2 +2\lceil \eta/2 \rceil -4) -\lceil \eta/2 \rceil + 5 + (-1)^\eta.$$
 Note that $\etB(n_1, \, n_2) \geq  B(n_1, n_2) \geq  n_1 + n_2$ (one gets equality if $n_2 = 0$).  

We now return to the issue left hanging and show by induction on $n_2$ that $\etB(n_1,\, n_2)$ has 
the required property.   If the image of the space $V$ is $\alpha_\eta(n_2)$-\gu\ mod the ideal $(\vect x {n_1})R$, 
then $\vect x {n_1}, \, \vect F {n_2}$ is a regular sequence of length $n_1 + n_2$
with the required property.  Since $\etB(n_1, \, n_2) \geq n_1 + n_2$,  we only 
need to consider the case where  the image of $V$ is not $\alpha(n_2)$-\gu.  
Then some nonzero $K$-linear combination $F$ of the $F_j$ has a $k$-collapse
in the quotient for some $k \leq \alpha(n_2)$, which leads to an equation
$$F = \sum_{s=1}^{n_1} x_s y_s + \sum_{t=1}^k L_t L_t'$$
where the $y_s$ and the $L_t, \, L'_t$ are linear forms or zero.
Hence, $F$ is in the $K$-algebra generated by at most $2n_1 + 2\alpha_\eta(n_2)$ linear forms
which include $\vect x {n_1}$ among them.  Let $W$ be the $K$-vector space spanned
by these linear forms, and $n_2-1$ elements that together with $F$ form a basis for
the component of $V$ of degree 2.  Then $W$ has a dimension sequence that is
at worst $(2n_1 + 2\al_\eta(n_2),\, n_2-1)$, and $K[V] \inc K[W]$.  By the induction hypothesis
we know that $K[W]$ is contained in a $K$-algebra generated by a regular sequence
consisting of at most $\etB(2n_1 + 2\al_\eta(n_2),\, n_2-1)$ linear and quadratic forms.  This completes the proof
that $\etB(n_1,\, n_2)$ defined recursively as above has the required property.  Hence:

\begin{theorem}\label{B2} Let $K$ be an algebraically closed field, and $R = K[\vect xN]$. 
Then we may take $B(n_1, n_2) = 2^{n_2}(n_1+  2n_2-4) + 4$, and we may take
$\etB(n_1,n_2) = 2^{n_2}(n_1+2n_2 +2\lceil \eta/2 \rceil -4) -\lceil \eta/2 \rceil + 5 + (-1)^\eta$.  \smallskip

Hence, if we have $n_1$ linear
forms and $n_2$ quadratic forms, they are contained in an algebra generated over $K$ by a regular
sequence of linear and quadratic forms of length at most $2^{n_2}(n_1+  2n_2-4) + 4$.  If $I$ is an
ideal generated by $n$ quadratic forms, the projective dimension of $R/I$ over $R$ is at most 
$2^{n+1}(n-2)+4$. 
 \qed\end{theorem}
 
\section{Key functions and recursions for $\etuA(\delta)$}\label{keyf}

The key functions $\fK_i$, $i \geq 3$, 
were defined in Definition~\ref{key} of \S\ref{intro}. Our main result on obtaining the functions 
$\etA$ from the functions $\fK_i$ is stated in Theorem~\ref{SJ} of \S\ref{intro}, and provides obvious motivation for proving the existence of key functions in  such a way that one has explicit formulas for them.  While \cite{AH2} proves  that these functions exist in general, the methods are not constructive and do not yield explicit bounds.  This section gives the proof of Theorem~\ref{SJ}.  The rest of this paper is devoted to giving explicit formulas for the key functions
$\fK_i$ when $i=3$ and when $i=4$ and the characteristic is not 2 or 3.

In order to prove Theorem~\ref{SJ}, we need several preliminary theorems. 
The first two are Theorems 2.1 and 2.4 of \cite{AH2}, but the latter is given in a generalized form. 

 \begin{theorem}\label{codimsing} Let $K$ be an algebraically closed field, let $R =K[\vect xN]$ be a polynomial
 ring.  Let 
 $V = V_1 \oplus\,\cdots\, \oplus V_d$, where $V_i$ is spanned by forms of degree $i$, and suppose that
 $V$ has finite dimension $n$.  Assume that a homogeneous basis for $V$ is a regular sequence in $R$.  
 Let $X$ be defined by the vanishing of all the elements of $V$.  
 Let $S$ be the family of all subsets of $V$ consisting of nonzero forms with mutually distinct degrees, so that
 the number of elements in any member of $S$ is at most the number of nonzero $V_i$. 
 For $\sigma \in S$, let $C_\sigma$ be the codimension of the singular locus of $V(\sigma)$
 in $\A^N_K$.  
 Then the codimension in $\A^N_K$ of the singular locus of $X$  is at least $(\min_{\sigma \in S} C_\sigma) - (n-1)$.\end{theorem}
 
We  will use the improved version of Theorem 2.4 of \cite{AH2} given below.  The result in \cite{AH2} is the case where $I = 0$.

\begin{theorem}\label{MAX2} Let $K$ be a field, let $R$ be a polynomial ring over $K$, and let
$M$ be an $h \times N$ matrix such that for $1 \leq i \leq h$,  the $i\,$th row consists of forms
of degree $d_i \geq 0$ and the $d_i$ are mutually distinct integers.  Let $I$ be a proper homogeneous ideal
of $R$.   Suppose that for $1 \leq i \leq h$,
the height of the ideal generated by the entries of the $i\,$th row together with $I$ is at least $b$. (If the row consists
of scalars, this is to be interpreted as requiring that it be nonzero.)  Then the ideal
generated by the maximal minors of the matrix together with $I$ has height at least $b - h + 1$. \smallskip

If $h > N$,  one has the stronger result that the depth of the ideal generated by any one maximal minor  $\Delta$ 
of the matrix together with $I$ is at least $b - N + 1$. 
 \end{theorem}

\begin{proof}  If $h > N$ we may reduce to the case where $h = N$ by working with the matrix formed by the  $N$ rows of original matrix 
corresponding to the specific minor $\Delta$.  Hence, we may assume $h \leq N$.  (Note that when $I = 0$, if $h > N$ then, since $b \leq N$ is automatic, 
we have that if  $b - h + 1 \leq 0$ and the theorem is vacuously true.) 
Because the proof is very similar to that of Theorem 2.4 of \cite{AH2}, we only indicate where, in that proof,
an ideal needs to be replaced by its sum with $I$.  In the fifth sentence of that proof, change ``every nonsingular
row generates ..."  to ``every nonsingular row, together with $I$, generates ... ."   At the end of the sentence beginning 
``For each $i$,"  replace ``an ideal $J_i$ of height $b$."  with
``an ideal $J_i$ such that $J_i+I$ has height $b$."  The linear form chosen immediately thereafter should be
``not in any of the minimal primes of the $J_i+I.$"  At the beginning  of the second paragraph of that proof, $P$ should
be a minimal prime of the ideal generated by the maximal minors together with $I$.  In the third sentence of the
third paragraph of the proof (beginning ``But then") it should say ``the ideal generated by the maximal minors
of these rows together with $I$."  No other changes are needed. \end{proof}

\begin{definition}\label{jrk} Let $F$ be a form of positive degree in a polynomial ring $R$ over a field.  By the
$J$-{\it rank} of $F$ we mean the height of the ideal generated by $F$ and all of its partial derivatives.
\end{definition}

We make the convention that the $\rkj$ of a nonzero linear form is $+\infty$.  If $R = K[\vect x N]$,  where $K$ is an algebraically closed field, the $\rkj$ of $F$ is the codimension of the singular locus of $V(F)$ in $\A^N_K$.  We note 
that the $\rkj$ of a quadratic form is the same as its rank:  see Proposition~\ref{quadcl} 
and the preceding discussion and Proposition~\ref{onequad}.

\begin{discussion}\label{fRet} Let $e_d$ denote the dimension sequence $(0,\, ,0,\, \ldots, \, 0, 1)$, with the 1 in the
$d\,$th spot.  It will be convenient to have a special name for the functions $\etA(e_d)$. In the definition
below,  $\fJ_i(k)$ for $k \geq 2$ has the same defining property as $\etA(e_i)$ with $\eta = k-2$.  We shall see that
the existence of the functions $\fK_i$ implies the existence of the functions $\fJ_i$, and this in turn
implies the existence of all of the functions $\etA$.  \end{discussion}

\begin{definition}\label{rK} Let
$\cC$ be a set of possible characteristics.
Let $i \geq 1$.  We shall say that $\fJ_i:\N_+ \to \N_+$ is a {\it J-rank function} for $\cC$ and 
for $i$ if for every algebraically closed field $K$ of characteristic in $\cC$,  for every $N \in \N_+$ and
polynomial ring $R$ in $N$ variables over $K$, and for every form $F \in R$ of degree $i$,  if $F$ is
$\fJ_i(k)$-strong then the J-rank of $F$ is at least $k$.  \end{definition}

\begin{theorem}\label{SJrank} Let $d \geq 1$ and for $i$ with  $1 \leq i \leq d$, let $\fJ_i$ be a nondecreasing J-rank
function for the set of characteristics $\cC$.   Let $K$ be an algebraically closed field
of characteristic in $\cC$ and let $R$ be a polynomial ring in $N$ variables over $K$.
Let $\delta = (\vect nd)$ be a dimension sequence, let $n = \sum_{i=1}^d n_i$, and let $h$ be the number 
of nonzero elements among $n_2, \, \ldots, \, n_d$.  
Then we may take $\etA_1(\delta) = 0$ and $\etA_i(\delta) = \fJ_i(h-1+2(n-n_1)+\eta) + n_1$ for $2  \leq i \leq d$. 
That is, if $V \inc R$ is a graded $K$-vector space with dimension sequence $\delta$ such that
$V_i$ is $\bigl(\fJ_i(h-1+2(n-n_1)+\eta)+ n_1\bigr)$-strong for $2 \leq i \leq d$,  then $\rom{R}_\eta$ holds for 
the algebraic set defined by the vanishing of $V$ or any set of forms in $V$. \end{theorem}

\begin{proof}  If there is only one form $F$ of degree $d$ the result is clear: no condition is needed if
$d=1$, and otherwise we have $n_1 = 0$ and $n-n_1 = h = 1$, so that the condition required is
$\fJ_d(\eta +2)$.  This gives codimension $\eta +2$ for the singular locus of $F$ in $\A^N$, and so
codimension $\eta+1$ in $V(F)$.  

We next note that by passing to the polynomial ring  $R/(V_1)$, i.e., by first killing the 1-forms in $V$,
we may assume that $n_1 = 0$.  The levels of strength imposed all drop by at most $n_1$, by 
Proposition~\ref{obv}(c).

By induction on the dimension of $V$ we know that any set of linearly independent
forms in $V$ generating a subspace of smaller dimension defines an algebraic set that satisfies
$\rom{R}_\eta$.  Hence, any basis for $V$ consisting of forms is a regular sequence.  
To show that  $\rom{R}_\eta$
holds for the algebraic set $X$ defined by $V$ with dimension sequence $\vect n d$,  we need the
codimension of the singular locus to be $\eta+1$ in $X$, and hence to
be $n +\eta+1$ in the ambient space.   In the worst case, by Theorem~\ref{codimsing} 
we need the height of the ideal
of minors of the Jacobian obtained from at most one form of each degree to be bounded below by 
$n+\eta+1 + (n-1)$,  which means, by Theorem~\ref{MAX2}, that it suffices if the height of
the span of every row is at least $h-1+2n+\eta$ (there will be a maximum of $h$ rows). If $n_1 = 0$,  this means 
that one can simply take the $i\,$th entry of $\etuA(\vect n d)$ to be $\fJ_i(h-1 + 2n + \eta)$,
where $h$ is the number of nonzero elements in the dimension sequence. Note that $n = n-n_1$
in this case.   If $n_1 \not=0$,  one kills the ideal generated by the $n_1$ linear
forms:  the height of the ideal generated by the entries of the $i\,$th row cannot drop lower
than $\fJ_i(h-1+2(n-n_1) +\eta)$,  where we are now in the case where $n_1 = 0$ and
$n$ has been replaced by $n-n_1$. \end{proof}.

Note that for the case where $n_1 \not= 0$, we may proceed alternatively without killing
the linear forms.   When we reduce to studying a Jacobian matrix for one form of each degree,
we have at most $h+1$ rows, one of which is a nonzero row of scalars. By Theorem~\ref{codimsing}, 
if we guarantee that the
height of the ideal of minors is  $n + \eta + 1 +(n-1) = 2n+\eta$, this will yield the required codimension
for the singular locus in $X$.  However, the number of rows the matrix may now be $h+1$, and so
the strength condition on the forms of degree $i$ in $V$ is $\fJ_i(h+2n+\eta)$.  In practice, the
functions $\fJ_i$ grow quickly enough that the result in Theorem~\ref{SJrank} is better.  

We first recall that a subset $Y$ of the closed points $X$ in a scheme of finite type over 
an algebraically closed field $K$ is
called {\it constructible} if it is a finite union of locally closed subsets of $X$.  The image of a constructible set
under a $K$-regular map is constructible.  Constructible sets include open sets and closed sets in $X$, and
the family of constructible sets is closed under finite union, finite intersection, and complementation within $X$. 
Note that if a constructible set is dense in a variety, then one of the locally closed sets in the union must
be dense, and so the constructible set contains a dense open subset of the variety.  

\begin{theorem}\label{SaG} Let $K$ be an algebraically closed field.  Let  $R = K[\vect x N]$ be a polynomial ring
over $K$.  Let $\cM$ be an $1 \times M$ matrix whose entries are $d$-forms in $R$,  $d \geq 1$. 
Let $b, k \geq 1$ be integers. Let $\cV$ denote the $K$-span of the entries of $\cM$.  Suppose that for 
every dense open subset  $U$ of $\cV^b$, the $b$-tuples of elements of $\cV$,  some 
$K$-linear combination of the entries of some $v \in U$ with at least one nonzero coefficient has a $k$-collapse.
Then there exists a $K$-vector subspace $\Theta$ of $\cV$ 
of codimension at most $b-1$ such that every nonzero 
element of $\Theta$ has a strict $bk$-collapse. 

Hence, if there is no  vector subspace $\Theta$ of $\cV$ of codimension at most
$b-1$ such that every nonzero element of $\Theta$ has a strict $bk$-collapse,  then
there is an open dense subset $U$ of the $b$-tuples of elements of $\cV$ such that 
the entries of each element of $U$ are linearly independent and $k$-strong. 
 \end{theorem}

\begin{proof}  We may replace $\cM$ by a matrix whose entries are a basis for $\cV$, and so we assume
that the entries of $\cM$ are $K$-linearly independent and that the vector space dimension of $\cV$ is $M$.  
We may therefore assume that the nonzero entries of $\cM$ are linearly independent.  Moreover,
we may omit the entries that are zero.  We write
$\cV^b$ for the $K$-vector space of $b$-tuples of elements of $\cV$. It should not cause confusion if we also think 
of elements of $\cV^b$ as  $1 \times b$  matrices whose entries are in $\cV$.  We write $\PP(\cV^b)$ for the
corresponding projective space over $K$:  it consists of $b$-tuples of elements $\cV$,  not all 0, modulo multiplication
by scalars in $K^{\times} = K-\{0\}$.  We write $\PP(W_b)$ for the projective space associated with
the vector space $W_b$ of  $b \times 1$ matrices over $K$.  Given $v \in \cV^b$ and 
$w \in W_b$, we write $w * v$ for the unique entry of $vw$,  which is an element of $\cV$.  For a given choice of $w$,
$w *v$ is a $K$-linear combination of the entries of $v$:  all of the entries of $w$ occur as coefficients.  By varying
$w$, we get all $K$-linear combinations of entries of $v$.  Let $\mu$ denote the map $W_b \times \cV^b \to \cV$
such that $(w,v) \mapsto w*v$.

Next, consider the subset $\cZ \inc W_b \times \cV^b$ consisting of pairs $(w,v)$ such that $w*v$ has 
a strict $k$-collapse or is 0.  The set of forms of degree $d$ with a strict $k$-collapse together with 0 form a
constructible set by Proposition~\ref{colcl}(c).  $\cZ$ is the inverse image of this set under the regular morphism $\mu$, and so 
$\cZ$ is constructible in $W_b \times \cV^b$.
 Since $\cZ$ is closed under multiplying either coordinate by a nonzero scalar,
it determines a constructible set  $X \inc  \PP(W_b) \times \PP(\cV^b)$.    If the projection map  $\pi_2:X \to \PP(\cV^b)$ 
does not have dense image, the set of matrices representing points of the complement of the closure of its image   
is a dense open set $U$ in $\cV^b$  such that every nonzero linear combination of the entries of any element 
of $U$ is $k$-\gu. Therefore we may assume that the projection map has dense, constructible image in $\cV^b$, 
and so the image has nonempty interior.

Since the image of the projection map contains a dense open subset of $\PP(\cV^b)$, the dimension of $X$ 
is at least $bM-1$.  It follows that there is a point of  $\PP(W_b)$ such that the fiber of $\pi_1: X \to \PP(W_b)$  is a constructible
subset of $\PP(W_b)$  that
 has dimension at least  $(bM-1)-(b-1) = bM-b$.  Fix a nonzero $b \times 1$ matrix $w$ representing an element 
 where the fiber has dimension at least $bM - b$, and so has codimension at most $bM-1 - (bM-b) =b-1$ in $\PP(\cV^b)$.
Choose an irreducible locally closed subset $X_1$ of the fiber over $w$ of codimension at most $b-1$ in $\PP(\cV^b)$.  
Let $Y_1 = Y$ be the affine cone over $X_1$: 
$Y_1$ is also irreducible, and has codimension at most $b-1$ in the affine space $\cV^b$.   Let
$Y_s = Y + \cdots + Y$, by which we mean $\{v_1 + v_2 + \cdots + v_s: v_j \in Y, \, 1\leq j \leq s\}$.  Then $Y_s$
is a family of $1 \times b$ matrices $v$ with entries in  $\cV$  such that every
$w*v = w*v_1 + w*v_2 + \cdots + w*v_s$  has a strict $sk$-collapse or is 0. Note that $Y_s$ is a cone in the affine space
$\cV^b$, and it is an
irreducible constructible set, since it is the image of $Y^k$ under a regular morphism. Let $Z_s$ be the closure
of $Y_s$.  It is an irreducible closed set in the affine space $\cV^b$, and it is also a cone.  The image $X_s$ of $Z_s -\{0\}$ is 
 a closed irreducible set in $\PP(\cV^b)$. 
 
 The chain of closed varieties $Z_1 \inc Z_2 \inc \cdots Z_j \inc \cdots$ must have the property that for some $s$,  
 $Z_s = Z_{s+1}$,  since whenever it increases strictly the codimension drops, and so the number of increases
 cannot exceed the codimension, which is a most $b-1$.  Thus, $s \leq b$.  We claim that if $Z_s = Z_{s+1}$ then 
 $Z_j = Z_s$ for all $j \geq s$. By induction, it suffices to show that $Z_{s+1} = Z_{s+2}$.  But since
 $Z_s = Z_{s+1}$, we have that $Y_s$ is dense in $Y_{s+1}$,  and so $Y_s \times Y$ is dense in $Y_{s+1} \times Y$.  
 But then the image of $Y_s \times Y$ is dense in the image of $Y_{s+1} \times Y$ under the map  $(v',v) \mapsto
 v' + v$, which means that $Y_{s+1} = Y_s + Y$ is dense in $Y_{s+1}  + Y = Y_{s+2}$,  and this implies that
 $Z_{s+1} = Z_{s+2}$.   Hence,  $Z_s = Z_{2s}$.  Similarly, since
 $Y_s\times Y_s$  is dense in $Z_s \times Z_s$, $Y_{2s}  = Y_s + Y_s$ is dense in $Z_s + Z_s$.  Hence, 
 $Z_s + Z_s$ is contained in the closure of $Y_{2s}$, which is $Z_{2s} = Z_s$.  It follows that the cone $Z_s$ is a vector
 space, which we denote $\{w\}\times\cW$, where $\cW \inc \cV^b$, and it has codimension at most $b-1$ in $\cV^b$.  
 The points in $v \in \cW$ such that $w*v$ has a strict $bk$-collapse are dense.   But this set is also the inverse image of a closed
 set in $R_d$ under a regular map, and is therefore closed.  Hence, for every $v \in \cW$, 
$w*v$ has a strict $bk$-collapse. The map $v \mapsto w*v$ is clearly a surjection of $\cV^b$ onto
$\cV$.  Let $\Theta$ be the image of this map.  Since $\cW$ has codimension at most $b-1$ in 
$\cV^b$,  we have $\Theta$ has codimension at most $b-1$ in $\cV$.  Since every element of
$\Theta$ has a strict $bk$-collapse, the proof is complete. \end{proof}

Given a finite-dimensional vector space $V$ over a field $K$, we say that a condition holds for elements of the
vector space in general position if it holds for the elements of a dense open subset of $V$.  In the theorem
below, this is applied to the vector space of $b$-tuples of elements of $\cD F$ for a form $F$. 
   
 \begin{theorem}\label{GtoS} Let $i \geq 2$ be an integer, and let $K$ be an algebraically closed field of characteristic 
 not dividing $i$.   Let $\fK_i$ be a key function from 
 $\N_+ \to \N_+$ for degree $i$ for polynomial rings over $K$.   Then for all $b \in \N_+$,  
 if $F$ is a form of degree $i$ that  has strength at least  $\fK_i(bk) + b-1$, then any $b$-tuple of elements of $\cD F$
 in general position has entries that are linearly independent over $K$ and spans a vector space of strength 
 at least $k$.  
 \end{theorem}
 \begin{proof}  Assume that $F$ is a form of degree $i$ that 
 does not have a strict  $(\fK_i(bk) + b-1)$-collapse.
If the conclusion of the theorem fails, then by Theorem~\ref{SaG},
 there is a $K$-vector subspace $\Theta$ of $\cD F$ of codimension
 at most $b-1$  such that every element of that $K$-vector subspace has a strict $bk$-collapse.  
 We may make a linear change of variables and so assume that if $j < N - (b-1)$ then ${\partial F \over \partial x_j}
 \in \Theta$. Let $Q$ be the ideal generated the final substring of at most $b-1$ consecutive variables 
 $x_{s+1}, \, \ldots, \, x_N$
 such that all the partial derivatives with respect to variables $\vect x s$ not in $Q$ are in $\Theta$.  Since the formation 
 of partial derivatives with respect to $x_j$ not in $Q$ commutes with killing $Q$, 
 if we let  $\ov{F}$ denote the image of $F$ modulo $Q$ in $K[\vect x s]$,  where $N-s \leq b-1$,  then we have
 that every element in the span of the partial derivatives of $\ov{F}$ has a strict $bk$-collapse.  By
 the defining property of $\fK_i$,  $\ov{F}$ has a strict $\fK_i(bk)$-collapse, and is in an ideal $J$
 generated by at most $\fK_i(bk)$ forms of smaller degree.  This means that $F \in J + Q$,  where
 $Q = (x_{s+1}, \, \ldots, x_N)$ is generated by at most $b-1$ elements of degree 1, and so has an
  $(\fK_i(bk) + b-1)$-collapse, a contradiction.
\end{proof}

\begin{corollary}\label{KtoR} Let $K$ be an algebraically closed field for which we have a key function
$\fK_i$ for some $i \geq 2$, where the characteristic of $K$ does not divide $i$.  
Suppose also $A_{i-1}$ is a function such that a $K$-vector space of dimension
$n$ whose nonzero elements are forms of degree $i-1$ of strength at least $A_{i-1}(n)$ is generated by
a regular sequence.  Then we may take $\fJ_i(k):=  \fK_i\bigl(kA_{i-1}(k)\bigr) + k-1$.  \end{corollary}
\begin{proof} Let $F$ be a form of degree $i$.   
We may apply Theorem~\ref{GtoS} using $k$ and $A_{i-1}(k)$ for the values of $b$
and $k$.  It follows that a vector subspace of $\cD F$ generated by $k$ elements in general position
has strength $A_{i-1}(k)$,  and so is generated by a regular sequence. \end{proof}

We are now ready to give the proof of Theorem~\ref{SJ}.\\

\noindent\emph{Proof of Theorem~\ref{SJ}.} This is simply the result of combining Theorem~\ref{SJrank} and Corollary~\ref{KtoR}.  
The condition in the definition of $\etA_2$ simply guarantees, when $n_1 = 0$,  
that the height of $\cD F$ is at least $b$ for every form $F \in V$.  \hfil \qed \\

In the case where one simply has a vector space of dimension $n$ whose nonzero elements are all forms
of degree $d$, this takes a simpler form.  
  
 Theorem~\ref{SJ} shows that for a class $\cC$ of characteristics and degree at most $d$,  
 if one has explicit key functions $\fK_i$ for the characteristics in $\cC$ and $i \leq d$, then
 one obtains specific $\etA$ for degree sequences bounded by $d$.  The results of \S~\ref{BfromA}
 then yield the functions $\etB$ up to degree $d$ (i.e., the existence of small subalgebras) and, hence, 
we obtain explicit bounds for projective dimension, proving Stillman's conjecture in degree $d$, but
in a constructive way.  Moreover, these bounds are obtained for a situation more general than that
of homogeneous ideals, as shown in \S~\ref{BfromA}:  one gets them for modules such that one has bounds
on the size of the presentation, and without homogeneity assumptions.             
  
  In \S\ref{cubic}, we obtain our main results on cubics using both this result and other ideas.
  The cases of characteristics 2 and 3 
 require larger choices of $\etuA$. The difficulties are handled by using choices of $\fJ_3$ arising form variant 
 methods and applying Theorem~\ref{SJrank}.   In \S\S\ref{sum}-\ref{hesscol} we construct the functions $\fK_4$ for 
 characteristic not 2, 3,  which enables us to give explicit functions $\etA$, $\etB$, and $C(r,s,d)$ for degree 
 $d \leq 4$ if the characteristic is not 2 or 3.  The argument needed for $d=4$ is very intricate.

\section{The cubic case: $\fK_3$, $\fJ_3$, and $\etuA(n_1,n_2,n_3)$}\label{cubic}

The main result of this section is Theorem~\ref{A3b} below.  We first note:

 \begin{theorem}\label{G3} For any algebraically closed field of characteristic $\not= 2,\,3$ we may
 take $\fK_3(k) = 2k$.  \smallskip
 Hence, if $K$ is a field of characteristic $\not= 2,\,3$,  we may take $\fJ_3(k) = (2k+1)(k-1)$. \end{theorem}
 \begin{proof} The first statement is immediate from Theorem~\ref{2tr}, and the second follows at once
 from Corollary~\ref{KtoR} and the fact that by the last statement in Theorem~\ref{ngu}, we may 
 choose $A_2(k) = k-1$. \end{proof}
  
 Hence, by Theorem~\ref{GtoS}:
 
\begin{theorem}\label{K2} For any algebraically closed field $K$ of characteristic $\not=2,3$, if $b$ is a positive integer 
and $F$ is a form of degree 3 in a polynomial ring over $K$ such that $F$  has strength at least $2bk + b-1$,
then any $b$-tuple of elements of $\cD F$ in general position has entries that are linearly independent over $K$ and 
span a vector space of strength at least $k$.   \end{theorem}
 
 We next want to show that one can construct a function of $\fJ_3$ for cubics when $K$ is an algebraically
closed field of characteristic 2 or of characteristic 3. 

\begin{theorem}\label{R323} Let $R$ be a polynomial ring over an algebraically closed field.
If $K$ has characteristic 2, then we may take $\fJ_3(k) = 2(k-1)(2k+1)$.  
If $K$ has characteristic 3, we may take $\fJ_3(k) =  2k^2-k$.
\end{theorem} 

\begin{proof} Let $k$ be given and let $F$ be a cubic form that is $\fJ_3(k)$-strong according to the
appropriate formula above.  
We shall show that the J-rank of $F$ is at least $k$. By the last statement in Theorem~\ref{ngu}, 
it will suffice if $\cD F$ contains a $k$-tuple of linearly independent quadrics over $K$ spanning  a $K$-vector
space of strength at least $k-1$. If this is not true, let $b = k$ 
and note that by the last sentence of Theorem~\ref{SaG}, 
there is a vector space $\Theta$ of codimension at most $b-1 = k-1$ in $\cD F$ such that every element has
a $k(k-1)$-collapse.   Choose an element  $G \in \Theta$ such that the 
rank  $r$ of $G$ is maximum.  Since $G$ has a $k(k-1)$-collapse, $r \leq 2k(k-1)$. 

In the characteristic 2 case, let $r = 2h+\epsilon$, where $\epsilon$ is 0 or 1, and $2h$ is the dimension of
$\cD G$. We then have that $h \leq k(k-1)$.  The ideal $(\cD G)R$ is generated by $2h$ variables.
By Corollary~\ref{cDsq},  the image of $\Theta$ mod the ideal $(\cD G)R$, is a vector space
consisting of squares of linear forms and 0:  otherwise, a linear combination of $G$ and another element of $\Theta$
will have larger rank.  Let $q$ be the dimension of the space $W$ spanned by these linear forms.  
Note that if $q - 2h \geq k$,  then the J-rank of $F$ must be at least $k$, since the image of $\cD F$
modulo $(\cD G)$ will contain the image of the squares of the linear forms in $W$. Therefore, we may assume that
$q \leq 2h+k-1 \leq 2k(k-1) + k-1$.  Then $\Theta$ is in the ideal spanned by the $2h$ variables in $\cD G$
and the $q$ linear forms spanning $W$.  It follows that $\cD F$ is in the ideal generated by
these $2h+q$ elements, along with at most $k-1$ elements of $\cD F$:  as many as needed to span $\cD F/\Theta$.
Since we may use Euler's formula (the characteristic is 2, and the degree of $F$ is 3),  we have
that $F$ is in the ideal generated by $2h+q + k-1 \leq 2k(k-1)+ 2k(k-1)+k-1 + k-1 = 2(k-1)(2k+1)$
elements of lower degree, a contradiction. 

In the characteristic 3 case, every element of $\Theta$ is 
in the ideal $J$ generated by $\cD G$, by  Corollary~\ref{cDsq}.  Choose a $K$-basis $\vect y N$ for $R_1$ 
such that $\vect y r$ is a basis for $\cD G$:  since the characteristic of $K$ is not 2,  the rank of $G$ is the
same as the $K$-vector space dimension of $\cD G$.
Now consider the image $\ov{F}$ of $F$ modulo $J$ in the polynomial ring $R/J \cong K[y_{r+1}, \, \ldots, y_N]$.
If $j > r$, then $\partial \ov{F}/\partial y_j$ is simply the image in $R/J$ of $\partial F/\partial y_j$.  Hence,
$\cD \ov{F}$ is contained in the image of $\cD F$ mod $J$. This is spanned by the images of
at most $k-1$ elements of $\cD F$, since $\Theta \inc J$.  Thus $\cD \ov{F}$ has $K$-vector space
dimension at most $k-1$ over $K$.  

We may take a new basis $\vect z {N-r}$ for the forms of degree 1 in $R/J$ such that $\partial\ov{F}/\partial z_h$
are 0 for $h > k-1$.   This implies that every term of $\ov{F}$ that is not a scalar times the cube of a variable is in 
$(\vect z {k-1})R$:   if $cz_iz_jz_t$  or $cz_iz_j^2$ occurs in $\ov{F}$ for $i,j,t$  (respectively, $i,j$) distinct and with 
$c \in K-\{0\}$, then $\partial F/\partial z_i$ has a $cz_jz_t$ term (respectively, a $cz_j^2$ term), and so $z_j$ or 
$z_t$ (respectively, $z_j$) must be in $J$.  The sum of the terms of $\ov{F}$ involving cubes of variables may be 
written as the cube of a single linear form $L$, and this implies that $\ov{F}$ is in the ideal generated by the 
$k$ elements $\vect z{k-1}$ and $L$.  It follows that $F$ is in the ideal generated by $r + k$ linear forms, and 
$r \leq 2(k-1)k$, which yields that $F$ has a $\bigl(2(k-1)k + k\bigr)$-collapse, a contradiction. \end{proof}
 
 We can now carry through the calculation of $\etuA$ in case $d \leq 3$.  

First, we assume that $n_1 =0$. Thus, $\delta = (0, n_2, n_3)$,  and the case where $n_3 = 0$ has
already been handled by Corollary~\ref{A2}.

By the last statement in Theorem~\ref{ngu}, we may take $A_2(0,n_2) = n_2-1$ and hence $A_2(n_1,n_2) =
n_1+n_2-1$.  

By Theorem~\ref{SJrank}, we have at once:

\begin{theorem}\label{A3b} Let $K$ be an algebraically closed field.
Let $b = 2(n_2+n_3) + \eta +1$ if $n_2 \not=0$,  and $2(n_2+n_3) + \eta$ if $n_2 = 0$. 
Then we may take
$$\etuA(0,n_2,n_3) = \Bigl(0, \lceil {b \over 2}\rceil, \fJ_3(b)\Bigr)$$
where $\fJ_3(b) = (2b+1)(b-1)$ if the $\ch(K)$ is not 2 or 3,   $\fJ_3(b) = 2(2b+1)(b-1)$ if
$\ch(K)$ is 2,  and $\fJ_3(b) = 2b^2-b$ if $\ch(K)$ is 3.   
More generally, 
$$\etuA(n_1,n_2,n_3) = \Bigl(0, \lceil {b\over 2}\rceil + n_1, \fJ_3(b) +n_1\Bigr).$$ \qed
\end{theorem}

For application to the degree four case  in \S\ref{4case} we need 
the following:

\begin{corollary}\label{explA3} Let $A_3(n) := \fJ_3(2n-1)$, which is $2(4n-1)(n-1)$ if the characteristic is 
different from 2, 3.
Then a sequence of $n$ linearly independent cubic forms in a polynomial ring $R$ over an algebraically closed 
field such that the vector space $V$ spanned by the forms is $A_3(n)$-strong is a regular sequence. \end{corollary}
\begin{proof} As in Remark~\ref{stregh},  it suffices if any $n-2$ or fewer of the forms generate an ideal whose 
quotient ring is a UFD, and so it suffices if the strength of $V$ is at least $\etA(0,0,n-2)$ with $\eta = 3$.  In this case,  
$b = 2(n-2) + 3 = 2n-1$. \end{proof}

Note that in all cases,  $\fJ_3(b)$ is quadratic in $b$ with leading coefficient 2 or 4.

By the results of \S\ref{BfromA}  we obtain as well explicit choices
of $\etB$ and $C$.  However, the construction of $\etB$ is recursive,
and the discussion below shows that one gets extraordinarily large values from these techniques.
It remains of great interest to find better bounds for the functions $\etB$.  So far as we know,
it is even possible that one can use a polynomial in $n$ of degree $d$ to bound the projective
dimension of an ideal generated by $n$ forms of degree at most $d$.  If $d=2$, a specific bound that is
quadratic in $n$ is presented in Question 6.2 of \cite{HMMS1}, but it is open question whether this
bound is correct for the case of $n$ quadrics.

\begin{discussion} {\bf The size of $\bm{B(n_1,n_2,n_3)}$.}  The behavior of the value of $B(0,0,n)$ 
obtained from the method of \S\ref{BfromA} is complicated, and the values are very large.  The roles
played by the linear and quadratic forms that arise in the process of repeated application of the method
have a heavy impact. Assume the simplest case, which is that the characteristic is not 2 or 3.  
What can happen is that, for a certain constant positive integer $a$,  one cubic form may collapse, producing
roughly $an^2$ linear forms and $an^2$ quadratic forms.    What may happen next is that all of the quadratic
forms collapse.  This may produce roughly $\cN = an^2 2^{an^2}$ linear forms.  To get another cubic to collapse 
(mod the linear forms one already has) the number $\cN$ of linear forms must be added into the level of strength needed
to get one more cubic to collapse.  That means that getting one more cubic form to collapse will result in at
least  $\cN$ new linear forms and $\cN$ new quadratic forms.  If all  the quadratic forms collapse before another
cubic collapses,  one obtains $\cN2^\cN$ new linear forms, and this number must be added to the level of strength
one needs for the cubics to guarantee a regular sequence.  It should be clear that continuing
in this way, where at each stage all the quadratic forms introduced collapse  before the next 
cubic does, will produce, so far as one can tell {\it a priori}, something worse than $n$-tuple exponential 
behavior, since $n$ repetitions are needed before all the cubics have collapsed. 
Because we believe that these results are far, far larger than needed to bound projective dimension, we have not made 
detailed estimates.
\end{discussion} 

 \section{Sum decompositions of symmetric matrices}\label{sum}
 
 \subsection*{Overview of the argument for quartics}
 
  The results of this and the next three sections will handle the proof of the existence of 
 $\etuA(n_1,n_2, n_3, n_4)$.  The problem we confront is the construction of $\fK_4$.  
 Roughly speaking, we must show  that if $F$ is quartic and every element of $\cD F$ has 
 a strict $k$-collapse, then $F$ itself has a strict $k'$-collapse, where $k'$ is ``small" in the 
 sense that it may depend on $k$ but not on $N$.  To prove
 this we study the Hessian $\fH$ of $F$.  We let $\GL(N,K)$ act on the variables to put
 them in general position. 
 
 Given a matrix $\cM$ with entries in a vector space over a field $K$, we refer to a $K$-linear combination
 of the columns of $\cM$ as an  \LCK-column of $\cM$.  The assumption that every element of $\cD F$
 has a $k$-collapse implies that every \LCK-column of the Hessian $\big(\partial^2F/\partial X_i \partial X_j\bigr)$
 has a collective $(k,\,k)$-collapse (cf. see the last three paragraphs of Definition~\ref{defsafety}), i.e.,
 all of the entries of the \LCK-column are in the sum of an auxiliary ideal generated by at most $k$ homogeneous polynomials
 of  degree 1 and the auxiliary vector space spanned by at most $k$ quadrics.  The reason is that every
 \LCK-column consists of the partial derivatives of an appropriate linear combination $G$ of the $\partial F/\partial X_i$.
 Since $G \in \cD F$,  $G$ has a $k$-collapse, and we can write $G = \sum_{t=1}^k L_t Q_t$, where the $L_t$
 are linear forms 
 and the $Q_t$ are quadrics.  The product rule implies that every $\partial G/\partial X_j$ is in the $K$-vector space 
 sum of the ideal generated by the $L_t$ and the $K$-vector space spanned by the $Q_t$. 

 We want to show that under this $(k,k)$-collapse condition on the \LCK-columns of the Hessian $\fH$,
 we can write $\fH = \fH_1 + \fH_2$,  where these are symmetric matrices
 of quadrics such that (1) all of the \LCK-columns of $\fH_1$ have entries with ``small" 
 rank (independent of the column)  and (2) the entries of
 $\fH_2$ span a vector space $W$ of ``small" dimension.  This means that we will show, in effect,
 that a single vector space $W$ of ``small" dimension can be used for the collective $(k',h')$-collapse
 of all of the columns of $\fH$:  $k'$, $h'$ are typically larger than the original $k$ and $h$ (which were $k$ and
 $k$) but are
 given by functions of them independent of the number of variables generating the polynomial ring
 and also of the base field.  Achieving condition (2) is studied in the next section. 
 After that, in  Section~\ref{hesscol}, we study symmetric matrices $\fH_1$ of quadrics in which  every entry of every \LCK-column has small rank, independent of the column.  
 We prove that for such a matrix, $G = (G_{ij})$, $\sum_{ij} x_ix_j G_{ij}$
 has a small collapse.  We then use the fact that an integer multiple (a restriction on the characteristic
 is needed to make certain the integer is nonzero in the base field) of $F$ is congruent to $G$ modulo an ideal
 with a small number of generators (a basis for $W$) to conclude that $F$ has a small strict collapse.  All of this is
 made precise in the sequel.
 
 Some of the needed results are valid in higher degree cases, and we prove them in that generality.
 
 It will be convenient to use general position arguments in which a certain finite family of elements is taken to 
 be algebraically independent over $K$.  For this purpose we shall sometimes pass to a larger
 algebraically closed extension field obtained from $K$ by adjoining finitely many indeterminates
 and taking an algebraic closure. In many cases this is essentially equivalent to working 
 on a Zariski dense open set of a variety parametrizing the family of elements.  The following
 result gives such an equivalence. 
 
 \begin{theorem}\label{genkh}    Let $\fH$ be an $N \times N$ 
 matrix of $d$-forms in a polynomial ring $R$ over an algebraically closed field $K$.  Let $B$ denote a varying
 element in GL$(N,K)$ and 
 let $v$ denote a varying $N \times 1$ column over $K$. 
 \begin{enumerate}[(a)] 
\item  For a given $v$, $\fH v$ has a collective $(k,h)$-collapse if and only if $B\fH v$ has a collective $(k,h)$-collapse
for some $B \in \GL(N,K)$
 if and only if $B\fH v$ has a
collective $(k,h)$-collapse
 for all $B \in \GL(N,K)$.  
 \item The following conditions are equivalent.
\begin{enumerate}[(1)] 
\item For a dense open set $U$ of $N \times 1$ matrices over $K$,  $\fH v$ has a collective $(k,h)$-collapse 
for all $v \in U$.
\item For a dense open set $U$ of $N \times 1$ matrices over $K$ and all $B \in \GL(N,K)$,  
$B\fH v$ has a collective $(k,h)$-collapse for all $v \in U$.
\item For a dense open set  $\cU$ of $\GL(N,K)$,  every column of $\fH A$ has a collective $(k,h)$-collapse for all $A \in \cU$.
\item For a dense open set $\cU$ of $N \times N$ matrices over $K$ and all $B \in \GL(N,K)$, every column of  $B\fH A$ 
has a collective $(k,h)$-collapse for all $A \in \cU$.
\item For a dense open set $\cU$ of $\GL(N,K)$, every column of  $A\tr \fH A$ has collective $(k,h)$-collapse for all $A \in \cU$.
\item Let $\vect t N$ be indeterminates over $K$ and let $L$ denote the algebraic closure of $K(\vect t N)$.
Let $T$ be the $N \times 1$ column matrix  $ \bigl(t_1\,\, \ldots \,\, t_N\bigr)\tr$. 
Then $\fH T$ has a collective $(k,h)$-collapse over $L$.  
Moreover, the coefficients of all elements occurring in this collective collapse may be taken in 
$D$, where $D$ is a module-finite extension of $K[\vect t N]_g$ and $g \in K[\vect t N] -\{0\}$. 
\end{enumerate} 
\end{enumerate}
 \end{theorem}

 \begin{proof} (a) is immediate from the observation that the entries of an $N \times 1$ column $w$ have the same
 $K$-span as those of $Bw$ for all $B \in \GL(N,K)$.  
 Part (a) makes it immediately clear that (1) $\dbimp$ (2).
 
(2) $\imp$ (3)  because $U \times \cdots \times U$  ($N$ copies) with elements  thought of as $N \times N$ 
matrices over $K$
 (we may think of the latter as $N$-tuples of columns) is a Zariski dense open set of the $N \times N$ matrices
 over $K$,  and its intersection with $\GL(N,K)$
 is therefore a dense open subset of $\GL(N,K)$.  For the converse, consider a dense open set $\cU$ of
 $\GL(N,K)$, again thought of as $N$-tuples of columns.  The projection map from $\cU$ to, say, its first column
 must have dense image in the space $V$ of $N \times 1$ columns:  otherwise, if the closure has dimension
 $N' < N$,  the dimension of $\GL(N,K)$ would be at most $N' + (N-1)N < N^2$.  The image therefore contains
 a dense open subset $U$ of $V$. Thus, (1) and (3) are equivalent. The equivalence of (2), (4), and (5) is then
 immediate from part (a). 
 
 Let $\fH = \bigl(m_{ij}\bigr)$.  
 The statement in $(1)$ is equivalent to an equational statement formulated below.  
 In the sequel, quantification on $i$, $\mu$ and $\nu$ is 
 as follows: $1 \leq i \leq N$, $1 \leq \mu \leq k$, and $1 \leq \nu \leq h$.
 Let $\ud = \vect d k$  be  a sequence of $k$ positive integers such
 that for every $\mu$, $1 \leq d_{\mu} < d$.  For every $\mu$,  let $H_{\mu}$ be a form
 of degree $d_\mu$ in $\vect x N$ with unknown coefficients and for all $\mu, \, i$ let $H'_{\mu, i}$ be a form
 of degree $d - d_{\mu}$ in $\vect x N$ with unknown coefficients.  For all $\nu$, let
 $P_\nu$ be a form of degree $d$ in $\vect x N$ with unknown coefficients, and for all $\nu,\,i$,  let $Z_{\nu,i}$ be an
 unknown. Denote all the unknown coefficients including the $Z_{\nu,i}$ as $\vect Ys$.   
 For all $\ud$, let $\cE_{\ud}$ denote the set of coefficients of the monomials in the $x_j$ that
 occur in the elements 
  $$
  \sum_{j=1}^N  m_{ij}t_j - \Bigl(\sum_{\mu = 1}^k H_\mu H'_{i,\mu} + \sum_{\nu = 1}^h P_{\nu} Z_{\nu,i} \Bigr),
  $$
  for $1 \leq i \leq N$.  These coefficients will be in $K[\vect t N,\, \vect Y s]$.  
 Then (1) is equivalent to the following statement.  There exists $g \in K[\vect t N] - \{0\}$ such that 
 for every choice of $\ud$, if the $t_i$ are specialized
 to elements $\vect c N$ in $K$ such that $g(\vect c N) \not = 0$, 
 then at least one of the systems of equations obtained by setting all $\cE_{\ud}$ equal to 0 and
 setting $t_i = c_i$ for all $i$ has   a solution in $K$.  Here, the condition $g(\vect c N) \not=0$
 defines the open subset of linear combinations of the columns that have a collective $(k,h)$-collapse.

If we do not have a collective $(k,h)$-collapse for $\fH T$ over $L$  then, for every choice of $\ud$,
$\cE_{\ud}$ defines the empty set over $L$.   Hence, for every choice of $\ud$, the elements
$\cE_{\ud}$ generate the unit ideal in $L[\vect y s]$.  
Since  $L$ is the directed union of the rings  $D$ where  $D$ is a module-finite
extension of  a ring of the form $K[\vect t N]_g$ for some $g \in K[\vect tN]-\{0\}$, we may choose
a subring  $D$ of $L$ module-finite over $K[\vect t N]_g$ for some $g \not=0$ such that
each of the ideals  $(\cE_{\ud})D[\vect Y s]$ is the unit ideal.  Since an ideal becomes the unit ideal
after extension to a module-finite overring iff it was already the unit ideal, it follows that we can
choose  $g \in K[\vect tN] -\{0\}$ such that each of the ideals  $(\cE_{\ud})K[\vect t N]_g[\vect Ys]$ is the
unit ideal.   This is equivalent to the statement that the product of these ideals is the unit ideal, and
we can conclude thats there exists $g \in K[\vect t N]-\{0\}$ with a power in the product of 
the ideals $(\cE_{\ud})K[\vect t N][\vect Y s]$.  

Consider the dense open set in $K^N$ where $g$ does not vanish.
It meets the dense open set  $U$ given in the statement of (1).   Specialize the $t_i$ to values $c_i$ in $K$ such
that $\vect c N$ is in the intersection.
For some $\ud$ the polynomials $(\cE_{\ud})$ vanish for a choice of values for $\vect Ys$  in $K$: because the point is in $U$, 
there is a collective $(k,h)$-collapse over $K$.   But this forces $g$ to vanish, a contradiction.  

Now suppose that one has the collective $(k,h)$-collapse over $L$ and, hence, over some $D_g$ as described.  
Then for any point $v$  of $K^N$ (thought of as $N \times 1$ matrices over $K$) where $g$ does not vanish,  there is a maximal ideal  $\fm$ of $D_g$ lying over the maximal ideal 
of $K[\vect tN]$ corresponding to the  point $v$  (since 
$D_g$ is a module-finite extension of $K[\vect t N]_g$).  The quotient of $D_g$ by $\fm$ is $\cong K$,  and the
specialization $D_g \surj D_g/m \cong K$ yields a collective $(k,h)$-collapse of $\fH v$.
 \end{proof}
  
 \begin{discussion}\label{gensetup} Let $\fH = (m_{ij})$ be an $N \times N$ 
 matrix of forms of degree $d$ in a polynomial ring $R$ over an algebraically closed field $K$. Let $\vect t N$ be indeterminates
 over $K$ and $T$ be the column matrix $\bigl(\vect t N\bigr)\tr$.  Let $L$ be an algebraic closure of $K(\vect t N)$.
 We refer to a collective $(k,h)$-collapse for $\fH T$ over $L\otimes_KR$ as a {\it generic} collective $(k,h)$-collapse
 for the columns of $\fH$. 
 Suppose that
 $$(*) \quad\sum_{j=1}^N m_{ij}t_j = \sum_{\mu=1}^k H_{\mu}H'_{\mu,i} + \sum_{\nu=1}^h P_{\nu}z_{\nu,i}, \ 1 \leq i \leq N$$
displays the collapse.  Here, every  $H_{\mu}$ has degree $d_i$, $1 \leq d_i < d$, every $H'_{\mu,i}$ has 
degree $d-d_i$, every $P_{\nu}$, $1 \leq \nu \leq h$  is a $d$-form and the $z_{\nu,i}$ are scalars. 

 We may choose $D \inc L$ where $D$ is module-finite over $K[\vect t N]_g$ for 
 $g \in K[\vect t N]-\{0\}$,  so that $D$ contains all the coefficients occurring in $(*)$.  
If $h$ cannot be decreased, which we typically will be able to assume, 
the $h$-rowed matrix consisting of all coefficients of the
$P_{\nu}$ has rank $h$,  and some $h \times h$ minor (the entries will be in $D$) will be a nonzero element 
$\beta \in D$.  By enlarging $D$, if necessary,
we may assume that some minor $\beta$ is a unit in $D$.  The $P_\nu$ form a basis for an $h$-dimensional
vector space over $L$ that we denote $W_L$.  We shall write $W_D$ for the free $D$-module spanned
by the $P_\nu$.  Let $\theta:D \to \Omega$ be a $K$-linear homomorphism from $D$ to an algebraically
closed field $\Omega$ containing $K$.  By applying $\theta$ we obtain a collective $(k,h)$-collapse
for $\fH \theta(T)$ over $\Omega\otimes_KR = \Omega[\vect x N]$.  We use images of the $k$ elements 
of $D[\vect x N]$ of degree strictly smaller than $d$ and the images of the $P_\nu$:  the latter span a 
vector space of $d$-forms over $\Omega$
that we denote $W_\theta$.  This vector space may be identified with $\Omega \otimes_D W_D$.  

 A major case will be that where $\theta:D \surj K$ is simply obtained by 
killing a maximal ideal of $D$.  Note that if we have $\theta_0:K[\vect t N] \to \Omega$ such
that $\theta_0(g) \not=0$,   where $D$ is module-finite over $K[\vect t N]_g$,  then $\theta_0$ extends to
$K[\vect t N]_g$ and, hence, to $D$.  

By killing a varying maximal ideal $\fm$ of $D$,  we obtain collective $(k,h)$-collapses 
for various \LCK-columns $\fH v$ of $\fH$,  where $v$ is a point of $K^N$ with $g(v) \not=0$.
Consistent with the conventions above, we denote the $h$-dimensional $K$-vector space of forms 
spanned by the images of  $P_{\nu}$ as $W_\theta$,  where $\theta: D \to D/\fm \cong K$.   The next result
analyzes when one has an element with a strict $s$-collapse in $W_L$.  \end{discussion} 
 
 \begin{theorem}\label{gencol} Let notation and terminology be as in Discussion~\ref{gensetup},
 so that we have fixed a  generic $(k,h)$-collapse for $\fH$ over $L$,  and suppose that
 it is defined over $D \inc L$ where $D$ is module-finite of $K[\vect t N]_g$,  with $g \in K[\vect tN]-\{0\}$. 
 Fix a positive integer $s$.  
 \begin{enumerate}[(a)]
\item  The following two conditions are equivalent
   \begin{enumerate}[(1)]
 \item There is a nonzero form of $W_L$  with a strict $s$-collapse in $L\otimes_KR$.
 \item For some larger choice of $D' \supseteq D$, and every specialization
 $\theta:D' \to K$,   $W_\theta$ has a nonzero form
 with a strict $s$-collapse.  The same holds for every larger choice $D''$ of $D'$.
 \end{enumerate}
 
\item The following two conditions are also equivalent.
 
 \begin{enumerate}[(1)] 
   \item There is no nonzero form of $W_L$  with a strict $s$-collapse in $L\otimes_KR$:  that is,
  $W_L$ is $s$-strong.
  \item For some larger choice of $D' \supseteq D$ and every specialization $\theta:D' \to K$,
  $W_{\theta}$ has no nonzero form with a strict $s$-collapse, i.e., $W_\theta$ is $s$-strong.  The same
  holds for every larger choice $D''$ of $D'$.  
 \end{enumerate}
 \end{enumerate}
 \end{theorem}
 \begin{proof}  If the first condition in part (a) holds we can choose $D'$ large enough to contain elements
 that are a basis for $W_L$ over $L$ such that these elements remain linearly independent over $K$ for any
 specialization $L' \surj K$ (by enlarging $D'$ to contain the inverse of a suitable minor), to contain the coefficients
 of a nonzero element that has a strict $s$-collapse, the inverse of at least one of its nonzero coefficients,
 as well as the coefficients of the polynomials needed to exhibit the collapse.  Then for every specialization
 $\theta: D' \surj K$, the image of the element with the collapse is nonzero and has a strict $s$-collapse. This shows
 $(1) \imp (2)$ in part (a).   The statement about larger choices $D''$ is then obvious, since any specialization
 $D'' \surj K$ restricts to a specialization $D' \surj K$.
   
 Now assume that no nonzero element of $W_L$ has a strict $s$-collapse.   Again choose $D'$ sufficiently
 large to contain the coefficients of a basis for $W_L$ over $L$ such that these elements remain linearly independent 
 over $K$ for any specialization $D' \to K$.  Call this basis $\vect \beta h$.  Let $\vect z h$ be new indeterminates.
 Let $\vect P s$, 
 $P'_1, \, \ldots, \, P'_s \in L[\vect x N]$ be polynomials in $\vect x N$  with new unknown coefficients $y_j$ 
 such that $P_i$ has degree $d_i$ and $P_i'$ has degree $d-i$.  The fact that only the zero element
 in $W_L$ has an $s$-collapse means that for every choice of $\vect d s$ and every choice of 
 specialization of $\vect zh$, $\vect yj$ to values in $L$, 
 if the unique entry $\Delta$ of $\bigl(z_1\,\, \cdots  \,\,z_h)\fH T$  is 0, then
 all of the $z_\nu$ have value 0. Let $\gamma_\lambda$ denote the coefficients of $\Delta$ when they are thought
 of as polynomials in $\vect x N$.  The $\gamma_\lambda$ are polynomials over $L$ in the $z_\nu$ and $y_j$.
 (Note that $T$ has entries in $L$.)  By Hilbert's Nullstellensatz, we then have that for every choice of $\vect d s$,
 every $z_\nu$ has a power in the ideal generated by the $\gamma_\lambda$ in $L[z_\nu, y_j: \nu, j]$.  
 The same holds over $D'$ for a sufficiently large choice of $D'$:  $D'$ can be chosen so large as to work for every
 choice of $\vect d s$.   For this or any larger
 choice of $D'$,  if one has a specialization $\theta:D' \surj K$,  there is no nonzero element with a strict $s$-collapse
 in $W_\theta$.    This shows $(1) \imp (2)$ in part (b).  
 
 Since (a) part (1) and (b) part (1) give a mutually exclusive exhaustion of the possibilities, we have that in
 every instance either the statement in (a) part (2) or the statement in (b) part (2) holds, these are obviously
 mutually exclusive.  It follows that in each part, statements (1) and (2) are equivalent.  
 \end{proof}
 
When one has a collective $(k,h)$-collapse for a set of forms of degree $d$,  the set of $k$ forms of degree
strictly smaller than $d$ generates an ideal $\fA$: recall that is called the {\it auxiliary} ideal of the collapse, while the set of 
$h$ forms of degree $d$ span a vector space of dimension at most $h$ that we call the {\it auxiliary} vector 
space of the collapse.  The next result will be useful in modifying a collective $(k,h)$-collapse so that the representation of an element
as a sum of an element in the auxiliary ideal and another in the auxiliary vector space is unique.

\begin{theorem}\label{mvclpse}  Let $a,k,h \in \N$ with $a,k >0$ and let $S$ be a set of $d$-forms over the 
polynomial ring $R$  with a collective $(k,\,h)$-collapse with auxiliary ideal $\fA$ and auxiliary vector space $V$.   
 Define a sequence $(k_i, b_i)$ by the rule $k_0 = k$, $b_0 = 0$,  $b_{i+1} = ak_i$,  $k_{i+1} = k_i + b_{i+1}$. 
 By a straightforward induction, $b_i = (a+1)^{i-1}ak$ for $ i \geq 1$ and $k_i = (a+1)^i k$, $i \geq 0$. 
 Choose a maximal sequence of linearly independent elements $\vect v m \in V$ such that $v_i$ has a strict
 $b_i$-collapse.  Necessarily,  $m \leq h$.  Write $V = V_1 \oplus W$,  where $V_1$
 is the span of $\vect v m$.  Because $v_i$ has a $b_i$-collapse for $1 \leq i \leq m$, for each such $i$ we
 can choose an ideal  $\fJ_i$ with at most
 $b_i$ homogeneous generators of positive degree smaller than the degree of $v_i$ such that 
 $v_i$ is in $\fJ_i$.   Then $S$ has a collective $(k', h-m)$-collapse with 
 $k' = k_m = (a+1)^{m}k$. The auxiliary ideal is $\fA + \sum_{i=1}^m\fJ_i$ and the auxiliary vector space
 is $W$.  \end{theorem}
 
\begin{proof}  If no element of $V-\{0\}$ has a strict $ak$-collapse,  then the 
 sequence of $v_i$ is empty,  $m=0$,  and $W = V$.  If some $v_1 \not=0$ has an $ak$-collapse,
 then $S$ has a collective $(k + ak, \, m-1)$ collapse, where the auxiliary ideal is the sum of 
 $\fA + \fJ_1$, which has at most $k+ak$-generators, and the auxiliary vector space is the span of
 $v_2, \ldots, \, v_m$ and $W$.  By a straightforward induction on $i$,  $S$ has a collective $(k_i, m - i)$
 collapse in which the auxiliary ideal is 
 $\fA + \sum_{\nu=1}^i \fJ_\nu$ 
  and the auxiliary vector space is
 the span of $v_{i+1}, \ldots, v_m$ and $W$:   one writes each element of $S$ as 
 the sum of an element  
 $F \in \fA + \sum_{\nu=1}^i \fJ_\nu$ 
 and a sum  $c_{i+1} v_{i+1} + c_{i+2}v_{i+2} + \cdots
 + c_mv_m + w$,  where the $c_\nu$ are scalars and $w \in W$, and then takes $F + c_{i+1}v_{i+1}$ as the
 new strictly $k_{i+1}$-collapsible part, and $c_{i+2}v_2 + \cdots + c_mv_m + w$ as the part in the auxiliary
vector space.    The number of generators needed for the auxiliary ideal increases
 from $k_{i}$ by the maximum number of generators for $\fJ_{i+1}$ which is at most $b_{i+1}$,  
 and it is clear that the dimension of the auxiliary vector space is at most $h-i$.  In particular,
 it follows for $i = m$, that one has a collective $(k', h-m)$-collapse as stated.   By the maximality
 of the sequence $\vect v m$,  no element of $W-\{0\}$ has a strict $ak' = ak_m = b_{m+1}$-collapse, 
 or we could take that element to be $v_{m+1}$.   \end{proof}

\begin{proposition}\label{uq} If a family of forms of degree $d$ has a collective $(k,h)$-collapse that uses an 
 auxiliary vector space $V$ in which no nonzero element has a strict $2k$-collapse (i.e., $V$ is $2k$-strong),  
 then every $K$-linear combination of elements of the family can be written uniquely as the sum
 of an element with a strict $k$-collapse and an element of $V$. \end{proposition} 
 
\begin{proof} If  $G_1 + v_1 = G_2 + v_2$ where the $G_i$ have a strict $k$-collapse and the $v_i \in V$,
 then $G_2 - G_1 = v_1 - v_2 \in V$ has a strict $2k$-collapse, a contradiction unless $v_1 = v_2$. \end{proof}
 
 \begin{discussion} {\bf Uniqueness of representation.}
 By combining Theorem~\ref{mvclpse} and Proposition~\ref{uq}, we can start with a 
 collective $(k,h)$-collapse of a family of $d$-forms and, taking $a = 2$ (or a larger value)
 in Theorem~\ref{mvclpse}, modify it to a collective
 $(3^{m}k, h-m)$-collapse for some $m$,  $0 \leq m \leq h$ in which the representation
 of every form in the family as the sum of an element with a strict $(3^{m}k)$-collapse
 and an element in the auxiliary vector space of dimension $h-m$ is unique.  Let
 $k' = 3^{m}k$ and $h' = h-m$.  Note that $k' \leq 3^hk$ and $h' \leq h$.   In the situation
 of Theorem~\ref{gencol},  we may modify a generic collective $(k,h)$-collapse in this
 way to a $(k',h')$-collapse in which the auxiliary vector space is $2k'$-strong.  By Theorem~\ref{gencol}, 
 the same condition holds for the collapses that arise by specialization as in Theorem~\ref{gencol} 
 part (b)(2).  \end{discussion}

\section{Using a single auxiliary vector space}\label{findW}         

 Again let all notation be as in Discussion~\ref{gensetup}, so that $L$ is an algebraic closure of
 $K(\vect t N)$, and that we have a collective $(k,h)$-collapse for $\fH T$:  we may assume that
 there is no such collapse for a smaller value of $h$, so that the auxiliary vector space has 
 dimension $h$.

\begin{discussion}\label{Tmtx}
Suppose that $t_{ij}$, where $1 \leq i,\,j \leq N$, are any $N^2$ algebraically independent elements in an algebraically closed field
$\Omega$.  Let $\cT = \bigl(t_{ij}\bigr)$ and let $T^{(j)}$ denote the $j\,$th column of $\cT\tr$, whose
$i\,$th entry is  $t_{ji}$.  For every $j$, $1 \leq j \leq N$, let $L_j$ be the algebraic closure of $K(t_{j1}, \, \ldots, \, t_{jN})$ within $\Omega$ and fix a $K$-isomorphism of $L$ with $L_j$ that extends the $K$-isomorphism of $K[\vect tN]$ with
$K[t_{j1}, \, \ldots, \, t_{jN}]$ that sends $t_i \mapsto t_{ji}$,  $1 \leq i \leq N$.   Let $\theta_j$ be the composite
$L \to L_j \inj \Omega$.
Specializing
$T$ to $T^{(j)}$  yields a collective $(k,h)$-collapse for $\fH T^{(j)}$:  let $W_j$ denote the auxiliary vector space over
$\Omega$ that is used. We write $\fA_j$ for the auxiliary ideal in $L_j[\vect x N]$, which will be generated
by $k$ or fewer polynomials of degree strictly less than $d$. Thus, we may write
$$
\fH T^{(j)} = \fC^{(j)} + \fD^{(j)}
$$ 
where the entries of columns $\fC^{(j)}$ and $\fD^{(j)}$ have coefficients in $L_j$,
the entries of $\fC^{(j)}$ are in the ideal $\fA_j$, and the entries of $\fD^{(j)}$ span the vector
space $W_j$.   By left multiplying by $\cT = \bigl(t_{ij}\bigr)$ we get a collective collapse for the $j\,$th 
column of $\cT \fH \cT\tr$  that uses the same vector space  $W_j$, since 
$$
\cT \fH T^{(j)} = \cT \fC^{(j)} + \cT \fD^{(j)}.
$$  
Because we have 
a collective $(k,h)$-collapse for each of the columns of $\cT \fH \cT\tr$,  we can write
$$\cT \fH \cT\tr = \fP + \fQ$$ 
where each column of the $N \times N$ matrix $\fP$ has a strict $k$-collapse and the $j\,$th column of the matrix
$\fQ$ is in the vector space $W_j$  generated over $\Omega$ by $d$-forms with coefficients in $L_j$.  
Each $W_j$ has dimension $h$.  \medskip
\end{discussion}

The next result plays a critical role in the proof of our main result for quartics.

\begin{theorem}\label{oneW}  Let $\fH$ be a symmetric matrix with entries that are 0 or else forms of degree
$d$ in a polynomial ring $R$ over an algebraically closed field $\Omega$. Let $j$ be any integer, $1 \leq j \leq N$.  Suppose that  for $i = j$ and for $N-h \leq i \leq N$,  the $i\,$th column  $\fH^{(i)}$ of $\fH$ 
can be written as $\fP^{(i)} + \fQ^{(i)}$,  where the entries of $\fP^{(i)}$ are in an ideal
$\fA_i$ generated by at most $k$ forms of degree strictly less than $d$, and the $\fQ^{(i)}$ span a vector
space $V_i$ over $\Omega$ of dimension at most $h$ whose nonzero elements are $d$-forms.
This explicitly shows that for $i = j$ and for $N-h \leq i \leq N$,  the $i\,$th column of $\fH$ has
a $(k,\,h)$-collapse using the ideal $\fA_i$ and the auxiliary vector space $V_i$.  
Also assume that if $i < N-h$ (respectively, $i \geq N-h$),   $V_i$ is spanned by the bottommost 
$h$ elements of $\fQ^{(i)}$ (respectively, the bottommost $h+1$ elements of $\fQ^{(i)}$).
Let $W_\Omega = W$ be the vector space spanned over $\Omega$ by the entries of the bottommost $h+1$ 
elements of the $\fQ^{(i)}$ for $N-h \leq i \leq N$. Let $\fd := \di_{\Omega}(W)$.   Then the $j\,th$ column of $\fH$ has a
$\bigl(k(h+1), \fd \bigr)$-collapse using the ideal $\fA = \fA_j + \sum_{i=N-h+1}^N \fA_i$ and the
auxiliary vector space $W$. \end{theorem}

\begin{proof} The statement is obvious for elements of the $j\,$th column if $j \geq N-h$: each is the sum of
an element with a strict $k$-collapse using $\fA_j$ and an auxiliary vector space contained in $W$.
Now consider any element $F$ 
of the vector space spanned by the $j\,$th column of $\fH$ for $j < N-h$.  It is the sum of an element
$F_0$ with a strict $k$-collapse using the ideal $\fA_j$ for that column and an element $q$ in the span of 
$\fQ^{(j)}$.  Then $q = \sum_{i=N-h+1}^N c_i q_i$  where the $c_i \in \Omega$ and 
$q_{N-h+1}, \, \ldots, \, q_N$  are the $h$ 
bottommost elements in  $\fQ^{(j)}$.   
Each of the $q_i$ 
comes from writing an element $H_i$ among the
bottommost $h$ entries of the $j\,$th column of $\fH$ as $G_i + q_i$,  where $G_i$ has a strict $k$-collapse 
using $\fA_j$. But by transposing the symmetric matrix $\fH$, we may also think of $H_i$ as an element 
of one of the rightmost $h$ columns of $\fH$,   and, as such, we can write $H_i = P_i + Q_i$  where  $P_i$ has
a strict $k$-collapse using the ideal $\fA_i$ for the $i\,$th column of $\fH$  and $Q_i \in W$.   
Thus, $q_i = H_i - G_i = P_i - G_i + Q_i$ and
$$F = F_0 + \sum_{i = N-h+1}^N c_i(P_i - G_i + Q_i),$$
where each term is either in $\fA_j$, or in $\fA_i$, $N-h+1 \leq i \leq N$, or in $W$.  Since 
$\fA = \fA_j + \sum_{i= N-h+1}^N \fA_i$
is the sum of $h+1$ homogeneous ideals with at most $k$ homogeneous generators of degree strictly less than $d$,
$\fA$ is a homogeneous ideal with at most $k(h+1)$ homogeneous generators of degree strictly less
than $d$, and the $j\,$th column of $\fH$ therefore has a $\bigl(k(h+1),\, \fd\bigr)$-collapse
using the ideal $I$ and the auxiliary vector space $W$. \end{proof}

\begin{theorem}\label{propW} Let notation and hypothesis  be as in Discussions~\ref{gensetup}~and~\ref{Tmtx}.
\begin{enumerate}[(a)]
\item The vector space spanned over $\Omega$ by a column of $\fQ$ is spanned by any $h$ off-diagonal entries of
the column.  
\item The vector space $W_\Omega = W$ spanned by the $h+1$ rightmost columns of $\fQ$  is the same as 
the vector space spanned by the entries of the columns  $\fD^{(j)}$ for $N-h \leq j \leq N$.  It is also the same 
as the  span of the entries of $\fQ$ with indices  $i, j$ satisfying  $N-h \leq i, \, j \leq N$ with $i \not= j$,  i.e., of the 
off-diagonal entries of the $(h+1) \times (h+1)$ submatrix of $\fQ$ in the lower right corner. 
The latter set of generators of $W$ consists of elements that have coefficients in the field $\cL$
that is the algebraic closure in $\Omega$ of the field $K(t_{ji}: N-h \leq  i \leq N, \, 1 \leq j \leq \N)$.  
Moreover, $\fd := \di_\Omega(W)$ is at most $h(h+1)$.   

\item Every column of $\cT \fH \cT\tr$ has
a collective $\bigl(k(h+1),\, \fd \bigr)$-collapse in which the $\fd$ elements of degree $d$ can be
taken to be any basis for $W$. In particular, these $\fd$ elements may be chosen from the off-diagonal
elements of the square size $h+1$ submatrix of $\cT \fH \cT\tr$ in the lower right corner, in which case
they have coefficients in $\cL$.  The ideal of the collapse is generated by the ideals of the collapses
of the given column and the rightmost $h$ columns of $\cT \fH \cT\tr$.  
 Thus the auxiliary vector space for all these $\bigl(k(h+1),\, \fd \bigr)$-collapses, which is $W$,  is independent
 of the choice of the column.  
   
 \item Let $Z_0$ be any $N-(h+1)$ size square matrix over $\Omega$ and let $Z$ be the direct sum
 of $Z_0$ with a size $h+1$ identity matrix $I$, so that, in block form, 
 $Z = \begin{pmatrix} Z_0 & 0\\ 0 & I \end{pmatrix}$. Suppose that $ZT$ has entries that are algebraically
 independent over $K$.  Then every column of  $(Z\cT)\fH(Z\cT)\tr$ has a $\bigl(k(h+1),\, \fd\bigr)$-collapse
 using $W$ as the auxiliary vector space. 
 
  \item We also have the $\bigl(k(h+1),\, \fd \bigr)$-collapses described in part (c) and (d) if $\Omega$ is replaced
 by any larger algebraically closed field $\Omega'$:   $W$ is replaced by $W_\Omega' := \Omega' \otimes_\Omega W$
 but may be taken to have the same basis.  
 
\end{enumerate}
\end{theorem}

\begin{proof}  (a)  Consider the $j\,$th column $\fD^{(j)}$ of $\fD$,  which consists of polynomials with
coefficients in $L_j$.  We can choose distinct integers $1 \leq i_1 < \cdots < i_h \leq N$ such
that the entries of the column in the positions with indices $i_1, \, \ldots, i_h$ are a basis for
the span of the entries of the column.  Then $h$ off-diagonal entries of the $j\,$th column of $\fQ$,
which is $\cT \fD^{(j)}$ are given by integers $1 \leq a_1 < \cdots < a_h \leq N$ such that the
$a_i$ are all different from $j$.  Let $\cT_0$ be the submatrix of $\cT$ consisting of the rows
of $\cT$ indexed by $\vect a h$.  Then $\cT_0$ is an $h \times N$ matrix of indeterminates
$t_{a_i,b}$ with $a_i \not= j$.  Thus, these indeterminates are algebraically independent over
$L_j$.  The fact that the entries of $\cT_0 \fD^{(j)}$ are linearly independent now follows
from the fact that we can specialize the entries of $\cT_0$ without affecting $\fD^{(j)}$.  In particular
we can specialize the entries of  $\cT_0$ so that its $\nu\,$th row has 1 as its $i_{\nu}$ entry
and 0 for all of its other entries.   Then $\cT_0$ specializes to a matrix $M$ such that the entries
of $M \fD^{(j)}$ are the entries of $\fD^{(j)}$ indexed by $\vect i h$, which are independent over
$L_j$.  

(b) It is now clear that the vector space over $\Omega$ spanned by the $h$ bottommost off-diagonal entries 
of one of the rightmost $h+1$ columns over $\fQ$ is the same as the $\Omega$-vector space spanned by all of 
the entries of that column. 
Since multiplying by the invertible matrix $\cT$ does not affect the $\Omega$-span of the entries of a column,
this is also the same as the span of all entries of the $\cD^{(j)}$ for $0 \leq N-j \leq h$.  

(c) and (d).  We may write $\cT\fH\cT\tr$ in block form as 
$\begin{pmatrix} M_{11} & M_{12} \\ M_{12}\tr & M_{22} \end{pmatrix}$
where $M_{11}$ is square and symmetric of size $N-h-1$,  $M_{12}$ is $(N-h-1) \times (h+1)$,  
and $M_{22}$ is square and symmetric of size $h+1$.  Then $\cZ\cT \fH (\cZ \cT)\tr =
\begin{pmatrix} Z_0M_{11}Z_0\tr & Z_0M_{12} \\ M_{12}\tr Z_0\tr & M_{22} \end{pmatrix}$.  In particular,
the last $h+1$ columns are the columns of $Z \begin{pmatrix} M_{12} \\ M_{22} \end{pmatrix}$
and each of these consists of $\Omega$-linear combinations of the entries of the last $h+1$
columns of $\cT \fH\cT\tr$.  Hence, each of the last $h+1$ columns of $\cZ\cT \fH (\cZ \cT)\tr$
has a collective $(k,h)$-collapse using the ideal for the corresponding column of $\cT\fH\cT\tr$
and a subspace of $W$. We may now apply Theorem~\ref{oneW}, taking $\fH$ to be
$\cT\fH\cT\tr$ for part (a) and $\cZ\cT \fH (\cZ\cT)\tr$ for part  (d).   (Note that (c) is actually
the special case of (d) where $Z_0$ is an identity matrix.)  The conditions in Theorem~\ref{oneW}
that the auxiliary vector spaces be spanned by certain bottommost elements follow from
part (a) and the fact that the entries of $Z\cT$ (or $\cT$) are algebraically independent over $K$.  
 
(e) is obvious.  \end{proof} 

\begin{theorem}\label{W0}  Let notation and hypothesis as in Discussions~\ref{gensetup}~and~\ref{Tmtx}.
Let $W$ be defined as in Theorem~\ref{propW}(b), and apply Theorem~\ref{mvclpse} with $a = 3$ to find
an integer $m \in \N$ and subspace $W_0$ of $W$ of dimension $\fd_0 = \fd-m$ such that every column of
$\cT \fH \cT\tr$ has a collective $(k', \fd_0)$-collapse with $k' = 4^mk$ and no nonzero element
of $W_0$ has a strict $3k'$-collapse, i.e., $W_0$ is $3k'$-strong.  Then we can write
$\cT \fH \cT\tr = \fP + \fQ$ uniquely, where $\fP$ has the property that every column
has a strict $k'$-collapse, and $\fQ$ has entries in $W_0$. The matrices $\fP$ and $\fQ$ are 
symmetric.  Moreover, if $Z_0$ is
an $N-h-1$ by $N-h-1$ matrix with algebraically independent entries over $K(t_{ij}: 1\leq i,j \leq N)$
and $Z$ is the direct sum of $Z_0$ with a size $h+1$ identity matrix, then
every column of $Z\fP Z\tr$  has a strict collective $k'$-collapse.  In particular, if $\fP_0$ is the
size $N-h-1$ square submatrix of $\fP$ in the upper left corner,  then
every column of $Z_0 \fP_0 Z_0\tr$ has a strict collective $k'$-collapse. \end{theorem}

\begin{proof}  Let $M = N-h-1$ and let $Z_0 = \bigl( z_{ij}\bigr)$ have entries that are algebraically independent
over the algebraic closure $\cK$ of $K(t_{ij}: 1 \leq i,j \leq N)$.     

For $1 \leq i, j \leq M$ with $i \not=j$ and $y \in \Omega$ let 
$E_{ij}(y)$ denote the elementary size $M$ square matrix obtained by adding $y$ times
the $j\,$th row of the size $M$ identity matrix to the $i\,$th row.  
Let $D_i(y)$ be the diagonal matrix with $y$ in the $i\,$th spot on the diagonal and 1 in all
other spots on the diagonal.  We refer to all of these as {\it elementary matrices}.
We may work over $\kappa(z_{ij}: 1 \leq i,j \leq M)$, where $\kappa$ is the prime field of $K$,
to write 
$$
(*) \quad Z_0 = \prod_{j=1}^M \Biggl(D_j(y_{jj})\biggl(\,\prod_{1 \leq i \leq M,\, i\not= j} E_{ij}(y_{ij})\biggr)\Biggr)
$$
This simply corresponds to performing elementary operations on $Z_0$ until one gets an
identity matrix, and then writing down the product of the inverses of the elementary matrices
used.  At the first step,  $y_{11} = z_{11}$.  The next $M-1$ steps subtract multiples
of the first entry of the first row so that the other entries of the first row become 0.  After $sM$ steps
the first $s$ rows agree with those of the identity matrix.   One takes $y_{s+1,s+1}$ to
be the entry of the matrix obtained at the $(s+1, s+1)$ spot. One additional left multiplication
by $D_{s+1}(y_{s+1,s+1})$ enables us to assume that the $(s+1, s+1)$ entry is 1.   One then performs
column operations to make all of the other of the entries of the $s+1$ row 0.  The elements we need to 
invert in this process are never 0,  because that does not happen even if we specialize
the $z_{ij}$ to the entries of the size $M$ identity matrix.  Since the $z_{ij}$ are in the
field (even the $\kappa$-algebra) generated by the $y_{ij}$, the $M^2$ elements $y_{ij}$
are algebraically independent both over $\kappa$ and over $L$:  the $z_{ij}$ and the
$y_{ij}$ generate the same field over $\kappa$ or $L$.  

Let $s = M^2$ and denote the product on the right hand side of $(*)$ as $E_s \cdots E_1$.
Let  $Y_\nu$ be the direct sum of $E_\nu$ and a size $h+1$ identity matrix, $1 \leq \nu \leq s$.
Then $Z = Y_m \cdots Y_1$. By a straightforward induction on $\nu$,  each of the
matrices  $$Y_1, Y_2Y_1, \, \ldots, \, Y_{\nu}\cdots Y_1, \, \ldots, \,
Y_m \cdots Y_1 = Z$$
has entries that are algebraically
independent over $L$,
and it follows that each of the matrices 
$$\cT, Y_1\cT,  Y_2Y_1\cT, \, \ldots, \,Y_{\nu}\cdots Y_1\cT,\, \ldots, \,
Y_m \cdots Y_1\cT = Z\cT$$ has algebraically independent
entries over $K$.  Let $\cT_\nu = Y_\nu \cdots Y_1 \cT$, $1 \leq \nu \leq s$ and $\cT_0 = \cT$.  Then for each of the
matrices $\cT_\nu$, $0 \leq \nu \leq s$, we have a decomposition:  
$\cT_\nu \fH \cT_\nu\tr \fP_\nu + \fQ_\nu$,   where $\fP$ and $\fP_\nu$ have the property
that every column has a collective $(k',h')$-collapse, and the $\fQ_\nu$ have all
entries in $W_0$.  

The choice of $a = 3$ guarantees that these representations are unique.  Since $\cT\fH\cT\tr$ is symmetric,
$\fP$ and $\fQ$ are symmetric  as well:  by Proposition~\ref{uq}, an element of $\cT\fH\cT\tr$ can be written as the sum of an element with a strict $k'$-collapse and an element in  $W_0$ in only one way (the element with the $k'$-collapse   
need not be assumed to come from a specific auxiliary ideal --- one only needs that every element of $W_0 -\{0\}$
is at least  $2k'$-strong). 
We now show by induction on $\nu$ that for $\nu \geq 1$ we have
$$
\fP_\nu = (Y_\nu \cdots Y_1)\fP(Y_\nu \cdots Y_1)\tr\hbox{\rm\ \ and\ \ }
\fQ_\nu = (Y_\nu \cdots Y_1)\fQ(Y_\nu \cdots Y_1)\tr.
$$ 
This yields the desired conclusion when $\nu = s$.  

This comes down to showing that if  the entries of $\cT_{\nu}$ and the element $y$  are algebraically
independent and $Y = Y_{\nu+1}$ is the direct sum of an  elementary matrix, either $E_{ij}(y)$ or $D_i(y)$, with
a size $h+1$ identity matrix,  then $Y\fP_{\nu} Y\tr = \fP_{\nu+1}$ and  $Y\fQ_{\nu} Y\tr = \fQ_{\nu+1}$.
Clearly, we have $\cT_{\nu+1} = Y\fP_\nu Y\tr + Y\fQ_\nu Y\tr$.  The last term is a matrix with entries
in $W_0$.  If $Y$ is diagonal, every entry of $Y \fP_\nu Y\tr$ is a multiple of an entry of $\fP$ and so
has a strict $k'$-collapse.  If $Y = E_{ij}(y)$, every entry of $Y\fP_\nu$ is either an entry of $\fP_\nu$
or a linear combination of two entries of $\fP_\nu$ in the same column.  Hence every entry of $Y\fP_\nu$
has a strict $k'$-collapse.  Every entry of $(Y\fP_\nu)Y\tr$ is, likewise, a linear combination of at most two entries
from a row of $Y\fP$,  and so has, at worst, a strict $(2k')$-collapse.   Thus, in all cases,  every entry of
$Y\fP_\nu Y\tr$ has, at worst, a strict $(2k')$-collapse.  Let $F$ be an entry of $Y\fP_\nu Y\tr$.  Then
we have found a representation $F = G_1 + H_1$  where $G_1$ has a strict $(2k')$-collapse and $G_1$
is in $W_0$.   From the fact that $Y \cT_{\nu} Y\tr = Y_{\nu+1} = \fP_{\nu+1} + \fQ_{\nu + 1}$
we also have $F = G_2 + H_2$,  where $G_2$ has a strict $k'$-collapse and $H_2 \in W_0$.
It follows that $G_2 - G_1 = H_1 - H_2 \in V_0$ is an element of $V_0$ with a strict $(3k')$-collapse.
This is a contradiction, by our choice of $a$,  unless $H_1 - H_2 = 0$.  This implies that
$H_1 = H_2$ and $G_1 = G_2$.  This completes the proof. \end{proof}

 \section{Hessian-like matrices of quadrics with entries of stably bounded rank}\label{hesscol}
 
Theorem~\ref{stabPcol} below plays a critical role in the 
construction of $\etuA(n_1,\,n_2, \, n_3, n_4)$ for algebraically closed fields of characteristic
0 or $\geq 5$ in the next section.   Note that in this result, we do not need to assume that
$\fP_0$ is symmetric.

\begin{lemma}\label{rankZP} Let $\cK$ be an algebraically closed field of characteristic different from $2$.
Let $M,\,N$ be positive integers.
Let $R= \cK[\vect x N]$ be a polynomial ring over $\cK$.   
Let $\cZ = \bigl(z_{ij}\bigr)$ be an $M \times M$ matrix of indeterminates over $\cK$. 
 Let $\fP_0 = \bigl(f_{ij}\bigr)$ be an $M \times M$ matrix whose 
entries are quadratic forms in $R$. 
Then all off-diagonal entries of  $\cZ \fP_0 \cZ\tr$ have the same rank $r_1$ over $\cF$,
and all diagonal entries have the same rank $r \leq r_1$. If $M = 1$ we make the convention
that $r_1 := r$.    For a Zariski 
dense open subset $U$ of $\GL(M,\,\cK)$, for all $\alpha \in U$ all off-diagonal (respectively, diagonal) 
entries of $\alpha \fP_0 \alpha\tr$ have
rank $r_1$ (respectively, $r$).  For an off-diagonal entry $f$  in the $(i,j)$ spot of $\alpha \fP_0 \alpha$ 
with $\alpha \in U$, the $i\,$th row and $j\,$th column
of $\alpha \fP_0 \alpha\tr$ are both contained in the ideal $(\cD f)R$.  
\end{lemma}
\begin{proof} Let $\cZ_i$ be the $i\,$th row of $\cZ$ and $\cZ^j$ the $j\,$th column.  It is clear that the 
maximum rank that can be achieved at the $(i,j)$ spot in 
$\alpha \fP_0 \alpha\tr$
is the same as the rank at the $(i,j)$ spot of $Z \fP_0 Z\tr$: rank $\rho$ is achieved if and only if the size
$\rho$ minors of certain matrix of polynomials in the $z_{ij}$ do not all vanish. (Each quadratic form
in $\fP_0$ corresponds to a symmetric $N \by N$ matrix of constants in $\cK$, and $Z_i \fP_0 \cZ^j$ will
correspond to an $N \by N$ symmetric matrix whose entries are in $\cK[z_{ij}:0\leq i,j\leq M]$).  
If $i\not=j$, by interchanging the  $i\,$th and $j\,$th rows of $\alpha$ one sees that if a certain rank
is achieved at the $(i,i)$ spot in $\alpha \fP_0 \alpha\tr$,  it is also achieved at the $(j,j)$ spot.  Moreover
this shows as well that the rank achieved at the $(i,h)$ spot is also achieved at the $(j,h)$ spot if $h \not=i,j$, 
so that for any column, the highest rank that can be obtained at an off-diagonal spot is constant.
Since the same is true for rows,  the highest possible ranks $r$, $r_1$ are constant both for spots on and 
for spots off the main diagonal.  We show that $r \leq r_1$ below.

Choose $\alpha \in U$ so that $\fP_1 = \alpha \fP_0 \alpha\tr = \bigl(g_{ij}\bigr)$ has maximum possible 
rank for all entries.
If $i\not=h,j$ and $c \in \cK$ by adding $c$ times the $h\,$th row to the $i\,$th row and also $c$ times the 
$h\,$th column to the  $i\,$th column in $\fP_1$  we obtain $cg_{hj} + g_{ij}$  at the $(i,j)$ spot. 
By Proposition~\ref{rk},  the rank of $cg_{hj} + g_{ij}$ will increase from $r_1$ for
infinitely many values of $c$ unless $g_{hj} \inc (\cD g_{ij})R$, and it will be at least the rank of
$g_{hj}$ even when $h = j$.  This shows that $r \leq r_1$ and that for $h, j$ distinct from $i$,
$g_{hj} \in (\cD g_{ij})$, even if $h =j$.  But $g_{ij} \in \cD g_{ij}$ in general, so that the entire $i\,$th
row of $\fP_1$ is in the ideal $\cD g_{ij}$.  The entire $j\,$th column is in $\cD g_{ij}$
by a similar argument.  \end{proof}

\begin{example} The ranks $r_1$ and $r$ can be different.  Let $V$ be a $1 \times M$ row of linearly
independent linear forms and let $\fP_0 = VV\tr$.  Then $A \fP_0 A\tr = (AV)(AV)\tr$,  and each diagonal
entry is the square of a linear form,  and so has rank 1.   Each off-diagonal entry is
the product of two linearly independent linear forms and has rank 2.    If $\fP_0$ is skew-symmetric
the diagonal entries are 0 while the off diagonal entries can have arbitrarily large rank. \end{example}

We are now ready to prove the last of the critical results needed to establish the existence of
$\uA(n_1,\,n_2,\,n_3,\,n_4)$.  We only need the case $a=1$ to handle quartics, but there
are potential applications to higher degree cases.

\begin{theorem}\label{stabPcol} Let $\cK$ be an algebraically closed field of characteristic different from $2$,
let $M$, $N$, and $a$ be positive integers, and
let $R= \cK[\vect x N]$ be a polynomial ring over $\cK$.   
Let $\cZ = \bigl(z_{ij}\bigr)$ be an $M \times M$ matrix of indeterminates over $\cK$ and 
let $\fP_0 = \bigl(f_{ij}\bigr)$ be an $M \times M$ matrix whose 
entries are quadratic forms in $R$.  Let $r_1$ be the greatest rank of an off-diagonal entry of $\cZ\fP_0\cZ\tr$  and let $r$ be the greatest 
rank of a diagonal entry.
Let $V = \bigl(L_1 \,\, \ldots \,\, L_M\bigr)$ be  a $1 \times M$ row of forms of degree $a$ in $R$.   
Then $f:=\sum_{0 \leq i,j \leq M} L_iL_j f_{ij}$, the unique entry of $V\fP_0 V\tr$, 
has a strict $(2r_1 + 2r)$-collapse (a strict $(r_1 + 2r)$-collapse 
if $\fP_0$ is symmetric).                         
\end{theorem} 

\begin{proof}
By Lemma~\ref{rankZP}
the ranks $r$ and $r_1$ will be achieved for the specialization of $\cZ$ to a matrix $\alpha$ in an open subset $U$
of $\GL(M,\, \cK)$.  Hence, we may replace $\fP_0$ by $\alpha \fP_0 \alpha\tr$ for such an $\alpha$ and assume that 
$f_{11}$ has rank $r$.  Because $Z\alpha^{-1}$ has algebraically independent entries over $\cK$, our hypothesis
is preserved.  At the same time we replace $V$ by $V\alpha^{-1}$.   

 Moreover, since the characteristic is different from 2, for each quadratic form $f_{ij}$ we have a corresponding symmetric $N\times N$ symmetric matrix  $\Lambda_{ij} = \bigl(\lambda^{ij}_{st}\bigr)$.  After a $\cK$-linear change of 
coordinates in $R$   we may assume that $\Lambda_{11}$ is diagonal, with the first $r$ entries on the diagonal equal 
to 1 and the other diagonal entries 0.  We note that in the discussion below we have $1 \leq i,j \leq M$ and
$1 \leq s,t \leq N$.  Sometimes there are additional restrictions on these indices.

Let $Z$ be the first row of $\cZ$.  To simplify notation we let
$z = z_{11}$ and $z_j = z_{1j}$ for $j >1$.   
Then  $G = \sum_{1 \leq i \leq M}z_i^2f_{ii} + \sum_{1 \leq i \not= j \leq M} z_iz_j f_{ij}$  is 
the $(1,1)$ entry of $\cZ \fP \cZ\tr$ and  has rank  $r$. Call the corresponding $N \times N$ symmetric matrix  
$\Lambda$:  it has coefficients in $\cK [z_i: 1 \leq i \leq M]$.  In $\Lambda$, the first $r$ entries of the main diagonal
have the form  
$$
(\#) \quad z^2 +\quad \hbox{\rm lower degree terms in\ }z
$$ 
Note that $z$ occurs with degree at most one in other
entries of $\Lambda$.    Consider the $N-r$ size square submatrix $C$ of $\Lambda$ in the lower right corner.
We want to study the constraints imposed on $C$ and on $\Lambda$ by the condition that $\Lambda$ 
have rank at most $r$, i.e., that the $r+1$ size minors of $\Lambda$ vanish.

Fix $s, t > r$ corresponding to some entry of $C$.  
Consider the  $r+1$ size square submatrix $\Theta_{s,t}$ of $\Lambda$
coming from the rows indexed $1,\, 2,\, \ldots,\, r,\, s$ and the columns indexed $1,\, 2,\, \ldots,\, r,\, t$.
It contains the upper left $r \times r$ block whose main diagonal consists of terms as described in $(\#)$.
When we calculate its determinant as a sum of $(r+1)!$ products of entries,  the main diagonal provides 
a term of degree  $z^{2r+1}$ or $z^{2r}$ depending on whether or not $z$ occurs with nonzero
coefficient in the $(s,t)$ entry of $\Lambda$.  We cannot obtain $z^{2r+1}$ in any other way. Hence, if
$z$ occurs with nonzero coefficient in the $(s,t)$ entry of $\Lambda$,  this $r+1$ size minor of $\Lambda$
does not vanish, a contradiction.  Hence, we may assume that for all $s>r$ and $t >r$,   $z$ does not
occur in the $(s,t)$ entry of $\Lambda$.  This implies that the determinant of $\Theta_{s,t}$ has degree
at most $2r$ as a polynomial in $z$.  We may calculate the coefficient of $z^{2r}$ in terms of
the $z_i$ for $2 \leq i \leq M$ and the scalars $\lambda^{ij}_{st}$.

As already mentioned, we obtain a term of degree $2r$ in $z$ by using the first $r$ terms on the 
main diagonal and the term in the $(s,t)$ spot.  

Since  $\Lambda = \sum_i z_i^2 \Lambda_{ii} + \sum_{i,j} z_iz_j \Lambda_{ij}$, the only other way 
to get  $z^{2r}$ when we compute the determinant of $\Theta_{s,t}$ is by taking  $r-1$ terms from the 
main diagonal, omitting, say,
the  $\rho$\,th term, where  $1\leq \rho \leq r$,  and two terms, one, $\Lambda_{s \rho}$, 
from the bottom row of $\Theta_{s,t}$ (which is part of the $s\,$th row of $\Lambda$) with column index $\rho$, 
and the other, $\Lambda_{\rho t}$, from the rightmost column of $\Theta_{s,t}$ (which is part of the $t\,$th column of 
$\Lambda$) with row index $\rho$.  
The contribution to the coefficient of $z^{2r}$ from the term of this type in the determinant of $\Theta_{s,t}$ 
coming from one choice of $\rho$ is 
$$ -(\hbox{\rm the coefficient of $z$ in $\Lambda_{s \rho}$})(\hbox{\rm the coefficient of $z$ in $\Lambda_{\rho t}$}).$$
 The coefficient of $z = z_1$ in $\Lambda_{st} = \sum_{i,j} \lambda^{ij}_st z_iz_j$  comes from the $z_1z_j$ terms,
 $j \geq 2$, and the $z_i z_1$ terms, $i \geq 2$,  so that the coefficient of $z$ is
 $$\sum_{j \geq 2} \lambda^{j1}_{st}z_j + \sum_{i \geq 2} \lambda^{1i}_{st}z_i$$ which we can rewrite as
 $$\sum_{i \geq 2} (\lambda^{i1}_{st} + \lambda^{1i}_{st})z_i$$ or as
 $$\sum_{j \geq 2} (\lambda^{j1}_{st} + \lambda^{1j}_{st})z_j.$$ 
 
Taking account also of the contribution to the coefficient of $z^{2r}$ that is entirely from the main diagonal in 
$\Theta_{s,t}$, which is the product of $z^{2r}$ with $\Lambda_{s,t}$ (we observed earlier that $z$ does not occur
in this entry), 
we have that the coefficient of $z^{2r}$ in the determinant of $\Theta_{s,t}$ is
$$
\sum_{i,j\geq 2} \lambda_{st}^{ij}z_iz_j -
\sum_{\rho=1}^r \biggl(\sum_{i\geq 2} (\lambda^{i1}_{\rho s} + \lambda^{1 i}_{\rho s})z_i 
\sum_{j\geq 2} (\lambda^{j 1}_{t \rho} + \lambda^{1j}_{t \rho})z_j\biggr).
$$
Since $\Theta_{s,t}$ has rank $r$, the displayed expression vanishes, i.e., for all $s,t > r$,
$$ (\dagger)\quad \sum_{i,j\geq 2} \lambda_{st}^{ij}z_iz_j =
\sum_{\rho=1}^r \biggl(\sum_{i\geq 2} (\lambda^{i1}_{\rho s} + \lambda^{1 i}_{\rho s})z_i 
\sum_{j\geq 2} (\lambda^{j1}_{t  \rho} + \lambda^{1 j}_{t \rho})z_j\biggr).$$

Since the $z_j$ are indeterminates, we may substitute the $L_j$ for the $z_j$ in $(\dagger)$ 
to obtain that for all $s,t > r$,
$$ (\dagger\dagger)\quad \sum_{i,j\geq 2} \lambda_{st}^{ij}L_iL_j =
\sum_{\rho=1}^r \biggl(\sum_{i\geq 2} (\lambda^{i1}_{\rho s} + \lambda^{1 i}_{\rho s})L_i 
\sum_{j\geq 2}(\lambda^{j1}_{t \rho} + \lambda^{1j}_{t \rho})L_j\biggr).$$

Choose an off-diagonal entry $g_1$ in the first row of $\fP_0$ and an off-diagonal entry $g_2$ of
the first column.  Then by Lemma~\ref{rankZP},  the $2r_1$ linear forms needed to span $\cD g_1$ and
$\cD g_2$ generate an ideal that contains all the $f_{1j}$ and all of the $f_{j1}$.  Note that
if $\fP_0$ is symmetric, we may take $g_1 = g_2$ and we only need $r_1$ elements.   (If
$M = 1$,  we use instead $f_{11}$ as the generator of the ideal need.)  Let $\fA$ be the ideal
generated by these elements and $\vect x r$.  The congruences below are taken modulo $\fA$,
which has at most $2r_1 + r$ generators ($r_1 + r$ if $\fP_0$ is symmetric).

$$ 
f = \sum_{i,j} L_iL_j f_{ij} \equiv \sum_{i \geq 2, j \geq 2} L_iL_j f_{ij}
$$
since the $f_{i,1}, \, f_{1,j} \in \fA$.  We can rewrite this as
$$      
\sum_{i \geq 2, j \geq 2} L_iL_j\biggl( \sum_{s,t} \lambda^{ij}_{st} x_s x_t \biggr)
$$
which, since $\vect x r \in \fA$,
$$
\equiv \sum_{i \geq 2,j \geq 2}L_iL_j \biggl(\sum_{s>r, t>r}\lambda^{i,j}_{s,t} x_sx_t\biggr)
$$
since $\vect x r \in \fA$.  This becomes
$$ 
\sum_{i \geq 2, j \geq 2} \biggl(\sum_{s>r,t>r} (\lambda^{ij}_{st} x_sx_t L_i L_j) \biggr)=
$$
$$ 
\sum_{s>r,t>r} \biggl( \sum_{i \geq 2,j \geq 2} (\lambda^{ij}_{st} L_i L_j)x_sx_t\biggr) =
\sum_{s>r,t>r} x_sx_t \biggl( \sum_{i \geq 2, j \geq 2}(\lambda^{ij}_{st} L_i L_j)\biggr)
$$ 
Using $(\dagger\dagger)$ this becomes
$$ 
\sum_{s>r,t>r} x_sx_t \,\sum_{\rho=1}^r \biggl(\sum_{i\geq 2} (\lambda^{i1}_{\rho s} + \lambda^{1i}_{\rho s})L_i 
\sum_{j\geq 2} (\lambda^{j1}_{t \rho} + \lambda^{1j}_{t \rho})L_j\biggr) =
$$
$$
\sum_{\rho=1}^r \Biggl(\sum_{s>r,t>r} x_sx_t \, \biggl(\sum_{i\geq 2} (\lambda^{i1}_{\rho s} + \lambda^{1 i}_{\rho s})L_i 
\biggr) \biggl(\sum_{j\geq 2} (\lambda^{j1}_{t \rho}+ \lambda^{1j}_{t \rho})L_j\biggr)\Biggr) = 
$$
$$
\sum_{\rho=1}^r  \Biggl(\sum_{s>r} \biggl(x_s \sum_{i\geq 2} (\lambda^{i1}_{\rho s}+\lambda^{1i}_{\rho s})L_i \biggr)\Biggr) \Biggl(\sum_{t>r} \biggl(x_t\sum_{j\geq 2} (\lambda^{j1}_{t \rho} + \lambda^{1j}_{t \rho})L_j\biggr)\Biggr),
$$
which shows that $f$ has a strict $r$-collapse modulo an ideal generated by at most $2r_1+r$ forms ($r_1 + r$
if $\fP_0$ is symmetric).  This proves that $f$ has a strict $2(r_1+r)$-collapse (a strict $(r_1 +2r)$-collapse if $\fP_0$ is symmetric).  \end{proof}

\section{The quartic case:  $\fK_4$,  and $\etuA(n_1,n_2,n_3,n_4)$}\label{4case}

Before proving the existence of $\fK_4$ for characteristic not 2 or 3, we need two observations.
One is that by iterating Euler's formula, we have that for a homogeneous polynomial $F$ of degree
$d$ over a field $K$,  $d(d-1)F = \sum_{i,j} x_ix_j \partial^2F/\partial x_i \partial x_j$.

Second, we need the following observation:

\begin{lemma}\label{partials} Let $K$ be a field and let $K[\vect x N]$ be a polynomial ring. Let
$G$ be a form of degree $d \geq 3$ such that $d$ is not zero in $K$.
If  $G$ has a $k$-collapse such that $h$ of the $k$ terms in the collapse expression 
involve a linear factor, where $0 \leq h \leq k$, then $\cD G$ has a collective $(2k-h, h)$-collapse.   
Hence, $\cD G$ has a collective $2k$-collapse.
In particular, if $G$ has degree 3 and has a $k$-collapse, then $\cD G$ has a collective $(k,k)$-collapse.
\end{lemma}
\begin{proof} Suppose that $G = \sum_{s=1}^h L_s M_s + \sum_{t=1}^{k-h} N_tP_t$
where the  $L_i$ are linear forms, the $M_i$ have degree $d-1$,  and the $N_s$, $P_s$
have degrees between 2 and $d-2$.  By the product rule, each partial derivative of $G$ is in
the ideal generated by the $L_s$, $N_t$, $P_t$ ($k-h + k-h +h$ elements) and the $M_s$
($h$ elements).  The final two remarks follows at once. (For the final remark, note that every
term in a collapse of degree 3 polynomial involves a linear factor.) \end{proof} 

We are now ready to prove the existence of $\fK_4$.  If $G$ is a form in $K[\vect x N]$, $\nabla G$ denotes
the column of partial derivatives $\partial G/\partial x_i$ of $G$.

\begin{theorem}\label{main4G}  Let $K$ be an algebraically closed field of characteristic not $2$ or $3$.  Then 
we may take $\fK_4(k) = 6k(k+1)4^{k(k+1)} + (k+1)^2$.
\end{theorem}
\begin{proof}  Let $R = K[\vect x N]$ and let $F$ be a quartic form over $K$.  Suppose that 
 $\cD F$  consists entirely of forms with a $k$-collapse.  By Definition~\ref{key}, what we need to
 show is that $F$ has
$\bigl(6k(k+1)4^{k(k+1)} + (k+1)^2\bigr)$-collapse.    
By Proposition~\ref{basech}(b),  it suffices to show this after enlarging the field.  

Let $\fH$ be the  Hessian of $F$.  A linear combination of columns of $\fH$ is the same as $\nabla G$,  
where $G$ is the corresponding linear combination of  $\partial F/\partial x_1, \ldots, \partial F/\partial x_N$. 
It follows from Lemma~\ref{partials} that all of the \LCK-columns of
$\fH$ have a collective $(k,k)$-collapse:  each has the form $\nabla G$ for some $G \in \cD F$.  

Let notation and hypotheses be as in Discussions~\ref{gensetup}~and~\ref{Tmtx}.  In particular, let
$\cT$ be as in Discussion~\ref{Tmtx}, and let $Z$ and $Z_0$ be as in the statement of 
Theorem~\ref{propW}, part (d).
By Theorem~\ref{genkh}, Theorem~\ref{oneW}, 
and Theorem~\ref{propW} with $h = k$,  we first obtain a collective $(k(k+1), k(k+1))$-collapse for the columns of 
$\cT\fH\cT\tr$  using a single auxiliary vector space. We also obtain from Theorem~\ref{W0} 
that  $\cT \fH \cT\tr$  has a decomposition $\fP + \fQ$ such that,
if  (as indicated above)  $Z_0$ is a size $N-k-1$ square matrix of new indeterminates and
$\fP_0$ is the size $N-k-1$ square submatrix in the upper left corner of $\fP$ (this notation is the
same as in the statement of  Theorem~\ref{W0}) then the following hold.
\begin{enumerate}
\item  The elements of $Z_0 \fP_0 Z_0$ all have
a strict $k(k+1)4^{k(k+1)}$-collapse.
\item The elements of $\fQ$ are in a vector space $W_0$ of dimension
at most  $k(k+1)$ over the extended base field. 
\item Hence, the rank of each element of $Z_0\fP_0 Z_0$ is at most  $r = 2k(k+1)4^{k(k+1)}$. 
\end{enumerate}

The iterated Euler's formula discussed above for $d = 4$ yields that $\bigl(12F\bigr) = X \fH X\tr$, where 
$X$ is the row matrix with entries $\vect x N$. Then, 

$$(*) \quad \bigl(12F\bigr) = X(\fP + \fQ )X\tr = X\fP X\tr + X\fQ X\tr.$$

The entries of $\fQ$ are in $W_0$ and so they are in the ideal $I_1$ generated by at most  $k(k+1)$ linear forms. 
Let $I_2$ be the ideal generated by  $x_{N-k+1}, \, \ldots, x_N$.  Modulo $I_2$,  $X$ becomes the concatenation
of a $1 \times (N-K)$ matrix $X_1$ and a $1 \times k$ block of  zeros. 
Modulo $I_1$,  $\fQ \equiv 0$.   So   $X\fH X\tr \equiv 
X(\fP)X\tr$,  and modulo  $I_1 +I_2$,   $X \fH X\tr  \equiv X_1 \fP_0 X_1\tr$.  

Hence, from $(*)$, we  have $12F \equiv  X_1 \fP_0 X_1\tr$.   By Theorem~\ref{stabPcol} in the symmetric case, 
since $r$ bounds the ranks of all elements, 
 $X_1 \fP_0 X_1\tr$ has a strict $3r$-collapse,
i.e., a strict $6k(k+1)4^{k(k+1)}$-collapse.  Since $I_1 + I_2$ has $k(k+1) + (k+1)$ generators, 
this implies that $12F$ and, hence, $F$, has a strict
$\bigl(6k(k+1)4^{k(k+1)} + k(k+1) + k+1\bigr)$-collapse, as required. \end{proof}

\begin{theorem}\label{mainquart} For algebraically closed fields of characteristic $\not=2,3$,  a choice of the function 
$\etuA(n_1,n_2,n_3, n_4)$ exists for all $\eta \geq 1$, and can be calculated from
Theorem~\ref{SJ}.  Hence, choices of the functions $\etB(n_1,n_2,n_3, n_4)$ and
$C(n_1,n_2,n_3, n_4)$ can also be made explicit. \end{theorem}

In fact, one obtains a closed form for the functions $\etuA$.  The result is rather complicated, but
Corollary~\ref{explA3}, Corollary~\ref{explA}, and Theorem~\ref{main4G} imply at once:

\begin{corollary}\label{explA4} For algebraically closed fields of characteristic not 2 or 3,  if $\delta = (0,\,0,\,0,\,n)$,
corresponding to an $n$-dimensional vector space whose nonzero forms all have degree $4$, 
we may take $$\etA(\delta) = \fK_4\bigl((2n+\eta)A_3(2n+\eta)\bigr) + 2n+\eta-1$$ with $\fK_4$ as in
Theorem~\ref{main4G} and $$A_3(2n+\eta) = 2(8n+4\eta-1)(2n+\eta-1).$$ \end{corollary}

\section{Questions and conjectures}\label{qu} 

We believe that the values we have found for the the functions $\etA$ and $\etB$ are very far from best
possible (except for $\etA$ in the case of quadrics).  We strongly expect that the best possible Stillman bounds
are far better than those we have found.   The conjectures below express these expectations.

\begin{conjecture}\label{K} For all characteristics  the best possible bound for $\fK_d(k)$ with $d$ fixed and $k$
varying is at worst $c_dn^{\kappa(d)}$ for some positive constant $c_d$ and  function $\kappa$ of $d$. 
 \end{conjecture}
 
 Note that for $d= 3$,  $c_3 = 2$ and $\kappa(3) = 1$.  Of course, our result for $\fK_4$ is much worse than
 polynomial, but we believe that there is a polynomial bound.
\begin{conjecture}\label{A} For all characteristics  the best possible bound for $\etA$ for a given value of $\eta$ for
$n$ polynomials of degree $d$, where $d$ is fixed and $n$ varies, is, at worst, $C_{d,\eta}n^{\lambda(d)}$ for some positive constant $C_{d,\eta}$ depending only on $d$ and $\eta$ and some function $\lambda$ of $d$. \end{conjecture}

This has been proved here for $d \leq 3$, where we may take $\lambda(d)  = d-1$.  However, we do not feel that there is
enough evidence yet to make a conjecture about what $\lambda(d)$ is in general.  Because of the formula in part (3) of Theorem~\ref{SJ}, 
or by Corollary~\ref{explA} we have
that $\etA_d(0,\, \ldots, 0,\, n) = \fK_d\bigl((2n+\eta)A_{d-1}(2n+\eta)\bigr) + 2n +\eta -1$.  This shows that
Conjecture~\ref{K} implies Conjecture~\ref{A}, and that one has, at worst, $\lambda(d) = 
\kappa(d)\bigl(\lambda(d-1) +1\bigr).$
\begin{conjecture}\label{B} For all characteristics, the bound
for the projective dimension of an ideal generated by $n$ forms of degree at most $d$, where $d$ is fixed
and $n$ varies, is, at worst,  $C_dn^d$ for some
positive constant $C_d$ depending only on $d$. \end{conjecture}

Note that while the conjecture for $\etA$ has been
verified here for $d \leq 3$, the conjecture for Stillman bounds is not known even in the case $d=2$. 

The following is raised as Question 6.2 in \cite{HMMS1}.  We conjecture this explicitly: of course, it implies 
Conjecture~\ref{B} for the case $d=2$.
\begin{conjecture}\label{HMMS}  If $R$ is a polynomial ring over a field and  $I$ an ideal of $R$ generated by $n$ quadrics and  of height $h$, then  $\pd_R(R/I) \leq  h(n - h + 1)$. \end{conjecture}

The ideal $I_{m,n,d}$ constructed in \cite{Mc}       
 has height $m$, is minimally generated by $m+n$ homogeneous polynomials of degree $d$, while its 
 projective dimension is $m + n {m+d-2 \choose d-1}$, which is a polynomial of total degree $d$ in 
 $m, n$.  These examples for $d=2$ show that one cannot improve the bound in Conjecture~\ref{HMMS},
 and that in a bound that is polynomial in $n$  for the projective dimension for fixed $d$ must have degree at least $d$.
 
The problem of giving explicit bounds for $\etA$ remains in characteristic 2, 3 even when $d=4$.
The corresponding problem for $\etA$ and $\etB$ if $d > 4$ is untouched.

Moreover, so far as we know, there is almost nothing known about lower bounds for $\etB$, even for
quadrics, except  the obvious fact that it must exceed the Stillman bound on projective dimension. 
In particular, so far as we know, it is possible that a polynomial bound for  $\etB$ exists in every
degree.  This problem is wide open, even for quadrics.

\begin{question}
We note that the results of \cite{AH2} show that everything about the primary decomposition of
an ideal generated by $n$ forms of specified degrees is bounded in terms of $n$ and the
degrees of the forms: this includes the number of primes, and  the numbers of  minimal generators and
their degrees for both the primes and primary ideals occurring.   We want to point out that it is largely an 
open question what can happen even for a regular sequence of $n$ quadrics.  Since the multiplicity of the 
quotient is  $2^n$,
the number of associated primes (which are the same as the minimal primes) is at most $2^n$.
This can happen, e.g., if the regular sequence is $x_1x_2, x_3x_4, \ldots, x_{2n-1}x_{2n}$.   We do
not know what might happen with the numbers of generators nor with their degrees. \end{question}

\quad\bigskip

$\begin{array}{ll}
\textrm{Altair Engineering}                             &\qquad\qquad \textrm{Department of Mathematics}\\
\textrm{1820 E.\ Big Beaver Rd.}                  &\qquad\qquad \textrm{University of Michigan}\\
\textrm{Troy, MI 48083}                               &\qquad\qquad \textrm{Ann Arbor, MI 48109--1043} \\
\textrm{USA}                                                &\qquad\qquad \textrm{USA}\\
\quad & \quad\\
\textrm{E-mail: antigran@gmail.com }          &\qquad\qquad \textrm{E-mail: hochster@umich.edu}\\ 
        
\end{array}$

\end{document}